%% file: high_frequency_Calderon_problem.tex
\documentclass[11pt,reqno,oneside]{amsart}
\usepackage[a4paper, total={6in, 9in}]{geometry}

\usepackage[usenames]{color}
\usepackage{tikz}
\usepackage{amsmath,pdfsync,verbatim,graphicx,epstopdf,enumerate}
\usepackage{todonotes}
\usepackage{amsmath,amscd, amssymb, amsthm, mathrsfs}
\usepackage[abbrev,nobysame,alphabetic]{amsrefs}
\usepackage[colorlinks=true]{hyperref}
\usepackage{cancel}
\usepackage{accents}
\usetikzlibrary{patterns}
\usetikzlibrary{decorations.pathreplacing,decorations.markings,arrows.meta,intersections,calc}

\hypersetup{allcolors=blue}
\usepackage{mathtools}  
\mathtoolsset{showonlyrefs}

\newcommand{\I}{\mathrm{i}}

\DeclareMathOperator{\tr}{\mathrm{tr}}

\newcommand{\Beq}{\begin{equation}}
	\newcommand{\Eeq}{\end{equation}}
\newcommand{\beq}{\begin{equation*}}
	\newcommand{\eeq}{\end{equation*}}
\newcommand{\bal}{\begin{align}}
	\newcommand{\eal}{\end{align}}

\usepackage{mathtools}

\newtheorem{theorem}{Theorem}[section]

\newtheorem{lemma}[theorem]{Lemma}
\newtheorem{proposition}[theorem]{Proposition}

\theoremstyle{definition}

\newtheorem{remark}[theorem]{Remark}

\newcommand{\R}{\mathbb{R}}

\newcommand{\norm}[1]{\lVert #1 \rVert}
\newcommand{\abs}[1]{\lvert #1 \rvert}

\newcommand{\supp}{\mathop{\rm supp}}
\newcommand{\dif}[1]{\,\mathrm{d}{#1}}

\newcommand{\nrm}[2][]{ \| {#2} \|_{#1}}

\allowdisplaybreaks

\newcommand{\vol}{\mathrm{vol}}

\newcommand{\comp}{\mathrm{comp}}

\newcommand{\mc}{\mathcal}

\newcommand{\Spec}{\mathrm{Spec}}

\renewcommand{\div}{\mathrm{div}\,}

\numberwithin{equation}{section}

\title[Anisotropic Calder\'on problem at high fixed frequency]{The anisotropic Calder\'on problem for Riemannian metrics at high fixed frequency}

\author{Mihajlo Ceki\'c$^{\mathsection}$}
\address{$^{\mathsection}$ CNRS, Université Paris-Est Créteil, Université Gustave Eiffel, LAMA UMR8050, F-94010 Créteil, France}
\email{mihajlo.cekic@cnrs.fr}
\author{Katya Krupchyk$^\ast$}
\address{$^\ast$ Department of Mathematics, University of California, Irvine, CA 92697-3875, USA}
\email{katya.krupchyk@uci.edu}
\author{Suman Kumar Sahoo $^\dagger$}
\address{ $^\dagger$ Department of Mathematics, Indian Institute of Technology Bombay.}
\email{suman@math.iitb.ac.in}

\author{Gunther Uhlmann $^\ddagger$}
\address{$^\ddagger$Department of Mathematics, University of Washington
}
\email{gunther@math.washington.edu}

\begin{document}


\begin{abstract}
We study an inverse boundary value problem for the Helmholtz equation on a smooth compact non-trapping Riemannian manifold with strictly convex boundary. We prove that, given two such metrics, for sufficiently large but fixed frequency $\lambda$, equality of their Dirichlet-to-Neumann maps implies equality of their lens data, up to a smooth boundary-fixing diffeomorphism. This establishes a high-frequency bridge between the Calder\'on problem and lens rigidity. We also prove a quantitative version of this result uniformly valid in a suitable bounded set of Riemannian metrics. Crucially, we show that Gaussian beam type solutions concentrating on maximal geodesics, split near the boundary into outgoing/incoming parts, such that the phase of the outgoing part contains information on exit point, direction, and travel time. The recovery of lens data then proceeds by a careful stationary phase analysis and a boundary integral identity.
	\end{abstract}

\maketitle

\section{Introduction and main results}
\subsection{Problem setting}

Let $(M,g)$ be a smooth compact connected Riemannian $n$-manifold with boundary, with $n \geq 2$, and let $\lambda \geq 0$ be a frequency. We consider the boundary value problem
\begin{equation} \label{eq:1-AC22}
	\left\{ \begin{aligned}
		(-\Delta_g - \lambda^2)u&=0 \quad \mbox{in}\quad M, \\
		u&=f \quad \mbox{on} \quad \partial M,
	\end{aligned} \right.
\end{equation}
where $\Delta_g$ is the Laplace-Beltrami operator on $M$.
In local coordinates $(x_1, \dotsc, x_n)$,
\[
\Delta_g u
= \sum_{j, k = 1}^n |g|^{-\frac{1}{2}}\partial_{x_j}(|g|^{\frac{1}{2}} g^{jk} \partial_{x_k} u),
\]
where $(g^{jk})=(g_{jk})^{-1}$ and $|g| = \det(g_{jk})$. Let $\mathrm{Spec}(-\Delta_g)$ denote the set of Dirichlet eigenvalues of $-\Delta_g$. Suppose $\lambda^2 \not \in \Spec(-\Delta_g)$, and let $ u\in C^{\infty}(M)$ be the unique solution of \eqref{eq:1-AC22} for a Dirichlet boundary condition $f \in C^{\infty}(\partial M)$.
The Dirichlet-to-Neumann map (DN map) associated to \eqref{eq:1-AC22} is given by
\begin{equation} \label{eq:DN-AC22}
	\Lambda^{\lambda}_g: C^{\infty}(\partial M) \to C^{\infty}(\partial M), \qquad \Lambda^{\lambda}_g f
	:= \partial_{\nu} u |_{\partial M},
\end{equation}
where $\nu$ is the outward unit normal. If $\psi: M \to M$ is a diffeomorphism equal to the identity on the boundary, it is straightforward to check that $\Lambda_g^\lambda = \Lambda_{\psi^*g}^\lambda$, where $\psi^*g$ denotes pullback by $\psi$ of $g$.

Assume further that $(M, g)$ has strictly convex boundary (i.e. the second fundamental form is positive definite on $\partial M$). Let $SM \subset TM$ be the unit tangent bundle of $(M,g)$ and define the incoming $(-)$ and outgoing $(+)$ boundary of $SM  $ as:
\[
\partial_\pm SM := \left\{ (x,v) \in SM~|~ x \in \partial M, \pm g_x(v,\nu(x)) > 0 \right\}.
\]
For any $(x,v) \in SM$, the maximally extended geodesic $\gamma_{(x,v)}$, with initial condition given by $\gamma_{(x,v)}(0)=x$, $\dot{\gamma}_{(x,v)}(0)=v$ is defined on a time interval $[0,\tau_g(x,v)]$ where $\tau_g(x,v)\in \R_{>0}\cup \{\infty\}$. We make the assumption that $(M, g)$ is \emph{non-trapping}, that is, that $\tau_g(x, v) < \infty$ for all $(x, v) \in SM$. Denote the \emph{travel time} function as $\ell_g := \tau_g|_{\partial_-SM}$. We define the \emph{scattering map}
\[S_g(x,v) := \big(\gamma_{(x, v)}(\ell_g(x,v)), \dot{\gamma}_{(x, v)}(\ell_g(x,v))\big),\quad S_g : \partial_- SM \longrightarrow \partial_+ SM.\]
The pair $(\ell_g, S_g)$ is referred to as the \emph{lens data} of the Riemannian manifold $(M,g)$.

The lens data encode the boundary information of the geodesic flow that can be observed from the outside. A fundamental inverse problem arising in integral geometry and tomography is whether one can determine the interior geometry --- namely, the Riemannian metric $g$ in $M$ --- from knowledge of the lens data $(\ell_g,S_g)$. In geophysical applications, this problem corresponds to recovering the speed of wave propagation inside a medium such as the Earth; see for instance \cite{PSU_book}.

The inverse problem we are interested in is to recover the lens data of $(M,g)$ from DN map $\Lambda_g^{\lambda}$ for a large but \emph{fixed} $\lambda$. Our first main result is as follows:

\begin{theorem}\label{thm_main_1}
    Suppose $(M, g_1)$ and $(M, g_2)$ are smooth connected oriented non-trapping  Riemannian $n$-manifolds with strictly convex boundaries, where $n \geq 2$. If $n = 2$, we assume additionally that $g_1 - g_2$ vanishes to infinite order at the boundary. Then, there exists a finite constant $\lambda_0 = \lambda_0(g_1, g_2) > 0$ such that the following holds. If there exists $\lambda > \lambda_0$ such that $\lambda^2 \not \in \Spec(-\Delta_{g_1}) \cup \Spec(-\Delta_{g_2})$ and $\Lambda_{g_1}^\lambda = \Lambda_{g_2}^\lambda$, then there exists a smooth diffeomorphism $\psi: M \to M$ such that $\psi|_{\partial M} = \mathrm{Id}$ and $\psi^*g_2 = g_1$ as bilinear forms on $TM|_{\partial M}$, and under the resulting identification of incoming and outgoing boundaries $\partial_\pm S_{g_1}M = \partial_\pm S_{\psi^*g_2}M =: \partial_\pm SM$, we have
    \begin{equation}\label{eq:lens-data-equality}
        (\ell_{g_1}, S_{g_1}) = (\ell_{\psi^*g_2}, S_{\psi^*g_2})\quad \mathrm{on}\quad \partial_-SM.
    \end{equation}
\end{theorem}
We note that the Calder\'on problem in dimension $n = 2$ has an additional conformal symmetry, which does not appear here since the operator in question is $-\Delta_g - \lambda^2$ and $\lambda$ is large. In fact, the case $n = 2$ in the more general setting where no equality of boundary jets is assumed follows directly from \cite[Theorem 1.1]{Carstean-Liimatainen-Tzou-24}, but we include this case for completeness. Using the known uniqueness results for lens rigidity, the equality \eqref{eq:lens-data-equality} further implies that $g_1$ and $\psi^*g_2$ are isometric under one of the following geometric assumptions:
\begin{itemize}
    \item when $n \geq 3$, under a strictly convex foliation condition \cite[Theorem 1.3]{Local_global_rigidity_SU};
    \item when $g_1$ and $\psi^*g_2$ are sufficiently close to each other (local uniqueness), under suitable assumptions, see \cites{Stefanov_Uhlmann_duke, Croke-Dairbekov-Sharafutdinov-00, Cekic-Guillarmou-Lefeuvre-24, Colin_JAMS, Stefanov-Uhlmann-09} and references therein;
    \item when $n = 2$, under certain curvature or simplicity assumptions \cites{Otal-90, Croke-90, Pestov-Uhlmann-05}.
\end{itemize}

\subsection{Quantification of uniqueness.} We now quantify the results of Theorem \ref{thm_main_1}. Let us first introduce some notation. We let $N \geq 2$ be an integer, $B > 1$, $\delta > 0$, $T > 0$, and we fix an arbitrary non-trapping Riemannian metric $g_0$ on the compact manifold $M$, with strictly convex boundary. Denote the set of smooth Riemannian metrics on $M$ by $\mathbf{Riem}(M)$, and set
    \begin{align}\label{eq:space-of-metrics}
        \mathcal{M}_{N, B, \delta, T}(g_0) &= \{g \in \mathbf{Riem}(M) \mid \|g - g_0\|_{C^N} \leq B,\quad g|_{M \setminus M_{\delta}} = g_0|_{M \setminus M_{\delta}},\,\,\\ 
        &\hspace{180pt}\tau_g \leq T ,\quad B^{-1}g_0 \leq g \leq B g_0\},
    \end{align}
    where $M_{\delta} = \{x \in M \mid d_{g_0}(x, \partial M) \geq \delta\}$. Here $\tau_g \leq T$ means that $\sup_{(x, v) \in S_gM} \tau_g(x, v) \leq T$. In particular, every metric in $\mathcal{M}_{N, B, \delta, T}(g_0)$ is non-trapping, and all maximal geodesics have length $\leq T$. Given $g_1, g_2 \in \mathcal{M}_{N, B, \delta, T}(g_0)$, there is a natural distance $d((\ell_{g_1}, S_{g_1}), (\ell_{g_2}, S_{g_2}))$ between the corresponding lens data, defined using the Sasaki metric of $g_0$, see \S \ref{ssec:distance-lens-data} below for a precise definition.
    
    We will denote by $H^s_{\lambda^{-1}}$ the semiclassical Sobolev space associated to a small parameter $\lambda^{-1} > 0$ and of index $s \in \mathbb{R}$, see \cite{Zworski} for more details. In particular, we have
    \[
        \|u\|_{H^1_{\lambda^{-1}}(\partial M)}^2 := \|u\|_{L^2(\partial M)}^2 + \lambda^{-2} \|du\|_{L^2(\partial M, T^*M)}^2,\quad u \in C^\infty(\partial M).
    \]
    
    \begin{theorem}\label{theorem:quantification}
        Let $(M, g_0)$ be a fixed background smooth connected oriented non-trapping Riemannian manifold with strictly convex boundary of dimension $n \geq 2$. There exists an integer $N = N(n) \geq 2$, such that for every $B > 1$, $T > 0$, and $\delta > 0$, there are constants $\lambda_0 > 0$, $\varepsilon_0 > 0$, and $C > 0$, such that the following holds. 
        
        If $\lambda > \lambda_0$ and $g_1, g_2 \in \mathcal{M}_{N, B, \delta, T}(g_0)$ satisfy
        \[
            \lambda^2 \not \in \Spec(-\Delta_{g_1}) \cup \Spec(-\Delta_{g_2}),
        \]
        then
        \[
            \lambda^{-1}\|\Lambda_{g_1}^\lambda - \Lambda_{g_2}^\lambda\|_{H^1_{\lambda^{-1}} \to L^2} \leq \varepsilon_0 \implies d((\ell_{g_1}, S_{g_1}), (\ell_{g_2}, S_{g_2})) \leq C\lambda^{-\frac{1}{2}}.
        \]
    \end{theorem}
    
    The assumption that $g_1$ and $g_2$ agree with a fixed metric $g_0$ near the boundary is taken for simplicity, and can probably be removed with some more work on quantifying the boundary determination results of \cite{LeeUhlmann1989}. We expect a classical H\"older stability statement to be true, in line with the work \cite{Ma_Sahoo_Salo} and the phenomenon of increasing stability for inverse problems at large frequency (whereas for the Calder\'on problem at `low frequencies' we expect only a logarithmic type of stability). The quantity $\lambda^{-1}\|\Lambda_{g_1}^\lambda - \Lambda_{g_2}^\lambda\|_{H^1_{\lambda^{-1}} \to L^2}$ appears naturally when studying a suitable boundary integral identity. Although $\Lambda_g^\lambda$ is a classical pseudodifferential operator of order one for each fixed admissible $\lambda$, the full family $(h\Lambda^{h^{-1}}_g)_h$ where $h = \lambda^{-1}$ need not belong to the standard semiclassical pseudodifferential calculus as the associated operator $-h^2\Delta_g - 1$ has a non-empty characteristic set. In contrast, $-h^2\Delta_g + 1$ is semiclassically elliptic, and its normalized Dirichlet-to-Neumann map is a semiclassical pseudodifferential operator of order one.
    
\subsection{State of the art and previous work.} Let $\Omega \subset \mathbb{R}^n$ be a bounded domain with smooth boundary and let $\gamma(x)$ denote a positive conductivity function. Given a boundary voltage $f$, the electrical potential $u$ inside $\Omega$ satisfies the elliptic equation
\begin{equation}
\begin{cases}
\nabla \cdot (\gamma \nabla u) = 0 & \text{in } \Omega, \\
u|_{\partial \Omega} = f .
\end{cases}
\end{equation}
The corresponding boundary current flux is
\[
\gamma \partial_\nu u \big|_{\partial \Omega}.
\]
This defines the \emph{Dirichlet-to-Neumann map} (DN map)
\[
\Lambda_\gamma : f \mapsto \gamma \partial_\nu u|_{\partial \Omega}.
\]
The Calder\'on problem asks: Does the Dirichlet-to-Neumann map $\Lambda_\gamma$ determine the conductivity $\gamma$ uniquely? It was shown in \cite{Calderon_Paper} that the linearisation of DN map at $\gamma \equiv 1$ is injective, using special complex exponential harmonic solutions. A major breakthrough was achieved by Sylvester and Uhlmann \cite{sylvester1987global}, who constructed \emph{complex geometrical optics (CGO)} solutions and used them to prove global uniqueness for smooth isotropic conductivities in dimensions $n \geq 3$. In two dimensions, the problem was solved by Nachman \cite{Nachman_annals} for $C^2$ conductivities, and later extended to merely $L^\infty$ conductivities by Astala and P\"aiv\"arinta \cite{Astala_Paivarint_Annals_Math} using quasiconformal techniques. For the anisotropic conductivity equation $-\div(\gamma \nabla u) = \mu u$ at a fixed $\mu \neq 0$, counterexamples to uniqueness with $C^k$ regularity were constructed in \cite{Daude-Helffer-Kamran-Nicoleau-24}, and upgraded to $C^\infty$ and Gevrey coefficients in \cite{Daude-Enciso-Helffer-Kamran-Nicoleau-24}; the latter work also treats the underlying operator $-\Delta_g + V$ with a prescribed non-constant potential $V$. We emphasise that these formulations \emph{differ} from the one considered in this article.

The Calder\'on problem is geometric in nature and is most naturally formulated as in \eqref{eq:1-AC22} for $\lambda = 0$ and $n \geq 3$ using Riemannian metrics; see also \cite{Uhlmann2014}. Again, when $\lambda = 0$ and $n \geq 3$, the problem remains largely open in the smooth category. Uniqueness is known under additional analyticity assumptions on $(M,g)$; see \cites{LeeUhlmann1989,LassasUhlmann2001}. In the smooth case, uniqueness in a fixed conformal class is known on conformally transversally anisotropic manifolds under suitable injectivity assumptions for the transversal geodesic ray transform; see \cite{Dos_Jems, DKSU}. We also refer to the very recent works \cites{Stefanov-26, Lin-26} establishing local rigidity of the Euclidean metric for the anisotropic Calderón problem on smooth bounded domains in $\mathbb{R}^n$ for $n \geq 3$. See also \cite{Carstea-Feizmohammadi-Oksanen-23} for further related results on the anisotropic Calderón problem. In two dimensions, the problem admits an additional conformal invariance and is unconditionally solved in \cite{LassasUhlmann2001}; see also \cites{Belishev,Sharafutdinov} for related results in dimension two. 

Recently, there has been a lot of progress on Calder\'on-type inverse problems at high fixed frequency, starting from \cite{Uhlmann-Wang-24}, where uniqueness was shown for compactly supported potentials with fixed Riemannian metrics of non-positive curvature, and \cite{Ma_Sahoo_Salo} where uniqueness was proven under the assumption that the $X$-ray transform is stably invertible and that $(M, g)$ is non-trapping. Later, the Calder\'on problem for connections at high frequency was considered in \cite{St-Amant-24} and uniqueness was shown under similar geometric assumptions as in the present paper implying that the parallel transport transform is injective; we note that in comparison, in Theorem \ref{thm_main_1} we are able to remove the technical assumption of no self-intersection of geodesics at the boundary. Analogous uniqueness results were obtained for certain non-linear versions of the Calder\'on problem, see \cites{St-Amant-24, KMSSS} and also references therein.

In dimensions $n \ge 3$, for the wave equation $\partial_t^2 - \Delta_g$, Belishev and Kurylev showed, using the boundary control method, that the dynamical DN map determines the Riemannian manifold up to an isometry equal to the identity on the boundary, see \cites{Belishev-Kurylev-91, Belishev-Kurylev-92, Belishev-97, Katchalov-Kurylev-Lassas-01}; see also \cites{hyperbolic_dn_SU,hyperbolic_dn_map_nonsimple_SU}. However, for the stationary Laplace--Beltrami operator, the corresponding problem remains open. 

\subsection{Proof ideas} The proof has two main ingredients: understanding the geometry of Gaussian beam solutions near the boundary, and an analysis of their interactions in a boundary integral identity. In the proof of Theorem \ref{thm_main_1}, we begin by placing the two metrics in a common extension where they agree in an exterior collar. This is achieved using the determination of the boundary jets of the metrics from the equality of the Dirichlet-to-Neumann maps in dimensions $n \geq 3$, or the assumed agreement of the boundary jets in dimension $n = 2$. In the proof of Theorem \ref{theorem:quantification}, agreement with a fixed background metric in a boundary collar is built in to the definition of the metric class.

Strict convexity and non-trapping then yield a uniform geometric description of the
geodesics near the boundary. Geodesics which are sufficiently close to tangential remain
entirely in the common collar. On the other hand, for every geodesic entering the inner region, its intersection with a sufficiently thin boundary collar consists of two short, uniformly transverse segments: an incoming segment adjacent to the entry point and an outgoing segment adjacent to the exit point (see Figure \ref{figure} below). Correspondingly, near the boundary Gaussian beams admit a scattering-type decomposition
\[
    u = u^- + u^+,
\]
where $u^-$ and $u^+$ are concentrated near the incoming and outgoing endpoints, respectively.
We refer to them as the incoming and outgoing parts of the beam. 

Let us consider an incoming vector $(x, v)$ generating two geodesics: one for $g_1$ and the other for $g_2$; write $u_1$ and $u_2$ for the corresponding Gaussian beams having the same initial data at $(x, v)$. Since the metrics agree in the common collar, the incoming parts $u_1^-$ and $u_2^-$ agree and hence cancel in the boundary integral identity of Proposition \ref{prop:extension-2}. The outgoing parts encode the lens data directly. More precisely, the centre of Gaussian beam $u_i^+$ is the exit point of the corresponding geodesic, the tangential gradient of the real part of its phase at the centre gives the tangential component of the exit direction, and the value of the phase at the centre is the travel time. The positive definiteness of the boundary Hessian of the imaginary part of the phase gives uniform Gaussian localization near the exit point. This structure is the basis of the two proofs of Theorems \ref{thm_main_1} and \ref{theorem:quantification} described below.

Let us mention that we use the construction of Gaussian beam quasimodes concentrated near non-tangential
geodesics, together with the uniform estimates needed in the argument, from \cites{Ma_Sahoo_Salo, St-Amant-24}. The quasimodes are corrected to exact solutions of the Helmholtz equation by means of the high-frequency solvability estimate for non-trapping manifolds from \cite{Ma_Sahoo_Salo}. 

In the first proof, given in Section \ref{section:three-solution} below and referred to as `three-solution proof', the two solutions $u_1$ and $u_2$ are tested against a third solution $u$ for $g_2$ going ‘backwards’, in the sense that the geodesic is taken with the opposite orientation and the transport data are prescribed at the exit point of the original $g_2$-geodesic. Correspondingly, the incoming and outgoing labels are interchanged, so that $u^\pm$ carries the same phase as $u_2^{\mp}$. We then study the interactions of the incoming and outgoing $(\pm)$ parts in the integral identity. The ‘backscattering term’ involves the interaction of $u^+$ with $u_i^+$, for $i =1, 2$, and is shown to be of lower order in Step 5 of the proof. Thus, the principal contribution comes from the interaction of $u_1^+ - u_2^+$ with $u^-$. A careful step-by-step stationary and nonstationary phase analysis computes the principal contributions of the relevant interactions and shows that the exit points, exit directions, and travel times agree; see Steps 6–8. In fact, in this geometric setting the scattering map determines the full lens data; see Proposition \ref{prop:known?} below. Nevertheless, we also recover the travel times directly from the integral identity, by a stationary phase argument, since this procedure may be of independent interest and potentially useful in other settings.

The proof of Theorem \ref{theorem:quantification} quantifies the steps of the preceding argument, treating the difference of the Dirichlet-to-Neumann maps as an error; see Section \ref{section:quantification-three-solutions}. The distance between the scattering maps is estimated similarly. There is, however, an important difference in the
recovery of the travel times. The mixed integral corresponding to the interaction of $u^-$ and
$u_1^+$ involves an oscillatory phase which need not have a real non-degenerate critical point. An almost analytic extension of this phase has a unique nearby non-degenerate complex
critical point, up to negligible errors arising from the choice of extension, and we establish a Melin–Sj\"ostrand-type stationary phase lemma to analyse its leading asymptotics, see Lemma \ref{lemma:stationary-phase} below.

The second proof, given in Section \ref{section:two-solutions} below and referred to as `two-solution proof', uses only the two exact solutions obtained from $u_1$ and $u_2$, together with a quadratic boundary identity for their difference. After the common incoming parts cancel, the leading identity compares a sum of positive diagonal Gaussian contributions with the mixed interaction of the outgoing beams. Direct Gaussian
localization and integration by parts recover the exit points and directions, with the natural
$\mc{O}(\lambda^{-\frac{1}{2}})$ accuracy in the quantitative setting. The travel times are then recovered as before using Proposition \ref{prop:known?}. In the quantitative case, one integrates this identity to obtain a control of the difference of travel times by the distance between scattering data. Together with the direct Gaussian lower bound for the diagonal term, this yields a two-solution proof which does not require stationary-phase asymptotics, nor does it require backscattering analysis.

\subsection{Acknowledgements} MC was supported by an Ambizione grant
(project number 201806) from the Swiss National Science Foundation during a part of the work on this project. SKS is supported by IIT Bombay seed grant (RD/0524-IRCCSH0-021) and ANRF Early Career Research Grant (ECRG) (RD/0125- ANRF000-016). MC and SKS thank ETH Z\"urich for providing working space, where part of this work was carried out. The research of K.K. is partially supported by the National Science Foundation (DMS 2408793). The work of G.U. is partially supported by the National Science Foundation. We warmly thank Yiran Wang for discussions and his contribution to the paper. We are also grateful to Colin Guillarmou and Gabriel P. Paternain for discussions about Proposition \ref{prop:known?}.

No AI/LLM tools were used, except in the final stage for proofreading.

\section{Preliminaries}

In this section, we collect several results that will be used later.

\subsection{Resolvent estimate} The following high frequency resolvent estimate was derived in \cite[Proposition 3.1]{Ma_Sahoo_Salo}.

\begin{proposition} \label{prop_resolvent_estimate}
    Let $(M,g)$ be a compact non-trapping manifold with smooth boundary.
    There are $C = C(M,g) > 0$ and $\lambda_0 = \lambda_0(M,g) > 0$ so that for any $\lambda \geq \lambda_0$ and any $f \in L^2(M)$, the equation 
    \[
    (-\Delta_g   - \lambda^2) u = f \text{ in $M$}
    \]
    has a solution $u \in H^2(M)$ with
    \[
    \lambda \norm{u}_{L^2(M)} + \norm{\nabla u}_{L^2(M)} + \lambda^{-1} \norm{\nabla^2 u}_{L^2(M)}
    \leq C\nrm[L^2(M)]{f}.
    \]
\end{proposition}

Certain resolvent estimates also hold for Riemannian manifolds admitting trapped trajectories at frequencies suitably far from Dirichlet eigenvalues, see \cite[Theorem 2.1]{KMSSS}.

\subsection{Stationary phase} Here we recall the stationary phase lemma, see \cite[Theorem 7.7.5]{hormander1} which we only state for the leading term asymptotics. In the following, we write $\Phi'$ and $\Phi''$ for the gradient and Hessian of a function $\Phi$, respectively.
\begin{lemma}\label{lem_stationary_phase}
    Let $K \subset \R^m$ be a compact set containing $0 \in \R^m$, and let $U$ be an open neighbourhood of $K$. Let $c > 0$ and $\Phi \in C^\infty(U)$ a phase function such that
    \begin{align*}
        \Im \Phi \geq 0,\quad \Im \Phi(0) = 0, \quad \Phi'(0) = 0,\quad |\det \Phi''(0)| \geq c,\quad  |\Phi'(x)| \geq c|x|\quad \mathrm{for}\quad x \in U.
    \end{align*}
    Moreover, assume that $\Phi$ belongs to a fixed bounded subset of $C^{3 \lceil{\frac{m}{2} + 1}\rceil + 1}(U)$. Then, there exists $C > 0$, such that for all $u \in C^\infty_{\comp}(U)$ with $\supp(u) \subset K$, we have
\begin{equation}\label{eq:stationary_phase}
	\left|\lambda^{\frac{m}{2}} \int_{U} e^{\I\lambda \Phi(x)} u(x)\, dx - \det\left(\frac{\Phi''(0)}{2\pi \I}\right)^{-\frac{1}{2}} e^{\I \lambda \Phi(0)} u(0)\right| \leq C\lambda^{-1} \sum_{\abs{\alpha} \leq m + 3} \sup_{x \in U} \abs{D^\alpha u(x)}.
\end{equation}
\end{lemma}

We note that this leading order version of the stationary phase lemma follows directly from the statement in \cite[Theorem 7.7.5]{hormander1} for $k = \lceil{\frac{m}{2} + 1}\rceil$, after absorbing the higher order expansion terms inside the error; this is also the place where the lower bound on the determinant at zero is used. Also, note that the square root of the determinant is defined by selecting the usual holomorphic branch of square root, see the discussion in \cite[page 85]{hormander1}.

\subsection{Integral identity} We record an integral identity on the boundary, based on the self-adjoint property of the DN map. It essentially relates two DN maps on extensions of Riemannian manifolds. 

In what follows, we use the convention that the inner products are linear in the first argument and conjugate-linear in the second argument. We write $d\vol_g$ for the Riemannian volume measure of a Riemannian manifold $(M, g)$; on the boundary $\partial M$ we also use the natural volume measure. For simplicity, we will not make the restrictions to boundary explicit. By $M^\circ$ we denote the interior of $M$. We recall that for $\lambda \in \R$ such that $\lambda^2 \not \in \Spec(-\Delta_g)$, the DN map commutes with complex conjugation and is symmetric
\[
    \Lambda_g^\lambda \overline{f} = \overline{\Lambda_g^\lambda f},\quad \langle{\Lambda_g^\lambda f, h}\rangle_{L^2(\partial M)} = \langle{f, \Lambda_g^\lambda h}\rangle_{L^2(\partial M)},\quad f, h \in C^\infty(\partial M).
\]

\begin{proposition}\label{prop:extension-2}
    Let $M$ and $\widetilde{M}$ be smooth compact manifolds with boundary such that $M \subset \widetilde{M}$. For $i = 1, 2$, let $g_i$ and $\widetilde{g}_i$ be Riemannian metrics on $M$ and $\widetilde{M}$, respectively, such that $\widetilde{g}_i|_{M} = g_i$ and $\widetilde{g}_1 = \widetilde{g}_2$ in $\widetilde{M} \setminus M^\circ$. Assume that $\lambda^2 \not \in \Spec(-\Delta_{g_1}) \cup \Spec(-\Delta_{g_2})$. Then for all $u, u_1, u_2 \in C^\infty(\widetilde{M})$ such that
    \[
        (-\Delta_{\widetilde{g}_i} - \lambda^2)u_i = 0,\quad  (-\Delta_{\widetilde{g}_2} - \lambda^2) u = 0,
    \]
    we have
    \begin{itemize}
        \item[1.] $\langle{\partial_\nu(u_1 - u_2), u}\rangle_{L^2(\partial \widetilde{M})} - \langle{u_1 - u_2, \partial_\nu u}\rangle_{L^2(\partial \widetilde{M})} = \langle{(\Lambda_{g_1}^\lambda - \Lambda_{g_2}^\lambda)u_1, u}\rangle_{L^2(\partial M)}$,
        \item[2.] $\Im \langle{\partial_\nu(u_1 - u_2), u_1 - u_2}\rangle_{L^2(\partial \widetilde{M})} = \Im \langle{(\Lambda_{g_1}^\lambda - \Lambda_{g_2}^\lambda)u_2, u_1}\rangle_{L^2(\partial M)}$.
    \end{itemize}
\end{proposition}
For $i = 1, 2$, note that $\Spec(-\Delta_{g_i})$ denotes the Dirichlet spectrum of $-\Delta_{g_i}$ on $M$. Note also that by assumption the outward normals of $\widetilde{g}_1$ and $\widetilde{g}_2$ agree on $\partial \widetilde{M}$, denoted by $\nu$. We remark that the identities hold for $H^2$ solutions by density, and that for admissible $\lambda$, the DN map extends continuously as a first order pseudodifferential operator to
\[
    \Lambda_g^\lambda: H^{\frac{3}{2}}(\partial M) \to H^{\frac{1}{2}}(\partial M),\quad \Lambda_g^\lambda: H^{1}(\partial M) \to L^2(\partial M),
\]
where $H^{\frac{3}{2}}(\partial M)$ appears as the trace space of $H^2(M)$ solutions. However, it is also possible to work with smooth solutions in the remainder of the paper, since Proposition \ref{prop_resolvent_estimate} can be arranged to produce smooth correction terms by considering an extension of the manifold. 
\begin{proof}
    Firstly, we simply observe that
    \[
        \langle{\partial_\nu u_2, u}\rangle_{L^2(\partial \widetilde{M})} = \int_{\widetilde{M}} (\langle{du_2, du}\rangle_{\widetilde{g}_2} - \lambda^2 u_2 \overline{u})\, d\vol_{\widetilde{g}_2} = \langle{u_2, \partial_\nu u}\rangle_{L^2(\partial \widetilde{M})},
    \]
    where we used Stokes' theorem throughout (or equivalently, the quadratic form associated to the DN map). For the part of the integral identity involving $u_1$, we make the following computation 
    \begin{align*}
        &\langle{\partial_\nu u_1, u}\rangle_{L^2(\partial \widetilde{M})} = \int_{\widetilde{M}} (\langle{du_1, du}\rangle_{\widetilde{g}_1} - \lambda^2 u_1 \overline{u})\,d\vol_{\widetilde{g}_1}\\
        &= \int_{\widetilde{M} \setminus M} (\langle{du_1, du}\rangle_{\widetilde{g}_1} - \lambda^2 u_1 \overline{u})\, d\vol_{\widetilde{g}_1} + \langle{\Lambda_{g_1}^\lambda u_1, u}\rangle_{L^2(\partial M)}\\
        &= \int_{\widetilde{M} \setminus M} (\langle{du_1, du}\rangle_{\widetilde{g}_1} - \lambda^2 u_1 \overline{u})\, d\vol_{\widetilde{g}_1} + \langle{(\Lambda_{g_1}^\lambda - \Lambda_{g_2}^\lambda)u_1,  u}\rangle_{L^2(\partial M)} + \langle{u_1, \Lambda_{g_2}^\lambda  u}\rangle_{L^2(\partial M)}\\
        &= \int_{\widetilde{M} \setminus M} (\langle{du_1, du}\rangle_{\widetilde{g}_1} - \lambda^2 u_1 \overline{u})\, d\vol_{\widetilde{g}_2} + \int_{M} (\langle{du_1, du}\rangle_{g_2} - \lambda^2 u_1 \overline{u})\, d\vol_{g_2}\\
        &\hspace{8cm}+ \langle{(\Lambda_{g_1}^\lambda - \Lambda_{g_2}^\lambda)u_1,  u}\rangle_{L^2(\partial M)}\\
        &= \langle{u_1, \partial_\nu u}\rangle_{L^2(\partial \widetilde{M})} + \langle{(\Lambda_{g_1}^\lambda - \Lambda_{g_2}^\lambda)u_1,  u}\rangle_{L^2(\partial M)},
    \end{align*}
    where in the first two lines we again used Stokes' theorem (or the quadratic form associated to the DN map), in the third line we used that the DN map is (formally) self-adjoint, in the fourth line we used that $\widetilde{g}_1 = \widetilde{g}_2$ in $\widetilde{M} \setminus M$ as well as Stokes' theorem, which we also used in the last line. Subtracting the previous identity from this one yields the first integral identity.

    By the first part, by replacing $u$ with $u_2$ we get
    \[
        \langle{\partial_\nu(u_1 - u_2), u_2}\rangle_{L^2(\partial \widetilde{M})} - \langle{u_1 - u_2, \partial_\nu u_2}\rangle_{L^2(\partial \widetilde{M})} = \langle{(\Lambda_{g_1}^\lambda - \Lambda_{g_2}^\lambda)u_1, u_2}\rangle_{L^2(\partial M)},
    \]
    and by swapping $u_1$ and $u_2$ (as well as $g_1$ with $g_2$) in the preceding identity we get
    \[
        \langle{\partial_\nu(u_1 - u_2), u_1}\rangle_{L^2(\partial \widetilde{M})} - \langle{u_1 - u_2, \partial_\nu u_1}\rangle_{L^2(\partial \widetilde{M})} = \langle{(\Lambda_{g_1}^\lambda - \Lambda_{g_2}^\lambda)u_2, u_1}\rangle_{L^2(\partial M)}.
    \]
    The second identity is obtained by subtracting the first identity from the second one, and the fact that the DN map is self-adjoint.
\end{proof}

\begin{remark}
    A direct alternative proof of Proposition \ref{prop:extension-2}, Item 2, is to apply Green's second identity to $u = u_1 - u_2$ satisfying $(-\Delta_{\widetilde{g}_1} - \lambda^2) u = 0$ in the exterior domain $\widetilde{M} \setminus M$.
\end{remark}

\subsection{A geometric identity}

    Here we give a proof of a standard geometric identity for which we could not locate a proof in the literature. Given a Riemannian manifold $(M, g)$, write 
    \[
        \alpha(x, v)(\xi) := g_x(d\pi (\xi), v)
    \]
    for the Liouville contact form, and $(\varphi_t)_{t}$ for the geodesic flow. In other words, for $(x, v) \in SM$, we have
    \[
        \varphi_t(x, v) = (\gamma_{x, v}(t), \dot{\gamma}_{x, v}(t))
    \]
    for all $t$ for which the geodesic is defined. Denote by $X(x, v) = \partial_t|_{t=0} \varphi_t(x, v)$ the geodesic vector field. We recall that $\pi: SM \to M$ is the footpoint projection. 

\begin{proposition}\label{prop:identity}
    Assume $(M, g)$ has strictly convex boundary and is non-trapping. Then
    \[
        d\ell_g = S_g^*\alpha - \alpha.
    \]
\end{proposition}
\begin{proof}
    Observe that by definition $\varphi_{\ell_g(x, v)}(x, v) = S_g(x, v)$ for $(x, v) \in \partial_-SM$. Note that $S_g$ is differentiable in $(x, v)$ since $(M, g)$ is assumed to be non-trapped and since it has strictly convex boundary (see \cite[Lemma 3.3.5]{PSU_book}). Differentiating in $(x, v)$, using the chain rule, and that $X$ generates the geodesic flow, we get for $\xi \in T_{(x, v)}\partial_-SM$ that
    \[
        dS_g(\xi) = X(S_g(x, v)) d\ell_g(\xi) + d\varphi_{\ell_g(x,v)}(\xi).
    \]
    Applying $d\pi$, and using that $d\pi(X(x, v)) = v$, we get
    \[
        d\pi dS_g(\xi) = S_g(x, v) d\ell_g(\xi) + d\pi d\varphi_{\ell_g(x, v)}(\xi)
    \]
    Taking inner product with $S_g(x, v)$ and using that it has norm one, we get
    \[
        S_g^*\alpha(\xi) = d\ell_g(\xi) + g_{\pi(S_g(x, v))}(d\pi d\varphi_{\ell_g(x, v)}(\xi), S_g(x, v)).
    \]
    We are left to show that the last term equals $\alpha(x, v)(\xi)$. Indeed, it equals $\varphi_{\ell_g(x, v)}^*\alpha(x, v)(\xi)$, and $\alpha$ is known to be invariant under the geodesic flow, see \cite[Lemma 1.30]{Paternain-99}. This completes the proof.
\end{proof}

This determines the travel times from the scattering map up to a constant. In favourable geometric situations, we may exactly recover the travel times.

\begin{proposition}\label{prop:known?}
    Assume that $(M, g_1)$ and $(M, g_2)$ are non-trapping with strictly convex boundary, such that $g_1$ and $g_2$ agree on $TM|_{\partial M}$. Then 
    \[
        S_{g_1} = S_{g_2} \implies \ell_{g_1} = \ell_{g_2}.
    \]
\end{proposition}
\begin{proof}
    By assumption that $g_1$ and $g_2$ agree on $TM|_{\partial M}$, we have that the respective contact $1$-forms agree, i.e. $\alpha_{g_1} = \alpha_{g_2}$ when restricted to $SM|_{\partial M}$. By Proposition \ref{prop:identity} and the assumption on equality of scattering maps, we therefore conclude that $d(\ell_{g_1} - \ell_{g_2}) \equiv 0$. Thus $\ell_{g_1} - \ell_{g_2}$ is constant on connected components of $\partial_-SM$ (note that the metrics are non-trapping), and these connected components precisely correspond to connected components of $\partial M$. However, note that if $(v_i)_{i\geq 1} \in \partial_-SM$ is a sequence of incoming vectors such that $d(v_i, S\partial M) \to 0$ as $i \to \infty$ (i.e. the sequence becomes nearly tangent to the boundary), we have 
    \[
        \ell_{g_1}(v_i) \to_{i \to \infty} 0,\quad \ell_{g_2}(v_i) \to_{i \to \infty} 0,
    \]
    since by strict convexity of the boundary the corresponding geodesics become short. This proves that the mentioned constants are zero, and so that $\ell_{g_1} \equiv \ell_{g_2}$.
\end{proof}

\subsection{A natural distance between lens data.}\label{ssec:distance-lens-data} Assuming the notation of Theorem \ref{theorem:quantification}, the distance between lens data can be defined explicitly as 
\[
    d((\ell_{g_1}, S_{g_1}), (\ell_{g_2}, S_{g_2})) = \max_{(x, v) \in \partial_- SM}\left( |\ell_{g_1}(x, v) - \ell_{g_2}(x, v)| + d_{g_{\mathrm{Sas}}}(S_{g_1}(x, v), S_{g_2}(x, v))\right),
\]
where $g_{\mathrm{Sas}}$ denote the Sasaki metric on $SM|_{\partial M}$; see \cite{Paternain-99} for more details on the Sasaki metric.

However in practise we will use the following distance $d'$ defined using parallel transport, which is equivalent to $d$ near the diagonal and for incoming vectors that build a uniformly positive angle with the boundary. If $r > 0$ is the injectivity radius of $(\partial M, g_0)$, then
    \[
        d'((\ell_{g_1}, S_{g_1}), (\ell_{g_2}, S_{g_2})) = \max_{(x, v) \in \partial_- SM}\left( |\ell_{g_1}(x, v) - \ell_{g_2}(x, v)| + d'(S_{g_1}(x, v), S_{g_2}(x, v))\right).
    \]
    Note that the maximum is indeed a maximum since for vectors close to being tangent to $\partial M$, by strict convexity and by the equality $g_1 = g_2$ near the boundary, $\ell_{g_1} = \ell_{g_2}$ and $S_{g_1} = S_{g_2}$ there. Also, here we set for $(y, u), (z, w) \in \partial_+SM$ that
    \[
        d'((y, u), (z, w)) = \begin{cases}
            1,\quad \mathrm{if}\quad  d_{\partial M}(y, z) \geq r,\\
            \min(1, |\mc{P}_{z \to y} w - u| + d_{\partial M}(y, z)),\quad \mathrm{otherwise},
        \end{cases} 
    \]
    where $d_{\partial M}(\bullet, \bullet)$ denotes the boundary distance, and $\mc{P}_{z \to y}$ denotes the parallel transport with respect to Levi-Civita connection of $g_0$ along the shortest boundary geodesic from $z$ to $y$.

\section{Gaussian beams with uniform estimates}

In this section we recall the result about uniform Gaussian beam solutions, \cite[Theorem 6.2]{Ma_Sahoo_Salo}, and derive further consequences. Notably, for manifolds with strictly convex boundary, we observe that the solutions split into incoming $(-)$ and outgoing $(+)$ solutions as in scattering theory. We recall that we always have $h = \lambda^{-1}$.

\begin{theorem} \label{thm:gbu-AC22}
	Let $(M,g)$ be a compact oriented manifold with smooth boundary. Fix constants $T > 0$, $p \in [1,\infty)$, $k \geq 0$, and $K \geq 0$. There is a constant $C = C(M,g,T,p,k,K) > 0$ such that for any $(x,v) \in \partial_-SM$ with $\ell_g(x,v)\leq T$ and $h \in (0,1)$, there is $u = u_{x,v,h} \in C^{\infty}(M)$ associated with $\gamma = \gamma_{x,v}$ and satisfying 
	\begin{gather*}
		\norm{u}_{L^p(M)} \leq C, \\
		\norm{(-h^2 \Delta_g - 1)u}_{W^{k,p}(M)} \leq C h^K, \\
		\mathrm{supp}(u) \subset \{ y \in M \mid \mathrm{dist}_g \big( y, \gamma([0,\ell_g(x,v)]) \big) \leq C^{-1} \},
	\end{gather*}
	uniformly over all $0 < h < 1$.
	
	There is also a symmetric complex $(1,1)$-tensor $H(t) = H_{x,v}(t)$ on $T_{\gamma(t)} M$, depending smoothly on $t \in [0,\ell_g(x,v)]$ and satisfying 
	\[
	\mathrm{Im}(H(t)^{\flat}) \geq 0, \qquad \mathrm{Im}(H(t)^{\flat})|_{\dot{\gamma}(t)^{\perp}} \geq C^{-1} g|_{\dot{\gamma}(t)^{\perp}},
	\]
	such that $u = u_{x,v,h}$ has the following form. If $x_0 \in \gamma([0,\ell_g(x,v)])$ and if $t_1 < \dotsb < t_{N_p}$ are the times in $[0,\ell_g(x,v)]$ when $\gamma(t_l) = x_0$, then in a small neighbourhood $U$ of $x_0$ we have 
	\[
	u|_U = u^{(1)} + \dotsb + u^{(N_p)}
	\]
	where each $u^{(l)}$ satisfies 
	\[
	u^{(l)}|_U = h^{-\frac{n-1}{2p}} e^{\I \Phi^{(l)}/h} (a_0^{(l)} + h a_1^{(l)} + \dotsb + h^N a_N^{(l)}) \rho.
	\]
	Here $N = N(M,g,T,p,k,K)$, and $\rho$ is a smooth cutoff function supported near $\gamma|_{[t_l-\varepsilon,t_l+\varepsilon]}$. The phase $\Phi = \Phi^{(l)}$ is independent of $h$ and satisfies for $t$ near $t_l$ 
	\[
	\Phi(\gamma(t)) = t, \quad \nabla \Phi(\gamma(t)) = \dot{\gamma}(t), \quad \nabla^2 \Phi(\gamma(t)) = H(t)^{\flat}, \qquad \norm{\Phi}_{C^k(\overline{U})} \leq C.
	\]
	The amplitudes $a_j^{(l)}$ are independent of $h$, and for $t$ near $t_l$ one has 
	\[
	a_0^{(l)}(\gamma(t)) = \exp \left[ -\frac{1}{2} \int_0^t \mathrm{tr}_g(H(s)) \dif s \right], \qquad \norm{a_j^{(l)}}_{C^k(\overline{U})} \leq C.
	\]
\end{theorem}

    We assume now that $(M, g)$ has strictly convex boundary; by definition, this means that the second fundamental form 
    \[
        \mathrm{II}(x)(u, v) := \langle{\nabla_u \partial_\nu, v}\rangle,\quad u, v \in T_x \partial M,
    \]
    is positive definite. Equivalently, if $\rho$ denotes the boundary distance function, we have $\mathrm{II} = -\nabla^2 \rho|_{T\partial M}$, where $\nabla^2$ is the Hessian defined by the Riemannian metric. For $\delta > 0$, write $M_\delta = \{x \in M \mid d(x, \partial M) \geq \delta\}$. By continuity, there exists $\delta_0 > 0$ such that $\partial M_\delta$ is strictly convex for $\delta < \delta_0$. If $\gamma$ is a geodesic starting at $\partial M$ and going inside $M$, we have
    \begin{equation}\label{eq:strictly-geodesic-convex}
        \partial_t^2 \rho(\gamma(t)) = \langle{\nabla_{\dot{\gamma}} \nabla \rho, \dot{\gamma}(t)}\rangle = \nabla^2 \rho(\dot{\gamma}(t), \dot{\gamma}(t)) = \nabla^2 \rho(\Pi_{(\nabla \rho)^\perp} \dot{\gamma}(t), \Pi_{(\nabla \rho)^\perp} \dot{\gamma}(t)) \leq 0,
    \end{equation}
    for all $t$ such that $\rho(\gamma(t)) \leq \delta_0$, where in the first equality we used that $\gamma$ is a geodesic (so $\nabla_{\dot{\gamma}} \dot{\gamma} = 0$), and in the last equality we wrote $\Pi_{(\nabla \rho)^\perp}$ for the orthogonal projector to $(\nabla \rho)^\perp$. We also used that 
    \[
        \nabla^2 \rho(\nabla \rho, v) = \langle{\nabla_{\nabla \rho} \nabla \rho, v}\rangle = 0,\quad v \in TM,
    \]
    because $\nabla \rho$ integrates to the geodesics orthogonal to the boundary, and in the last inequality that the second fundamental form of $\partial M_{\rho(\gamma(t))}$ is positive definite. As a consequence, any such geodesic $\gamma$ can have at most two possible intersections with $(\partial M_\delta)_{\delta < \delta_0}$.

    Recall that we say that $\gamma$ is \emph{maximal non-tangential} precisely if $\gamma = \gamma_{x, v}$ for some $(x, v) \in \partial_-SM$ (i.e. it cannot be extended further and intersects the boundary transversely).
    
\begin{proposition}\label{prop:boundary-estimate+geometry}
    Assume that $(M, g)$ has strictly convex boundary and is non-trapping, and let $\delta_0 > 0$ be such that $(\partial M_\eta)_{\eta \in [0, \delta_0]}$ are strictly convex. There exist $\delta \in (0, \delta_0)$ and $\varepsilon > 0$ such that the following holds. Let
    \[
        \mc{G}_{\varepsilon} = \{(x, v) \in \partial_-SM \mid -g_x(v, \nu(x)) > \varepsilon\}.
    \]
    Then we have
    \begin{itemize}
        \item Any maximal non-tangential geodesic $\gamma$ in $M_{\delta_0}$ is extended to a maximal non-tangential geodesic $\widetilde{\gamma}$ in $M$, which is the concatenation of $\gamma_1$, $\gamma$, and $\gamma_2$, so that $\gamma_1, \gamma_2 \subset M \setminus M_{\delta_0}$, and $\widetilde{\gamma} = \gamma_{x, v}$ for some $(x, v) \in \mc{G}_\varepsilon$.

        \item For all $\eta \in [0, \delta]$ and for each $(x, v) \in \mc{G}_{\varepsilon}$, we have that $\gamma_{x, v}$ intersects $\partial M_{\eta}$ exactly twice.

        \item There exists $\varepsilon_+ = \varepsilon_+(\varepsilon) > 0$ such that, for all $(x, v) \in \mc{G}_\varepsilon$, writing $S_g(x, v) = (x_+, v_+)$, we have
        \[
            g_{x_+} (v_+, \nu(x_+)) > \varepsilon_+.
        \]
    \end{itemize}
    
    Up to a further reduction of $\delta > 0$, the following holds. If $u$ is the Gaussian beam constructed in Theorem \ref{thm:gbu-AC22} along $\gamma \in \mc{G}_\varepsilon$, then
    \[
        u|_{M \setminus M_{\delta/2}} = u_- + u_+,
    \]
    where
    \[
        u^\pm = \lambda^{\frac{n-1}{2p}} e^{\I \lambda \Phi_\pm} (a_0^{\pm} + \lambda^{-1} a_1^{\pm} + \dotsb + \lambda^{-N} a_N^{\pm}) \chi_\pm,
    \]
    and the uniform estimates hold for $a^\pm, \chi_\pm, \Phi_\pm$. Also,
\begin{enumerate}
    \item[1.] $\|u\|_{L^p(\partial M)} = \mc{O}(1)$.
    \item[2.] $\|\partial_\nu u\|_{L^p(\partial M)} = \mc{O}(\lambda)$. Also, if 
    \[
        \partial_\nu u|_{\partial M} = \lambda^{\frac{n-1}{2p}}(e^{\I \lambda \Phi_+} \I\lambda \partial_\nu \Phi_+ a_0^+ \chi_+ + e^{\I \lambda \Phi_-} \I\lambda \partial_\nu \Phi_- a_0^- \chi_-) + R,
    \]
    then $\|R\|_{L^p(\partial M)} = \mc{O}(1)$.
\end{enumerate}
\end{proposition}

\begin{figure}[htbp]
    \centering
    \def\svgwidth{0.8\textwidth} 
    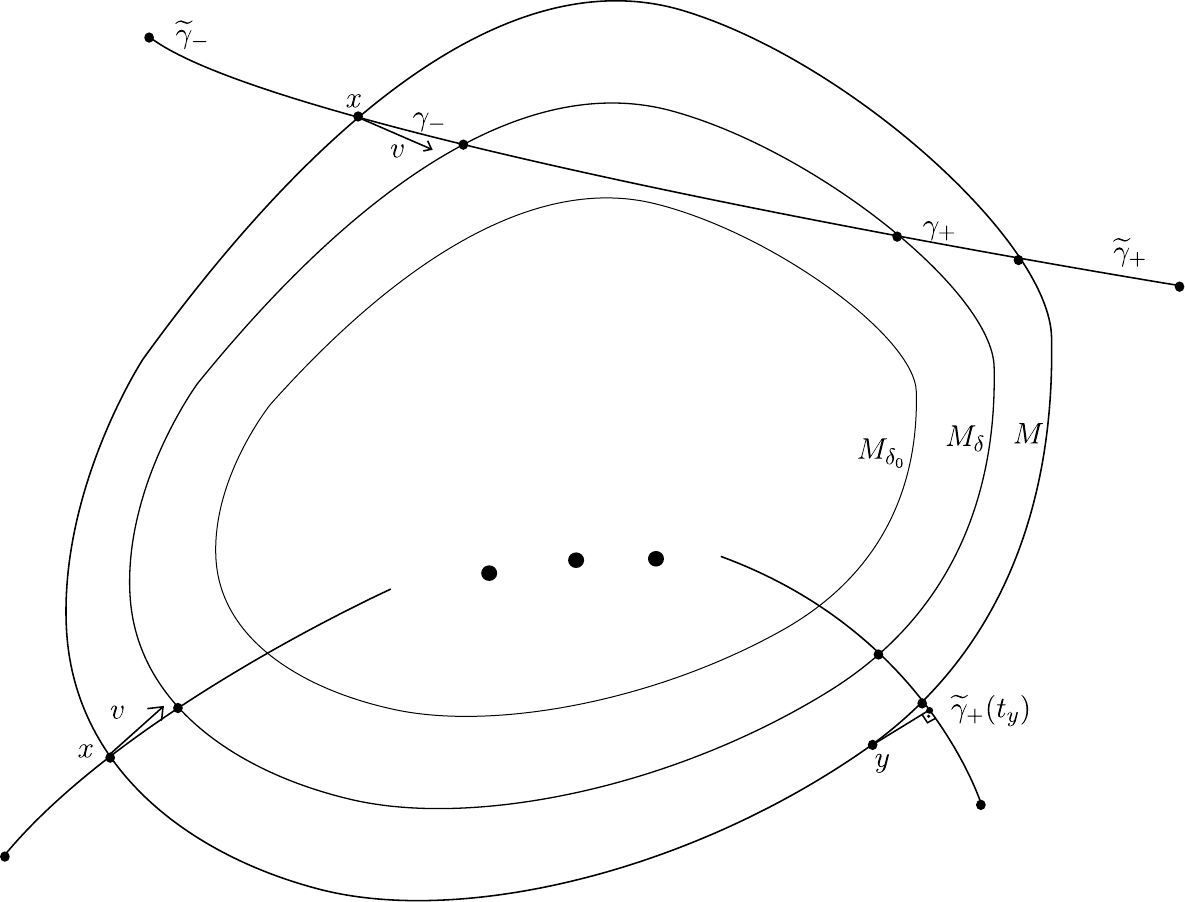
    \caption{A visualisation of the geometric construction of Proposition \ref{prop:boundary-estimate+geometry}. The picture emphasises that there are two kinds of geodesics: short geodesics (the upper one) and geodesics visiting the interior of $M_{\delta_0}$ (the lower one). For simplicity, we write $(x, v)$ for the incoming data of both of them, as well as $\widetilde{\gamma}_{\pm}$ for the exterior segments.}
    \label{figure}
\end{figure}

We note that in theory $\gamma$ could self-intersect near or on the boundary $\partial M$. This is not an issue here because $\gamma \cap M \setminus M_\delta = \gamma_- \cup \gamma_+$ splits as a union of two geodesic segments, each of which without self-intersections. Also, note that the maximal non-oriented geodesic $\widetilde{\gamma}$ has two intersections with $\partial M$ and is thus generated by both $(x, v)$ and by $-S_g(x, v)$; we must have $(x, v), -S_g(x, v) \in \mc{G}_\varepsilon$ as otherwise $\widetilde{\gamma}$ would not have any intersection with $\partial M_{\delta_0}$ by construction.

\begin{proof}
    Let
    \[
        I := \min \{-g_x(v, \nu(x)) \mid (x, v) \in \partial_-SM\,\,\mathrm{s.t.}\,\, \gamma_{x, v}\quad \mathrm{intersects}\quad \partial M_{\delta_0}\}.
    \]
    Note that $I > 0$ since for $v$ forming a small angle with $\partial M$, $\gamma_{x, v}$ is contained in the collar $M \setminus M_{\delta_0}$ by strict convexity. Choose $\varepsilon > 0$ such that $I > \varepsilon$. Note that by construction, if $(x, v) \not \in \mc{G}_\varepsilon$, then $\gamma_{x, v}$ is contained in the collar $M \setminus M_{\delta_0}$. Note also that any maximal non-tangential geodesic $\gamma$ in $M_{\delta_0}$ is extended to a maximal non-tangential geodesic $\widetilde{\gamma}$ in $M$, thanks to the strict convexity assumption, see \eqref{eq:strictly-geodesic-convex} above. The decomposition of $\widetilde{\gamma}$ in the collar also follows, as well that $\widetilde{\gamma}$ is generated by some $(x, v) \in \mc{G}_\varepsilon$, as required.

    Let 
    \[
        \delta_1 := \inf_{(x, v) \in \mc{G}_{\varepsilon}} \sup_{t \in [0, \ell_{\gamma}]} d(\gamma_{x, v}(t), \partial M).
    \]
    Note that $\delta_1 > 0$ by strict convexity of the boundary and the assumption on the angle of $v$ with $\partial M$. Let  $\delta \in (0, \min(\delta_0, \delta_1))$. By construction and by strict convexity (see \eqref{eq:strictly-geodesic-convex}), it follows that for any $(x, v) \in \mc{G}_\varepsilon$, and any $\eta \in [0, \delta]$, $\gamma_{x, v}$ intersects $\partial M_\eta$ exactly twice, proving the second bullet-point (see Figure \ref{figure}). For the third and final bullet-point,  simply notice that by strict convexity, if the angle a geodesic forms with boundary is sufficiently small, the geodesic stays in an arbitrarily small neighbourhood of the boundary (see \eqref{eq:strictly-geodesic-convex}) contradicting the second bullet-point. Thus there is a positive lower bound for the exit angle. 
    
    We now move to the construction of $u_\pm$. Assume $C > 0$ in Theorem \ref{thm:gbu-AC22} is large enough that $C^{-1} < \frac{\delta}{2}$, and let $\gamma = \gamma_{x_0, v_0}$ for some $(x_0, v_0) \in \mc{G}_\varepsilon$. Denote the two near-boundary geodesic segments of $\gamma$ by $\gamma_\pm$ (see Figure \ref{figure}), i.e.
    \[
        \gamma_+ \cup \gamma_- = \gamma \cap (M \setminus M_{\delta}).
    \]
    Embed $(M, g)$ into a closed Riemannian manifold, and extend $\gamma_{\pm}$ to the exterior of $M$ up to arc-length $\delta_2 > 0$, keeping them embedded; write $\widetilde{\gamma}_\pm$ for the obtained geodesic segments (see Figure \ref{figure}). Note that $\delta > 0$ can be chosen small enough so that the $\delta$-tubular neighbourhood of $\widetilde{\gamma}_-$ is embedded. Assume $x \in M \setminus M_{\delta/2}$ is in $\delta/2$-tubular neighbourhood of $\widetilde{\gamma}_-$, with the closest point $y \in \widetilde{\gamma}_-$ to $x$. Then by triangle inequality $d(y, \partial M) < \delta$ and so $y$ is outside $M_\delta$. The preceding discussion determines the $u_-$ part of the Gaussian beam in $M \setminus M_{\delta/2}$, and so the required expansion holds by inspecting the construction of Gaussian beams (see \cite[Proof of Theorem 6.2]{Ma_Sahoo_Salo} or \cite{St-Amant-24}). An analogous discussion is valid for $u_+$.

    Let us now verify the uniform boundary estimates. By the triangle inequality, it suffices to show the estimate for each $u_\pm$; let us show it for $u_-$, while the case of $u_+$ follows in the same way. In turn by triangle inequality it suffices to show that
    \begin{equation}\label{eq:GB-boundary-estimate}
        \lambda^{\frac{n-1}{2}} \int_{\partial M} |\chi_-|^p e^{-p\lambda \Im \Phi_-} |a_i^-|^p\, d\vol_g = \mc{O}(1),\quad i = 0, 1, \dotsc, N.
    \end{equation}
    We first claim that $\supp(\chi_-) \cap \partial M \subset B_{\partial M}(x_0, C\delta)$, for some uniform $C > 0$. Indeed, assume $\chi_-(y) \neq 0$ with $y \in \partial M$. Then $d(y, \widetilde{\gamma}_-) < \delta/2$; write $\widetilde{\gamma}_-(t_y)$ for the point which minimizes the distance from $y$ to $\widetilde{\gamma}_-$ (see Figure \ref{figure}). Note that by the preceding discussion $d(\widetilde{\gamma}_-(t_y), \partial M) < \delta$. By triangle inequality $d(y, x_0) < |t_y| + \delta/2$, so it suffices to show $|t_y| = \mc{O}(\delta)$. To see this, we compute
    \[
        \partial_t|_{t = 0}\rho(\gamma(t)) = \langle{\nabla \rho(x_0), v_0}\rangle = -\langle{\nu(x_0), v_0}\rangle > \varepsilon. 
    \]
    We extend $\rho$ as a smooth signed distance function to $\partial M$ (negative in the exterior). For the second derivative, we get from \eqref{eq:strictly-geodesic-convex} that
    \begin{equation}\label{eq:bounded}
        |\partial_t^2 \rho(\widetilde{\gamma}_-(t))| = |\langle{\nabla_{\dot{\widetilde{\gamma}}_-} \nabla \rho, \dot{\widetilde{\gamma}}_-(t)}\rangle| = |\nabla^2 \rho(\dot{\widetilde{\gamma}}_-(t), \dot{\widetilde{\gamma}}_-(t))| \leq \|\nabla^2 \rho\|_{C^0(\{|\rho| \leq \delta\})} =: C_0,
    \end{equation}
    as long as $d(\widetilde{\gamma}_-(t), \partial M) < \delta$. Integrating, we conclude that
    \begin{equation}\label{eq:angle-lower-bound}
        \partial_t \rho(\widetilde{\gamma}_-(t)) \geq \varepsilon - C_0|t|,\quad t \in I,
    \end{equation}
    where $I$ is the connected interval containing zero consisting of all $t$ such that $d(\widetilde{\gamma}_-(t), \partial M) < \delta$. Note that $t_y \in I$. We now choose $\delta > 0$ small enough such that the length of the geodesic segment of $\widetilde{\gamma}_- \cap \{|\rho| \leq \delta\}$ is bounded by $\frac{\varepsilon}{2C_0}$,
    i.e. so that the geodesic segments inside $\{|\rho| \leq \delta\}$ are small enough depending on the norm of $\nabla^2 \rho$. We then have by \eqref{eq:angle-lower-bound} that $\partial_t \rho(\widetilde{\gamma}_-(t)) \geq \frac{\varepsilon}{2}$ for $t \in I$, and so by integration
    \begin{equation}\label{eq:taylor-bound}
        \delta > |\rho(\widetilde{\gamma}_-(t_y))| = \left|\int_0^{t_y} \partial_s \rho(\widetilde{\gamma}_-(s))\, ds\right| > \frac{\varepsilon}{2} |t_y|
    \end{equation}
    which in particular implies that $|t_y| = \mc{O}(\delta)$, proving the original claim. (Note that the same argument also shows that $\ell(\gamma_-) = \mc{O}(\delta)$.)

    We record another consequence of \eqref{eq:taylor-bound}; we claim that
    \begin{equation}\label{eq:distance-equivalence}
        d_{\partial M}(y, x_0) \sim d(y, \widetilde{\gamma}_-(t_y)),\quad y \in \supp \chi_-,
    \end{equation}
    where $d_{\partial M}(\bullet, \bullet)$ denotes the distance along the boundary, and $\sim$ means uniformly equivalent. By Taylor's expansion, we have $d_{\partial M}(y, x_0) \sim d(y, x_0)$. Also, $d(y, \widetilde{\gamma}_-(t_y)) \leq d(y, x_0)$ by definition of $\widetilde{\gamma}_-(t_y)$. For the reverse inequality, we have
    \[
        d(y, x_0) \leq |t_y| + d(y, \widetilde{\gamma}_-(t_y)) < 2\varepsilon^{-1} |\rho(\widetilde{\gamma}_-(t_y))| + d(y, \widetilde{\gamma}_-(t_y)) \leq (1 + 2\varepsilon^{-1}) d(y, \widetilde{\gamma}_-(t_y)),
    \]
    where in the first estimate we used the triangle inequality, in the second one we used \eqref{eq:taylor-bound}, and in the last one we used that $|\rho(\widetilde{\gamma}_-(t_y))| \leq d(y, \widetilde{\gamma}_-(t_y))$ by definition of $\rho$. This proves the claim \eqref{eq:distance-equivalence}.

    We are now prepared to prove \eqref{eq:GB-boundary-estimate}. Indeed, we have
    \begin{align}\label{eq:usual-bound}
    \begin{split}
        \lambda^{\frac{n-1}{2}} &\int_{\partial M} |\chi_-|^p e^{-p\lambda \Im \Phi_-} |a_i|^p\, d\vol_g \\&\leq \mc{O}(1)\lambda^{\frac{n-1}{2}} \int_{B_{\partial M}(x_0, \mc{O}(\delta))} e^{-C\,\lambda p d_{\partial M}(x, x_0)^2}d\vol_g(x)
        = \mc{O}(1),
    \end{split}
    \end{align}
    where in the inequality we used that $a_i$ and $\chi_-$ are uniformly bounded, and that the phase satisfies $\Im \Phi_- \geq Cd(y, \widetilde{\gamma}_-)^2$ for some uniform $C > 0$ (see Proposition \ref{prop:phase-proposition} below), as well as \eqref{eq:distance-equivalence}. The passage to $\mc{O}(1)$ is justified by going to local normal geodesic coordinates based at $x_0 \in \partial M$ and a linear change of variables. This completes the proof of the first item.
    
    For the estimate $\|\partial_\nu u\|_{L^p(\partial M)} = \mc{O}(\lambda)$ we argue in the same way by expanding the terms; here we also make use of the fact that $\|\Phi_\pm\|_{C^1} = \mc{O}(1)$. The proof of the second item also follows from the same consideration. This completes the argument.
\end{proof}

We end this section with an additional auxiliary statement about the phase functions $\Phi^\pm$ and their non-degeneracy.

    \begin{proposition}\label{prop:phase-proposition}
        Let $(M, g)$, $\gamma = \gamma_{x_-, v_-}$ for some $(x_-, v_-) \in \mc{G}_{\varepsilon}$ be as in Proposition \ref{prop:boundary-estimate+geometry}, and write $u_\pm, \Phi^\pm$ for the Gaussian Beam approximate solution along $\gamma$. Write $S_g(x_-, v_-) = (x_+, v_+)$. Then there exists $C = C(g) >1$, such that 
        \[
            C d_{\partial M}(x, x_\pm)^2 \geq \Im \Phi^\pm(x) \geq C^{-1} d_{\partial M}(x, x_\pm)^2,\,\, x \in \supp \chi^\pm \cap \partial M. 
        \]
        Moreover, 
        \[
            \nabla^2_{\partial M} \Im \Phi^\pm(x) \geq C^{-1} g(x)|_{T\partial M},\quad x \in \supp (\chi^\pm) \cap \partial M,
        \]
        and by possibly further uniformly reducing $\delta > 0$ we may assume that $x_\pm$ is the only critical point of $\Im \Phi^\pm$ in $\supp(\chi^\pm) \cap \partial M$.
    \end{proposition}
    \begin{proof}
        For the first claim, let us consider the case of $\Phi^{-}$ without loss of generality; $\Phi^+$ is treated in the same way. We first recall that $\gamma_-$ denotes the (initial) segment of $\gamma$ in $M\setminus M_{\delta}$, and $\widetilde{\gamma}_-$ denotes a slight extension of $\gamma_-$ to the exterior of $M$ (embedded in a closed Riemannian manifold), see the proof of Proposition \ref{prop:boundary-estimate+geometry}. By Theorem \ref{thm:gbu-AC22}, 
        \[
            \nabla \Im \Phi^-|_{\gamma} = 0,\quad \nabla^2 \Im \Phi^-|_{\dot{\gamma}^\perp} \geq C^{-1} g|_{\dot{\gamma}^\perp},
        \]
        and in fact this also holds on $\widetilde{\gamma}_-$, see \cite[Lemma 6.3]{Ma_Sahoo_Salo}. By Taylor's expansion, and by going to Fermi's coordinates near $\widetilde{\gamma}_-$, we immediately get that $\Im \Phi^-(x) \geq C d(x, \widetilde{\gamma}_-)^2$ for $x$ sufficiently uniformly close to $\gamma$. After possibly uniformly decreasing the support of $\chi^\pm$, this implies the lower bound in the statement by applying \eqref{eq:distance-equivalence} (the proof of which uses that the geodesics we consider are non-tangential to the boundary in a uniform way). The upper bound follows by the uniform bound on the derivatives of degree $\leq 2$. This completes the proof of the first claim.

        We now prove the statement on the Hessian of $\Im \Phi^\pm$. Again, without loss of generality we prove the statement for $\Phi^+$; the case of $\Phi^-$ is treated in the same way. Since $\Phi^+$ has uniformly bounded derivatives, by Taylor's expansion it suffices to prove the estimate for $x = x_+$ for some uniform constant $C > 0$. The statement about uniqueness of the critical point follows, after possibly uniformly decreasing the support of $\chi^\pm$. We will use the following elementary formula valid for vector fields $X$ and $Y$ parallel to the boundary, and a function $f \in C^\infty(M)$,
        \[
            \nabla^2_{\partial M}f(X, Y) = \langle{(\nabla_{\partial M})_X \nabla_{\partial M} f, Y}\rangle = \langle{\nabla_X (\nabla f - (\partial_\nu f) \nu), Y}\rangle = \nabla^2 f(X, Y) - \partial_{\nu} f \mathbf{II}(X, Y),
        \]
        where in the second equality we used that the restriction covariant derivative is given by the orthogonal projection of $\nabla_X$, and in the last equality the definition of the second fundamental form. We will apply this with $f = \Im \Phi^+$, and use that $\partial_{\nu} \Im \Phi^+(x_+) = 0$. Note that we also have
        \begin{equation}\label{eq:zero-along-geodesic}
              \nabla^2 \Im \Phi^+(\dot{\gamma}, \bullet) = \langle{\nabla_{\dot{\gamma}} \nabla \Im \Phi^+, \bullet}\rangle = 0,  
        \end{equation}
        since $\nabla \Im \Phi^+|_{\gamma} = 0$ by Theorem \ref{thm:gbu-AC22}.
        
        Let $v \in T_{x_+} \partial M$ and write $v = av_+ + w$, where $a \in \mathbb{R}$ and $w \in (v_+)^\perp$. By definition $v$ is orthogonal to $\partial_\nu$, so
        \[
            a = -\frac{\langle{w, \nu}\rangle}{\langle{v_+, \nu}\rangle} \implies |a| \leq (\varepsilon_+)^{-1} |w|,
        \]
        where we used that $\langle{v_+, \nu}\rangle > \varepsilon_+$ thanks to Proposition \ref{prop:boundary-estimate+geometry}. It follows in particular that 
        \[
            |v|^2 = a^2 + |w|^2 \leq (1 + (\varepsilon_+)^{-2}) |w|^2.
        \]
        We may now compute
        \[
            \nabla^2_{\partial M} \Im \Phi^+ (v, v) = \nabla^2 \Im \Phi^+(v, v) = \nabla^2 \Im \Phi^+(w, w) \geq C|w|^2 \geq \frac{C}{1 + (\varepsilon_+)^{-2}} |v|^2,
        \]
        where in the second equality we used \eqref{eq:zero-along-geodesic}, and in the first inequality Theorem \ref{thm:gbu-AC22}. This completes the proof.
    \end{proof}

\section{Main proof: the three-solution strategy}\label{section:three-solution}

The section is devoted to the proof of Theorem \ref{thm_main_1}.

\begin{proof}[Proof of Theorem \ref{thm_main_1}]
We divide the proof into several steps.
\medskip

\emph{Step 1: boundary determination and extension.} Since by assumption $\Lambda_{g_1}^\lambda = \Lambda_{g_2}^\lambda$, by \cite[Proposition 1.3]{LeeUhlmann1989}, when $n \geq 3$ in boundary normal coordinates the full Taylor series of $g_1$ and $g_2$ agree. (This is unaffected by the $\lambda^2$ factor, see also \cite{DKSU}.) Let $\psi(x)$ be defined for $x$ in a sufficiently small neighbourhood of $\partial M$ as follows: follow the unique shortest geodesic $\gamma_1$ (with respect to $g_1$) orthogonal to $\partial M$ for time $t = d_{g_1}(x, \partial M)$, in the direction of $\partial M$, and then follow the unique geodesic $\gamma_2$ orthogonal to $\partial M$ (with respect to $g_2$) back to the interior for time $t$. This defines a diffeomorphism in a neighbourhood of the boundary, such that $\psi|_{\partial M} = \mathrm{Id}$, and $\psi^*g_2$ and $g_1$ agree to infinite order at $\partial M$ by the previous discussion (see also \cite[Proof of Lemma 2.1]{LeeUhlmann1989}). This diffeomorphism can be extended to a diffeomorphism of $M$, so we may assume that $g_1$ and $g_2$ have equal jets at the boundary. Note that $\partial M$ is still strictly convex with respect to $\psi^*g_2$, and $\psi^*g_2$ is non-trapping. When $n = 2$, we may simply take $\psi = \mathrm{Id}$ by the assumption of theorem. From now on, for simplicity we write $g_2$ instead of $\psi^*g_2$. 

We would now like to extend $(M, g_i)$ for $i = 1, 2$ so that the metrics agree. Observe first of all that the non-trapping assumption and strict convexity of the boundary are stable under $C^\infty$ perturbation (and even $C^2$ perturbation), see \cite{PSU_book}. Therefore, for $i = 1, 2$, $(M, g_i)$ admits an extension to non-trapping $(\widetilde{M}, \widetilde{g}_i)$ with strictly convex boundary, such that
\[
    \widetilde{g}_1|_{\widetilde{M} \setminus M} = \widetilde{g}_2|_{\widetilde{M} \setminus M},
\]
and moreover, we may assume that the equidistant hypersurfaces $\partial \widetilde{M}_{\eta}$ at distance $\eta \geq 0$ from $\partial \widetilde{M}$ are strictly convex for $\eta \in [0, \delta_0]$, for some $\delta_0 > 0$, such that $\partial \widetilde{M}_{\delta_0} = \partial M$. We are thus in a situation where Proposition \ref{prop:boundary-estimate+geometry} applies, and let $\varepsilon > 0$, $\delta \in (0, \delta_0)$ be as in the statement of that proposition so that its conclusions work for both $(\widetilde{M}, \widetilde{g}_1)$ and $(\widetilde{M}, \widetilde{g}_2)$ (just take the minimum of the respective constants for $\widetilde{g}_1$ and $\widetilde{g}_2$). Since $\widetilde{g}_1$ and $\widetilde{g}_2$ agree near $\partial \widetilde{M}$, and $g_1$ and $g_2$ have equal $C^\infty$-jets at $\partial M$, we may write $\partial_\pm S\widetilde{M}$ and $\partial_\pm SM$ for the respective common incoming and outgoing boundaries. We recall the notation
\[
    \mc{G}_\varepsilon := \{(x, v) \in \partial_- S\widetilde{M} \mid -(\widetilde{g}_1)_x(v, \nu(x)) > \varepsilon\}.
\]

We claim that 
\[
    (\ell_{\widetilde{g}_1}, S_{\widetilde{g}_1}) = (\ell_{\widetilde{g}_2}, S_{\widetilde{g}_2}) \quad \mathrm{on}\quad \mc{G}_\varepsilon \implies (\ell_{g_1}, S_{g_1}) = (\ell_{g_2}, S_{g_2}).
\]
Indeed, let $(x, v) \in \partial_-SM$. By Proposition \ref{prop:boundary-estimate+geometry}, for $i = 1, 2$, $\gamma_{x, v}^{g_i}$ (the $g_i$-geodesics generated by $(x, v)$) extends to a maximal non-tangential geodesic $\gamma_{(\widetilde{x}_i, \widetilde{v}_i)}^{\widetilde{g}_i}$ in $\widetilde{M}$, where $(\widetilde{x}_i, \widetilde{v}_i) \in \mc{G}_\varepsilon$. Moreover, the geodesic $\gamma_{(\widetilde{x}_i, \widetilde{v}_i)}^{\widetilde{g}_i}$ is a concatenation of two geodesic segments in the collar $\widetilde{M} \setminus M$ and $\gamma^{g_i}_{x, v}$. Since $\widetilde{g}_1 = \widetilde{g}_2$ in $\widetilde{M} \setminus M$, the incoming geodesic segments for $\widetilde{g}_1$ and $\widetilde{g}_2$ agree; denote them by $\gamma_{\mathrm{in}}$. It follows that $(\widetilde{x}_1, \widetilde{v}_1) = (\widetilde{x}_2, \widetilde{v}_2) =: (\widetilde{x}, \widetilde{v})$. As $S_{\widetilde{g}_1}(\widetilde{x}, \widetilde{v}) = S_{\widetilde{g}_2}(\widetilde{x}, \widetilde{v})$, the outgoing segments of $\gamma_{(\widetilde{x}, \widetilde{v})}^{\widetilde{g}_1}$ and $\gamma_{(\widetilde{x}, \widetilde{v})}^{\widetilde{g}_2}$ also agree, since again $\widetilde{g}_1 = \widetilde{g}_2$ in the collar; denote this segment by $\gamma_{\mathrm{out}}$. Following $\gamma_{\mathrm{out}}$ backwards to $\partial M$, we conclude that $\gamma_{(\widetilde{x}, \widetilde{v})}^{\widetilde{g}_1}$ and $\gamma_{(\widetilde{x}, \widetilde{v})}^{\widetilde{g}_2}$ end at the same point and direction; this proves that $S_{g_1}(x, v) = S_{g_2}(x, v)$ as required. Since
\[
    \ell_{g_1}(x, v) = \ell_{\widetilde{g}_1}(\widetilde{x}, \widetilde{v}) - \ell_{\widetilde{g}_1}(\gamma_{\mathrm{in}}) - \ell_{\widetilde{g}_1}(\gamma_{\mathrm{out}}) = \ell_{\widetilde{g}_2}(\widetilde{x}, \widetilde{v}) - \ell_{\widetilde{g}_1}(\gamma_{\mathrm{in}}) - \ell_{\widetilde{g}_1}(\gamma_{\mathrm{out}}) = \ell_{g_2}(x, v),
\]
we have that the travel times also agree, proving the claim. 

We note that if $(x, v) \in \partial_-S\widetilde{M} \setminus \mc{G}_\varepsilon$, by construction the geodesics $\gamma^{\widetilde{g}_1}_{x, v}$ and $\gamma^{\widetilde{g}_2}_{x, v}$ stay entirely in the collar $\widetilde{M} \setminus M$, and so they agree since $\widetilde{g}_1 = \widetilde{g}_2$ there. Therefore $(\ell_{\widetilde{g}_1}, S_{\widetilde{g}_1})(x, v) = (\ell_{\widetilde{g}_2}, S_{\widetilde{g}_2})(x, v)$. Therefore 
\begin{equation}\label{eq:equivalence-partial-scattering-data}
    (\ell_{\widetilde{g}_1}, S_{\widetilde{g}_1}) = (\ell_{\widetilde{g}_2}, S_{\widetilde{g}_2}) \quad \mathrm{on}\quad \mc{G}_\varepsilon \iff (\ell_{\widetilde{g}_1}, S_{\widetilde{g}_1}) = (\ell_{\widetilde{g}_2}, S_{\widetilde{g}_2}).
\end{equation}

From now on, for $i = 1, 2$, for simplicity we will write $(M, g_i)$ for $(\widetilde{M}, \widetilde{g}_i)$, with the aim of proving that the lens data agree on $\mc{G}_\varepsilon = \{(x, v) \in \partial_-SM \mid -g_x(v, \nu(x)) > \varepsilon\}$. By the previous paragraphs, this is sufficient to prove the theorem.

    \medskip

    \emph{Step 2: construction of suitable solutions and the setup.} Let $(x_0, v_0) \in \mc{G}_\varepsilon$, and denote $(x_i, v_i) := S_{g_i}(x_0, v_0)$ for $i = 1, 2$. Write $\gamma_i$ for the corresponding geodesics, and write $\gamma^\pm_i$ for the incoming and outgoing segments of $M \setminus M_{\delta} \cap \gamma_i$; note that $\gamma_1^- = \gamma_2^-$ since $g_1 = g_2$ there. Let $u_i$ be the Gaussian beam approximate solutions constructed in Theorem \ref{thm:gbu-AC22}; we take $p = 2$ and $K$ sufficiently large to be determined later. We also assume that the initial data from the construction of $u_i$ are the same at $(x_0, v_0)$.

    Similarly, we let $u$ be a (backwards oriented) Gaussian beam approximate solution with respect to $g_2$ whose initial data we control at $x_2$. According to Proposition \ref{prop:boundary-estimate+geometry}, for $i = 1, 2$ we may write
    \begin{align}\label{eq:GB-expansion}
    \begin{split}
        u_i|_{M \setminus M_{\delta/2}} &= u_i^- + u_i^+,\quad u_i^\pm = \lambda^{\frac{n - 1}{4}} e^{\I \lambda \Phi_i^\pm} a^{(i),\pm} \chi_i^\pm,\quad a^{(i),\pm} = \sum_{j = 0}^{N} a_j^{(i),\pm} \lambda^{-j},\\
        u|_{M \setminus M_{\delta/2}} &= u^- + u^+,\quad u^\pm = \lambda^{\frac{n - 1}{4}} e^{\I \lambda \Phi_2^\mp} a^{\pm} \chi_2^\mp,\quad a^{\pm} = \sum_{j = 0}^{N} a_j^{\pm} \lambda^{-j},
    \end{split}
    \end{align}
    where the cutoff functions $\chi_i^\pm$ are real-valued, with values in $[0, 1]$, supported in a $\delta_{\mathrm{GB}}$-neighbourhood of $\gamma^\pm_i$, and are equal to $1$ in a uniformly smaller neighbourhood of $\gamma^\pm_i$; here $\delta_{\mathrm{GB}} \in (0, \delta/2)$. (Note the discrepancy between the $\pm$ signs in the definition of $u^\pm$; this is because, by construction, we take the phase defining $u$ to be $\Phi_2$, but we reverse the direction of the geodesic $\gamma_2$.) Here $N = N(K)$ is sufficiently large depending on $K$. Note that in fact 
    \begin{equation}\label{eq:incoming-solution}
        u_1^-|_{M \setminus M_{\delta/2}} \equiv u_2^-|_{M \setminus M_{\delta/2}},\quad \chi_1^- \equiv \chi_2^-,
    \end{equation}
    since $u_i^-$ are supported in a $\delta_{\mathrm{GB}}$-neighbourhood of $\gamma_1^- = \gamma_2^-$, and they are determined uniquely from the metric and initial conditions, while the metrics agree on $M\setminus M_{\delta_0}$ by assumption (see also the proof of Proposition \ref{prop:boundary-estimate+geometry}).

    By Theorem \ref{thm:gbu-AC22} (note the shift in $K$ compared to the statement of the theorem),
    \begin{equation}\label{eq:Gaussian-Beam-approximate-solution}
        (-\Delta_{g_i} - \lambda^2) u_i = \mc{O}_{L^2}(\lambda^{- K}),\quad (-\Delta_{g_2} - \lambda^2)u = \mc{O}_{L^2}(\lambda^{- K}).
    \end{equation}
    We note that the norms in Proposition \ref{prop_resolvent_estimate} are defined using the metric. However, we here take Sobolev spaces to be defined using an arbitrary background metric; by compactness all such choices have equivalent norms and the implied constants depend only of $g_1$ and $g_2$. The exact solutions are constructed by solving the following equations for $i = 1, 2$
    \begin{align*}
        (-\Delta_{g_i} - \lambda^2) r_i = -(-\Delta_{g_i} - \lambda^2) u_i,\quad (-\Delta_{g_2} - \lambda^2) r = -(-\Delta_{g_2} - \lambda^2) u.
    \end{align*}
    By applying the resolvent estimate, see Proposition \ref{prop_resolvent_estimate}, and using \eqref{eq:Gaussian-Beam-approximate-solution}, we get for $i = 1, 2$ that

    \begin{equation}\label{eq:H^2-remainder}
        r_i = \mc{O}_{H^2}(\lambda^{1 - K}),\quad  r = \mc{O}_{H^2}(\lambda^{1 - K}),
    \end{equation}
    as well as that
    \[
        r_i = \mc{O}_{H^1}(\lambda^{-K}),\quad r = \mc{O}_{H^1}(\lambda^{-K}).
    \]
    By the trace theorem, as a  consequence we get for $i = 1, 2$ that
        \begin{align}\label{eq:remainder-estimate}
    \begin{split}
        &\|r_i|_{\partial M}\|_{L^2(\partial M)} \leq C \|r_i\|_{H^1(M)} = \mc{O}(\lambda^{- K}),\quad \|r|_{\partial M}\|_{L^2(\partial M)} = \mc{O}(\lambda^{- K}),\\
        &\|\partial_\nu r_i|_{\partial M}\|_{L^2(\partial M)} \leq C \|r_i\|_{H^2(M)} = \mc{O}(\lambda^{1 - K}),\quad \|\partial_\nu r|_{\partial M}\|_{L^2(\partial M)} = \mc{O}(\lambda^{1 - K}).
    \end{split}
    \end{align}
    \medskip

    \emph{Step 3: integration over $\partial M$.} By Proposition \ref{prop:extension-2}, we get that
        \begin{align*}
        &\langle{ \partial_\nu(u_1 - u_2), u}\rangle_{L^2(\partial M)} + \underbrace{\langle{ \partial_\nu(r_1 - r_2), u}\rangle_{L^2(\partial M)}}_{= \mc{O}(\lambda^{1 - K})} + \underbrace{\langle{ \partial_\nu((u_1 + r_1) - (u_2 + r_2)), r}\rangle_{L^2(\partial M)}}_{= \mc{O}(\lambda^{1 - K})}\\
        &=\langle{ \partial_\nu(u_1 + r_1) - \partial_\nu(u_2 + r_2), u + r}\rangle_{L^2(\partial M)}\\ 
        &=  \langle{u_1 + r_1 - (u_2 + r_2), \partial_\nu(u + r)}\rangle_{L^2(\partial M)}\\ 
        &= \langle{u_1 - u_2, \partial_\nu u}\rangle_{L^2(\partial M)} + \underbrace{\langle{r_1 - r_2, \partial_\nu u}\rangle_{L^2(\partial M)}}_{= \mc{O}(\lambda^{1 - K})} + \underbrace{\langle{u_1 + r_1 - (u_2 + r_2), \partial_\nu r}\rangle_{L^2(\partial M)}}_{=\mc{O}(\lambda^{1 - K})},
    \end{align*}
    where the terms involving the corrections to the approximate Gaussian Beams are estimated by Cauchy-Schwarz, and by Proposition \ref{prop:boundary-estimate+geometry} as well as \eqref{eq:remainder-estimate}. We thus get
    \[
        \langle{\partial_\nu(u_1 - u_2), u}\rangle_{L^2(\partial M)} - \langle{u_1 - u_2, \partial_\nu u}\rangle_{L^2(\partial M)} = \mc{O}(\lambda^{1 - K}).
    \]
    Using \eqref{eq:incoming-solution}, i.e. that $u_1^-$ and $u_2^-$ agree, we deduce
    \begin{equation}\label{eq:working-identity-1}
        \langle{\partial_\nu(u_1^+ - u_2^+), u}\rangle_{L^2(\partial M)} - \langle{u_1^+ - u_2^+, \partial_\nu u}\rangle_{L^2(\partial M)} = \mc{O}(\lambda^{1 - K}).
    \end{equation}
    According to $u = u^- + u^+$ we then deal with the corresponding terms arising from \eqref{eq:working-identity-1}.    
    \medskip
    
    \emph{Step 4: the principal term.} We now determine the principal term in \eqref{eq:working-identity-1}, that is the one obtained from the interaction of $u_i^+$ with $u^-$. The other term is treated below, and is shown to be of lower order.
    
    According to \eqref{eq:backscattering-lower-order} below and \eqref{eq:working-identity-1}, when $K \geq 1$ we have
    \begin{equation}\label{eq:working-identity-1'}
        \langle{\partial_\nu(u_1^+ - u_2^+), u^-}\rangle_{L^2(\partial M)} - \langle{u_1^+ - u_2^+, \partial_\nu u^-}\rangle_{L^2(\partial M)} = \mc{O}(\lambda^{\frac{1}{2}}).
    \end{equation}
    Let us now separate the top degree terms in \eqref{eq:working-identity-1'}. Using \eqref{eq:GB-expansion}, we get for $i = 1, 2$ that
    \begin{align}\label{eq:normal-derivative-expansion}
    \begin{split}
        \partial_\nu u_i^+ &= \lambda^{\frac{n - 1}{4}} e^{\I \lambda \Phi_i^+}\left( (\I \lambda \partial_\nu \Phi_i^+ \chi_i^+ + \partial_\nu \chi_i)\left(\sum_{j = 0}^N a_j^{(i), +} \lambda^{-j}\right) + \chi_i^+ \left(\sum_{j = 0}^N \partial_\nu a_j^{(i), +} \lambda^{-j}\right)\right)\\ 
        &= \lambda^{\frac{n - 1}{4}} e^{\I \lambda \Phi_i^+} a_0^{(i), +} \I \lambda \partial_\nu \Phi_i^+ \chi_i^+ + R_i^+,\\
        \partial_\nu u^- &= \lambda^{\frac{n - 1}{4}} e^{\I \lambda \Phi_2^+}\left( (\I \lambda \partial_\nu \Phi_2^+ \chi_2^+ + \partial_\nu \chi_2^+)\left(\sum_{j = 0}^N a_j^- \lambda^{-j}\right) + \chi_2^+ \left(\sum_{j = 0}^N \partial_\nu a_j^- \lambda^{-j}\right)\right)\\
        &=\lambda^{\frac{n - 1}{4}} e^{\I \lambda \Phi_2^+} a_0^- \I \lambda \partial_\nu \Phi_2^+ \chi_2^+ + R^-,
    \end{split}
    \end{align}
    where $R^-$ and $R_i^+$ are implicitly defined. Using Proposition \ref{prop:boundary-estimate+geometry} to treat the remainder terms, from \eqref{eq:working-identity-1'} we get that
    \begin{align*}
        \mc{O}(\lambda^{\frac{1}{2}}) &= \I \lambda^{\frac{n-1}{2} + 1} \int_{\partial M} (e^{\I \lambda \Phi_1^+} a_0^{(1), +} \partial_\nu \Phi_1^+ \chi_1^+ - e^{\I \lambda \Phi_2^+} a_0^{(2), +} \partial_\nu \Phi_2^+ \chi_2^+) e^{-\I \lambda \overline{\Phi}_2^+} \overline{a}_0^- \chi_2^+\, d\vol\\ 
        &+ \I \lambda^{\frac{n-1}{2} + 1} \int_{\partial M} (e^{\I \lambda \Phi_1^+} a_0^{(1), +} \chi_1^+ - e^{\I \lambda \Phi_2^+} a_0^{(2), +} \chi_2^+) e^{-\I \lambda \overline{\Phi}_2^+} \partial_\nu \overline{\Phi}_2^+ \overline{a}_0^- \chi_2^+\, d\vol.
    \end{align*}
    Collecting the terms with the same phase we get
    \begin{align}\label{eq:working-identity-2}
    \begin{split}
        \mc{O}(\lambda^{-\frac{1}{2}}) &= \underbrace{2\lambda^{\frac{n-1}{2}} \int_{\partial M} e^{-2\lambda \Im \Phi_2^+} a_0^{(2), +} \overline{a}_0^- \Re(\partial_{\nu}\Phi_2^+) (\chi_2^+)^2
\, d\vol}_{I_2:=}\\
        &- \underbrace{\lambda^{\frac{n-1}{2}} \int_{\partial M} e^{\I \lambda(\Phi_1^+ - \overline{\Phi}_2^+)} a_0^{(1), +} \overline{a}_0^- (\partial_\nu \overline{\Phi}_2^+ + \partial_\nu \Phi_1^+) \chi_1^+ \chi_2^+\, d\vol}_{I_1 :=}.
    \end{split}
    \end{align}
    \medskip

    \emph{Step 5: ``backscattering" is of lower order.} We claim that 
    \begin{equation}\label{eq:backscattering-lower-order}
        \langle{\partial_\nu(u_1^+ - u_2^+), u^+}\rangle_{L^2(\partial M)} - \langle{u_1^+ - u_2^+, \partial_\nu u^+}\rangle_{L^2(\partial M)} = \mc{O}(\lambda^{\frac{1}{2}}).
    \end{equation}      
    By an expansion similar to \eqref{eq:normal-derivative-expansion}, and using Proposition \ref{prop:boundary-estimate+geometry} to estimate the remainder terms, this is equivalent to showing
    \begin{align*}
        &\underbrace{\lambda^{\frac{n - 1}{2}}\int_{\partial M} e^{\I \lambda (\Phi_1^+ - \overline{\Phi_1^-})} a_0^{(1), +} \overline{a}_0^+ \chi_1^+ \chi_2^- (\partial_\nu \Phi_1^+ + \partial_\nu \overline{\Phi_1^-})\, d\vol}_{T_1 := }\\ 
        &- \lambda^{\frac{n - 1}{2}}\int_{\partial M} e^{\I \lambda (\Phi_2^+ - \overline{\Phi_2^-})} a_0^{(2), +} \overline{a}_0^+ \chi_2^+ \chi_2^- (\partial_\nu \Phi_2^+ + \partial_\nu \overline{\Phi_2^-})\, d\vol = \mc{O}(\lambda^{-\frac{1}{2}}),
    \end{align*}
    where we used that by construction $\Phi_1^- = \Phi_2^-$. It thus suffices to show that both terms on the left hand side are $\mc{O}(\lambda^{-\frac{1}{2}})$; by symmetry, it suffices to show this for the first term $T_1$. We will write $\Phi := \Phi_1^+ - \overline{\Phi_1^-}$.
    
    We first expand the last part of the integrand using Taylor's expansion. Note that by assumption we have $\partial_\nu \Phi_1^-(x_0) < -\varepsilon$, while $\partial_\nu \Phi_1^+(x_1) > \varepsilon_+$. By the uniform $C^2$ bounds on $\Phi_1^\pm$ we may assume that (after possibly reducing the constant $\delta_{\mathrm{GB}}$ which we recall controls the size of support of $\chi_i^\pm$, see Step 2)
    \begin{equation}\label{eq:uniform-lower-bound-exit}
        \Re \partial_\nu \Phi_{1}^-(x) < -\varepsilon/2,\quad x \in \supp (\chi_2^-),\quad \Re \partial_\nu \Phi_{1}^+(x) > \varepsilon_{+}/2,\quad x \in \supp(\chi_1^+),
    \end{equation}
    where we recall $\varepsilon_+ > 0$ comes from Proposition \ref{prop:boundary-estimate+geometry}. We may then compute for $x \in \supp(\chi_2^-) \cap \supp(\chi_1^+)$ that
    \begin{align*}
        &\partial_{\nu}(\Phi_1^+ + \overline{\Phi_1^-})(x) = \I \partial_{\nu}(\Im \Phi_1^+ - \Im \Phi_1^-)(x) + \partial_{\nu}(\Re \Phi_1^+ + \Re \Phi_1^-)(x)\\
        &=\mc{O}(d_{\partial M}(x_0, x) + d_{\partial M}(x_1, x)) +  \sqrt{|\nabla \Re \Phi_1^+(x)|^2 - |\nabla_{\partial M} \Re \Phi_1^+(x)|^2}\\ 
        &\hspace{9cm}- \sqrt{|\nabla \Re \Phi_1^-(x)|^2 - |\nabla_{\partial M} \Re \Phi_1^-(x)|^2}\\
        &= \mc{O}(d_{\partial M}(x_0, x) + d_{\partial M}(x_1, x)) +  F(x) (|\nabla \Re \Phi_1^+(x)|^2 - |\nabla \Re \Phi_1^-(x)|^2\\ 
        &\hspace{9cm}+ |\nabla_{\partial M} \Re \Phi_1^-(x)|^2 - |\nabla_{\partial M} \Re \Phi_1^+(x)|^2)\\
        &= \mc{O}(d_{\partial M}(x_0, x) + d_{\partial M}(x_1, x))\\ 
        &\hspace{3cm}+  F(x) \langle{\nabla_{\partial M} \Re \Phi_1^-(x) - \nabla_{\partial M} \Re \Phi_1^+(x), \nabla_{\partial M} \Re \Phi_1^-(x) + \nabla_{\partial M} \Re \Phi_1^+(x)}\rangle\\
        &= \mc{O}(d_{\partial M}(x_0, x) + d_{\partial M}(x_1, x)) -  F(x) \langle{\nabla_{\partial M} \Phi(x), \nabla_{\partial M} \Re \Phi_1^-(x) + \nabla_{\partial M} \Re \Phi_1^+(x)}\rangle,
    \end{align*}
    where we introduced the notation
    \[
            F(x) = \left[\sqrt{|\nabla \Re \Phi_1^+(x)|^2 - |\nabla_{\partial M} \Re \Phi_1^+(x)|^2} + \sqrt{|\nabla \Re \Phi_1^-(x)|^2 - |\nabla_{\partial M} \Re \Phi_1^-(x)|^2}\right]^{-1},
    \]
    and in the second line we used that $\Im \Phi_1^+$ and $\Im \Phi_1^-$ vanish to first order at $x_1$ and $x_0$, respectively, and Taylor's expansion, so that
    \[
        \partial_\nu \Im \Phi_1^+(x) = \mc{O}(d(x_1, x)),\quad \partial_\nu \Im \Phi_1^-(x) = \mc{O}(d(x_0, x)).
    \]
    Note crucially that the sign in front of the square root is different for $\Phi_1^+$ and $\Phi_1^-$. In the fourth line we used Taylor's expansion again and the fact that 
    \[
        |\nabla \Re \Phi_1^+(x_1)| = |\nabla \Re \Phi_1^-(x_0)| = 1,
    \]
    so that
    \[
        |\nabla \Re \Phi_1^+(x)|^2 = 1 + \mc{O}(d_{\partial M}(x, x_1)),\quad |\nabla \Re \Phi_1^-(x)|^2 = 1 + \mc{O}(d_{\partial M}(x, x_0)).
    \]
    In the last two lines we used that $\|F\|_{C^0} = \mc{O}(1)$ and by Taylor's expansion that
    \[
        \nabla_{\partial M} \Phi(x) = \nabla_{\partial M} \Re \Phi_1^+(x) - \nabla_{\partial M} \Re \Phi_1^-(x) + \mc{O}(d_{\partial M}(x, x_0) + d_{\partial M}(x, x_1)).
    \]
    Note that $\langle{\bullet, \bullet}\rangle$ denotes the Hermitian inner product. Therefore, we may write
    \begin{align*}
        T_1 &= -\underbrace{\lambda^{\frac{n - 1}{2}}\int_{\partial M} e^{\I \lambda \Phi} a_0^{(1), +} \overline{a}_0^+ \chi_1^+ \chi_2^- F \langle{\nabla \Phi, \nabla_{\partial M} \Re(\Phi_1^+ + \Phi_1^-)}\rangle\, d\vol}_{A:=}\\
        &+ \underbrace{\lambda^{\frac{n - 1}{2}}\int_{\partial M} e^{\I \lambda \Phi} a_0^{(1), +} \overline{a}_0^+ \chi_1^+ \chi_2^- \mc{O}(d_{\partial M}(x_0, x) + d_{\partial M}(x_1, x))\, d\vol.}_{B:=}
    \end{align*}
    In what follows we will use the following notation: $Yf$ denotes the $Y$ derivative of a function $f$, where $Y$ is a vector field; $Y$ will be the gradient vector field of a suitable function. For the first term of $T_1$ we simply write
    \begin{align*}
        A &= \frac{1}{\I \lambda} \lambda^{\frac{n-1}{2}}\int_{\partial M} \nabla_{\partial M} \Re(\Phi_1^+ + \Phi_1^-) \left(e^{\I \lambda \Phi}\right) \,a_0^{(1), +} \overline{a}_0^+ \chi_1^+ \chi_2^- F\,d\vol\\
        &= \frac{1}{\I \lambda} \lambda^{\frac{n-1}{2}} \int_{\partial M} e^{\I \lambda \Phi} (-\nabla_{\partial M} \Re(\Phi_1^+ + \Phi_1^-) - \Delta^{\partial M}_{g_1} \Re (\Phi_1^+ + \Phi_1^-))(a_0^{(1), +} \overline{a}_0^+ \chi_1^+ \chi_2^- F)\, d\vol\\
        &= \mc{O}(\lambda^{-1}),
    \end{align*}
    where in the second equality we integrated by parts (we write $\Delta_{g_1}^{\partial M}$ for the Laplacian of the metric $g_1$ on $\partial M$), and in the last inequality we estimate the integral in the usual way (as in \eqref{eq:usual-bound}) using Proposition \ref{prop:phase-proposition}, as well as that the integrands are uniformly bounded.
    
    For the other term, we simply have
    \begin{align*}
        |B| &\leq \mc{O}(1)\lambda^{\frac{n-1}{2}} \int_{\partial M} e^{- C\lambda (d_{\partial M}(x_0, x)^2 + d_{\partial M}(x_1, x)^2)} \chi_1^+ \chi_2^- \mc{O}(d_{\partial M}(x_0, x) + d_{\partial M}(x_1, x))\, d\vol\\
        &\leq \mc{O}(1)\lambda^{\frac{n-1}{2}} \left(\int_{\partial M} e^{- C\lambda d_{\partial M}(x_0, x)^2} \chi_2^- d_{\partial M}(x_0, x)\, d\vol\right.\\ 
        &\hspace{5cm}+ \left.\int_{\partial M} e^{- C\lambda d_{\partial M}(x_1, x)^2} \chi_1^+ d_{\partial M}(x_1, x)\, d\vol\right)\\
        &= \mc{O}(\lambda^{-\frac{1}{2}}),
    \end{align*}
    where in the first line we used Proposition \ref{prop:phase-proposition}, and in the last line we estimated both integrals in the usual way (as in \eqref{eq:usual-bound}). Since $T_1 = B - A$, the claim follows immediately.
 \begin{remark}
        We note that the loss of half a power of $\lambda$ in \eqref{eq:backscattering-lower-order} might be necessary. Indeed, the function $\partial_\nu \Im \Phi_1^+$ need not vanish to second order at $x_1$, so the Taylor expansion $\partial_\nu \Im \Phi_1^+(x) = \mc{O}(d_{\partial M}(x, x_1))$ might be sharp. In fact, a straightforward calculation shows that $\partial_\nu \Im \Phi_1^+(x) = \mc{O}(d_{\partial M}(x, x_1)^2)$ if and only if the geodesic $\gamma_1$ exits $M$ perpendicular to $\partial M$. (The same analysis is of course valid near $x_0$ and the incoming direction.)
    \end{remark}
    \medskip

    \emph{Step 6: exit points agree.} We apply the stationary phase lemma to the first integral in \eqref{eq:working-identity-2}. According to Proposition \ref{prop:phase-proposition} (which gives the non-degeneracy of the phase), we get that
    \begin{equation}\label{eq:I_1}
        I_2 = \underbrace{\frac{2\pi^{\frac{n - 1}{2}}}{\sqrt{\det \nabla^2_{\partial M} \Im \Phi_2^+(x_2)}} \Re(\partial_\nu \Phi_2^+) (x_2) a_0^{(2), +}(x_2) \overline{a}_0^-(x_2)}_{F_2(x_0, v_0) :=} + R,
    \end{equation}
    where the remainder $R$ is bounded by
    \begin{equation}\label{eq:remainder-stat-phase}
        |R| \leq C \lambda^{-1} \|a_0^{(2), +} \overline{a}_0^- (\chi_2^+)^2 \Re (\partial_\nu \Phi_2^+)\|_{C^{n + 2}(\partial M)} = \mc{O}(\lambda^{-1}),
    \end{equation}
    where in the last inequality we used Theorem \ref{thm:gbu-AC22} to uniformly bound the growth of derivatives, as well as the fact that
    \[
        \|f\|_{C^k(\partial M)} \leq C \|f\|_{C^k(M)},\quad \forall k \in \mathbb{Z}_{> 0},\quad f \in C^\infty(M), 
    \]
    for some $C = C(k) > 0$, since the boundary covariant derivatives can be expressed as derivatives in $M$. Note that thanks to the formula from Theorem \ref{thm:gbu-AC22} (where $H^{(2)}$ denotes the solution to the Riccati equation corresponding to $g_2$),
    \begin{equation}\label{eq:amplitute-riccati}
        a_0^{(2), +}(x_2) = \exp\left(-\frac{1}{2} \int_0^{\ell_{g_2}(v_0)} \tr H^{(2)}(s)\, ds\right),
    \end{equation}
    there is a uniform lower bound for $|a_0^{(2), +}(x_2)|$. Note also that by assumption $a_0^-(x_2) = 1$. The determinant $\det(\nabla_{\partial M}^2 \Im \Phi_2^+)(x_2)$ is also uniformly bounded from above by Theorem \ref{thm:gbu-AC22}, while we recall from the assumption on $\gamma_2$ that $\partial_\nu \Re \Phi_2^+(x_2) > \varepsilon_+$. Therefore, from \eqref{eq:I_1}, for large enough $\lambda$ we get for some $C_2 = C_2(g_2) > 0$ that
    \begin{equation}\label{eq:I_1-lower-bound}
        |I_2| \geq C_2 > 0.
    \end{equation}

    We next bound the other term in \eqref{eq:working-identity-2} as follows
    \begin{align}\label{eq:I_2-upper-bound}
    \begin{split}
        |I_1| &\leq \mc{O}(1) \lambda^{\frac{n-1}{2}} \int_{\partial M} e^{-C \lambda(d_{\partial M}(x_1, x)^2 + d_{\partial M}(x_2, x)^2)} \chi_1^+ \chi_2^+\, d\vol\\ 
        &\leq \mc{O}(1) e^{- \frac{C\lambda}{4} d_{\partial M}(x_1, x_2)^2}  \lambda^{\frac{n-1}{2}} \int_{\partial M} e^{-\frac{C}{2} \lambda(d_{\partial M}(x_1, x)^2 + d_{\partial M}(x_2, x)^2)} \chi_1^+ \chi_2^+\, d\vol\\ 
        &\leq \mc{O}(1) e^{- \frac{C\lambda}{4} d_{\partial M}(x_1, x_2)^2}. 
    \end{split}
    \end{align}
    where in the first inequality we used Proposition \ref{prop:phase-proposition}, in the second one we applied the triangle inequality, as well as the arithmetic-geometric mean inequality, and in the last estimate we bounded the integral by going to normal coordinates based at $x_1$ in the usual way (see \eqref{eq:usual-bound}). Combining \eqref{eq:I_2-upper-bound} and \eqref{eq:I_1-lower-bound}, as well as \eqref{eq:working-identity-2} we get
    \[
        0 < C_2 \leq \mc{O}(1) e^{-C \lambda d_{\partial M}(x_1, x_2)^2} + \mc{O}(\lambda^{-\frac{1}{2}}).
    \]
    We set
    \begin{equation}\label{eq:Lambda1}
        \Lambda := \sup_{(x, v) \in \mc{G}_{\varepsilon}} d( \pi \circ S_{g_1}(x, v), \pi \circ S_{g_2}(x, v)).
    \end{equation}
    If $\Lambda = 0$, there is nothing to do; if $\Lambda > 0$, then taking $\lambda$ large enough depending on $\Lambda$ contradicts the previous inequality and shows that $x_1 = x_2$, completing the proof of this step.
    \medskip

    \emph{Step 7: equality of exit directions.} According to Step 6, we now assume that $x_1 = x_2 =: x_+$. For simplicity we write $\Phi := \Phi_1^+ - \overline{\Phi_2^+}$. The term $I_2$ in \eqref{eq:working-identity-2} is treated in the same way as in Step 6; we now study $I_1$ and apply the method of non-stationary phase. Indeed, by Theorem \ref{thm:gbu-AC22}, we have
    \[
        \nabla_{\partial M}\Phi(x_{+}) = \Pi_{TM \to T\partial M}(v_1 - v_2),
    \]
    where $\Pi := \Pi_{TM \to T\partial M}$ denotes the orthogonal projection to $T\partial M$. We let
    \begin{equation}\label{eq:exit-directions}
        \Lambda := \sup_{(x, v) \in \mc{G}_{\varepsilon}} \left|\Pi_{TM \to T\partial M} (S_{g_1}(x, v) - S_{g_2}(x, v)) \right|.
    \end{equation}
    If $\Lambda = 0$, we are done and we do nothing; if $\Lambda > 0$, we argue as follows. Assume that $(x_0, v_0) \in \mc{G}_{\varepsilon}$ is such that $|\nabla_{\partial M} \Phi (x_{+})| > \Lambda/2$. Since the second derivatives of $\Phi$ are uniformly bounded, after possibly reducing $\delta_{\mathrm{GB}}$, we may assume that
    \[
        |\nabla_{\partial M}\Phi(x)| \geq \frac{\Lambda}{4},\quad x \in \supp \chi_1^+.
    \]
    We now argue explicitly integrating by parts by considering the following first order differential operator (vector field)
    \[
        L := \frac{\nabla_{\partial M} \overline{\Phi}}{|\nabla_{\partial M}\Phi|^2}.
    \]
    We clearly have by construction that $L e^{\I \lambda \Phi} = \I \lambda e^{\I \lambda\Phi}$. The formal (real) transpose of this operator is given by, for $\phi \in C^\infty_{\comp}(\supp(\chi_1^+))$
    \begin{equation}
    \begin{split}\label{eq:transpose}
        &\hspace{-0.5cm}L^\top \phi = -\nabla_{\partial M}\overline{\Phi} \left(\frac{\phi}{|\nabla_{\partial M}\Phi|^2}\right) - \frac{\phi \Delta^{\partial M}_{g_1}\overline{\Phi}}{|\nabla_{\partial M}\Phi|^2}\\
        &\hspace{-0.5cm}= -\frac{\nabla_{\partial M}\overline{\Phi} (\phi)}{|\nabla_{\partial M}\Phi|^2} + \frac{\phi (\langle{\nabla_{\nabla_{\partial M} \overline{\Phi}} \nabla_{\partial M} \Phi, \nabla_{\partial M} \overline{\Phi}}\rangle + \langle{\nabla_{\partial M} \Phi, \nabla_{\nabla_{\partial M} \overline{\Phi}}  \nabla_{\partial M} \overline{\Phi}}\rangle)}{|\nabla_{\partial M}\Phi|^4} - \frac{\phi \Delta^{\partial M}_{g_1}\overline{\Phi}}{|\nabla_{\partial M}\Phi|^2},
    \end{split}
    \end{equation}
    where $\langle{\bullet, \bullet}\rangle$ denotes the complex bilinear pairing induced by the metric. We now have integrating by parts that
    \begin{align*}
        I_1 = \frac{1}{\I \lambda} \lambda^{\frac{n - 1}{2}} \int_{\partial M} e^{\I \lambda \Phi} L^\top(a_0^{(1), +} \overline{a}_0^- \chi_1^+ \chi_2^+ \partial_{\nu}(\Phi_1^+ + \overline{\Phi}_2^+))\, d\vol = \mc{O}(\lambda^{-1}),
    \end{align*}
    where in the final estimate we used that the integrands are uniformly bounded in $C^2$, and that by assumption $|\nabla_{\partial M} \Phi| \geq \Lambda/4$ on $\supp \chi_1^+$. We also used Proposition \ref{prop:phase-proposition} and estimated the Gaussian in the usual way. Thus by \eqref{eq:I_1-lower-bound} and \eqref{eq:working-identity-2}, we get that
    \[
        0 < C \leq |I_2| \leq |I_1| + \mc{O}(\lambda^{-\frac{1}{2}}) = \mc{O}(\lambda^{-\frac{1}{2}}),
    \]
    which is a contradiction for $\lambda$ large enough. This shows that $\Lambda = 0$, i.e. that the tangential components of $S_{g_1}(x, v)$ and $S_{g_2}(x, v)$ agree. Since the $v_i$ are uniquely determined by $\Pi_{TM \to T\partial M} v_i$, for $i = 1, 2$, the conclusion follows.
    \medskip

    \emph{Step 8: equality of exit times.} Note that thanks to \eqref{eq:equivalence-partial-scattering-data}, equality of $S_{g_1}$ and $S_{g_2}$ on $\mc{G}_{\varepsilon}$ implies $S_{g_1} = S_{g_2}$ on the whole $\partial_-SM$. Equality of travel times follows directly from Proposition \ref{prop:known?}. Alternatively, we argue as follows. 
    
    We will study the integral defining $I_1$. According to Steps 6 and 7, we may now assume that $(x_1, v_1) = (x_2, v_2) =: (x_{+}, v_{+})$. We will still write $\Phi := \Phi_1^+ - \overline{\Phi_2}^+$. It follows from the assumptions that $\nabla_{\partial M}\Phi(x_{+}) = 0$. The Hessian can be estimated, for each $v \in T\partial M \otimes \mathbb{C}$ and $x \in \supp (\chi_1^+) \cap \supp (\chi_2^+)$, as
    \begin{equation}\label{eq:nabla-squared-invertible}
       \Im(\nabla^2_{\partial M}\Phi(x_{+})(v, \overline{v})) = \nabla^2_{\partial M} \Im \Phi (x_{+}) (\Re v, \Re v) + \nabla^2_{\partial M} \Im \Phi (x_{+}) (\Im v, \Im v) \geq C|v|^2,
    \end{equation}
    for some uniform $C > 0$, where in the last estimate we used Proposition \ref{prop:phase-proposition}. It follows that $\nabla^2_{\partial M} \Phi(x_{+})$ is non-degenerate. (It can also be shown that $|\det(\nabla^2_{\partial M} \Phi(x_{+}))| > C$ for some uniform $C > 0$, but we will not need this.) Moreover, $x = x_{+}$ is the only critical point of $\Phi$ in $\supp(\chi_1^+) \cap \supp(\chi_2^+)$, since it is the unique critical point of $\Im \Phi = \Im \Phi_1^+ + \Im \Phi_2^+$ similarly to Proposition \ref{prop:phase-proposition}, using that the sum of two positive definite matrices is positive definite. We may thus apply the stationary phase lemma to $I_1$ to get
    \begin{equation}\label{eq:stationary-phase-I_2}
        I_1 = e^{\I \lambda(\ell_{g_1} - \ell_{g_2})(x_0, v_0)} \underbrace{\left[\det\left(\frac{\nabla^2_{\partial M} \Phi(x_{+})}{2\pi \I}\right)\right]^{-\frac{1}{2}} a_0^{(1), +}(x_+) \partial_{\nu}(\Phi_1^+ + \overline{\Phi}_2^+)(x_{+})}_{F_1(x_0, v_0):=} + \mc{O}(\lambda^{-1}),
    \end{equation}
    where the remainder is uniformly bounded similarly as before, see \eqref{eq:remainder-stat-phase}. By \eqref{eq:stationary-phase-I_2} and \eqref{eq:I_1}, as well as \eqref{eq:working-identity-2}, we may write
    \[
        e^{\I \lambda(\ell_{g_1} - \ell_{g_2})(x_0, v_0)} F_1(x_0, v_0) - F_2(x_0, v_0) = \mc{O}(\lambda^{-\frac{1}{2}}). 
    \]
    Similarly to the derivation of \eqref{eq:I_1-lower-bound}, $F_1$ is bounded from below by a positive constant. Indeed, note that the determinant part here is bounded because $|\det \nabla^2_{\partial M} \Phi| = \mc{O}(1)$, and the fractional power $-\frac{1}{2}$ is defined by picking the usual holomorphic branch, while $|a_0^{(1), +}(x_+)|$ is bounded uniformly away from zero similarly to Step 6 and the use of formula \eqref{eq:amplitute-riccati}. Also, as $(x_0, v_0) \in \mc{G}_{\varepsilon}$, we have by Proposition \ref{prop:boundary-estimate+geometry} that
    \[
        \partial_\nu(\Phi_1^+ + \overline{\Phi}_2^+)(x_{+}) = 2\langle{v_{+}, \partial_{\nu}}\rangle > 2\varepsilon_+.
    \]
    Therefore, if we write $F := \frac{F_1}{F_2}$, we have
    \begin{equation}
        F(x, v) e^{\I \lambda(\ell_{g_1} - \ell_{g_2})(x, v)} - 1 = \mc{O}(\lambda^{-\frac{1}{2}}),\quad (x, v) \in \mc{G}_\varepsilon.
    \end{equation}
    We also note that the functions $F_1, F_2: \mc{G}_{\varepsilon} \to \mathbb{C}$ are smooth functions of $(x_0, v_0)$, by smooth dependence of solutions to ODEs with respect to a parameter (see \cite{PSU_book} for instance for similar constructions). We note here that in fact $F(x, v) \equiv 1$ for $(x, v)$ near the boundary of $\mc{G}_\varepsilon$, so it admits an extension by $1$ to $\partial_-SM$. Indeed, by construction the geodesics generated by such $(x, v)$ are the same for $g_1$ and $g_2$, and the metrics agree near such geodesics. Thus the amplitudes and the phases agree as they are given by solving the same ODEs with the same initial conditions. We conclude that $F(x, v) = 1$, and also that $\ell_{g_1}(x, v) = \ell_{g_2}(x, v)$. Therefore
    \begin{equation}\label{eq:stationary-phase-consequence}
        F(x, v) e^{\I \lambda(\ell_{g_1} - \ell_{g_2})(x, v)} - 1 = \mc{O}(\lambda^{-\frac{1}{2}}),\quad (x, v) \in \partial_-SM.
    \end{equation}
    
    Consider the quantities
    \[
        \Lambda := \sup_{(x, v) \in \partial_-SM} |\ell_{g_1}(x, v) - \ell_{g_2}(x, v)|,\quad \Lambda_F := \sup_{(x, v) \in \partial_-SM} \left|1 - |F|\right|.
    \]
    If $\Lambda_F > 0$, then taking $\lambda$ large enough contradicts \eqref{eq:stationary-phase-consequence}. Thus we may assume that $|F| = 1$.

    Let $z \in \partial_-SM$. By strict convexity of the boundary (see Step 1 and Proposition \ref{prop:boundary-estimate+geometry}), there exists $z_0$ near the boundary of $\partial_-SM$ (i.e. $z_0$ is nearly tangent to $\partial M$), and such that $\pi(z_0) = \pi(z)$, i.e. $z_0$ and $z$ are in the same fibre of $\pi: \partial_-SM \to \partial M$. Moreover, $z_0$ and $z$ are connected by a path 
    \[
        \Gamma: [0, 1] \to \partial_-SM,\quad \Gamma(0) = z_0,\quad \Gamma(1) = z,
    \]
    so that $\Gamma \subset \partial_-SM(\pi(z_0))$ lies in the same fibre. Note that $\Gamma$ can be chosen of uniformly bounded length. Consider the map $F\circ \Gamma: [0, 1] \to \mathbb{S}^1$ and write $F\circ \Gamma = e^{\I f_\Gamma}$ where $f_\Gamma(0) = 0$ (possible since $F(z_0) = 1$) and $f_\Gamma$ is real-valued and smooth on $\Gamma$, and is uniquely determined. Such $f_\Gamma$ exists since $[0, 1]$ is contractible. (More generally, we could work with a logarithm function on a contractible set $U \subset \partial_-SM$.) 
    
    Now \eqref{eq:stationary-phase-consequence} and Taylor's expansion of trigonometric functions imply that for some uniform $C > 0$ we have
    \begin{equation}\label{eq:G_gamma}
        G_\Gamma(t) := \lambda(\ell_{g_1} - \ell_{g_2}) \circ \Gamma(t) + f_\Gamma(t)  \in 2\pi \mathbb{Z} + [-C\lambda^{-\frac{1}{2}}, C\lambda^{-\frac{1}{2}}],\quad \forall t \in [0, 1].
    \end{equation}
    We now take $\lambda$ large enough so that $C\lambda^{-\frac{1}{2}} < \pi$. This means that the right hand side of \eqref{eq:G_gamma} is a disjoint union of closed intervals, and since by construction $G_\Gamma(0) = 0$, by the intermediate value theorem $G_\Gamma$ has to take values only in $[-C\lambda^{-1/2}, C\lambda^{-1/2}]$. More invariantly, and evaluated at $t = 1$, we showed that
    \[
        \lambda(\ell_{g_1} - \ell_{g_2})(z) + \mc{O}(\lambda^{-\frac{1}{2}}) = -f_\Gamma(1) = -f_\Gamma(1) + f_\Gamma(0) = \I \int_\Gamma \frac{dF}{F}.
    \]
    Since $dF$ is uniformly bounded, we conclude that $\ell_{g_1}(z) - \ell_{g_2}(z) = \mc{O}(\lambda^{-1})$ with uniform implied constants. Since $z$ was arbitrary, we conclude that
    \[
        \ell_{g_1} - \ell_{g_2} = \mc{O}(\lambda^{-1}).
    \]
    
    If $\Lambda > 0$, taking $\lambda$ large enough gives a contradiction, and therefore the lens data agree. This completes the proof.
    \end{proof}
    
        \begin{remark}
        The equality of travel times $\ell_{g_1} \equiv \ell_{g_2}$ (which we showed in Step 8 of the above proof) and \eqref{eq:stationary-phase-consequence} immediately imply that $F_1 \equiv F_2$ for $\lambda$ large enough depending only on $g_1$ and $g_2$.
    \end{remark}

    \begin{remark}
        We note that in fact the uniformity in the construction of Gaussian beams (see Theorem \ref{thm:gbu-AC22}) is not necessary if we are not interested in the dependence of $\lambda_0$ on $g_1$ and $g_2$. Indeed, one could argue by contradiction, assuming that $S_{g_1} \neq S_{g_2}$ and that there exists a sequence $(\lambda_i)_i$, $\lambda_i \to \infty$ as $i \to \infty$, such that $\Lambda^{\lambda_i}_{g_1} = \Lambda_{g_2}^{\lambda_i}$. Pick $(x, v)$ in the incoming boundary such that $S_{g_1}(x, v) \neq S_{g_2}(x, v)$. Then the same argument as above gets us to a contradiction; note that we only used the Gaussian beam construction along a single geodesic. Equality of travel times follows by Proposition \ref{prop:known?}, for instance, or the procedure in Step 8.
    \end{remark}

    \section{Quantification of uniqueness: three-solutions}\label{section:quantification-three-solutions}

    Here we derive a result which quantifies the uniqueness result from Theorem \ref{thm_main_1}. We briefly recall the notation from the introduction: $N \geq 2$ is an integer, $B > 1$, $\delta > 0$, $T > 0$, and  $g_0$ an arbitrary non-trapping Riemannian metric on the compact manifold $M$ with strictly convex boundary $\partial M$. We consider a subset of Riemannian metrics $\mathcal{M}_{N, B, \delta, T}(g_0)$ defined in \eqref{eq:space-of-metrics}, bounded in $C^N$, agreeing with $g_0$ near the boundary, and with maximal geodesics having length $\leq T$. In the following we always have $h = \lambda^{-1}$.
    
    \subsection{Uniformity with respect to the metric.}\label{ssec:uniformity} We note here that all the constants appearing in Theorem \ref{thm:gbu-AC22} can be made uniform with respect to $g \in \mathcal{M}_{N, B, \delta, T}(g_0)$. Indeed, the injectivity radius is continuous under $C^2$ perturbations of the metric and all geodesics have uniformly bounded length by assumption. The phase $\Phi$ and the amplitude $a$ are obtained by solving ODEs along geodesics, see \cite[Lemmas 6.3, 6.4]{Ma_Sahoo_Salo}. According to \cite[Lemma 6.5]{Ma_Sahoo_Salo}, the coefficients of these ODEs are uniformly bounded and hence so are their solutions; the extension to the manifold is also uniformly bounded as the derivatives of the exponential map are. The lower bound on $H$ follows in the same way as in the proof of \cite[Lemma 6.3]{Ma_Sahoo_Salo}. The amplitudes are dealt with in a similar way. We are left to observe that the geodesics in $(M, g)$ have a uniformly bounded number of self-intersections, as in \cite[Lemma 6.7]{Ma_Sahoo_Salo}, since as noted above the injectivity radius is bounded from below and the geodesics have bounded lengths. This shows that the pushforward from a normal ball bundle of the geodesic to its tubular neigbhourhood is a map whose fibres have uniformly bounded finite cardinality, and completes the argument, see \cite[Section 6.2]{Ma_Sahoo_Salo}.

    We explain the uniformity in Proposition~\ref{prop_resolvent_estimate}. Using the common boundary collar, extend all metrics to the same closed manifold $\mathcal N$, with a common metric outside $M$ and uniform $C^N$ bounds and ellipticity. We denote the extensions again by $g$. Choose fixed nested extensions $M\Subset M_1\Subset M_2\Subset\mathcal N$, whose boundaries $\partial M_1$ and $\partial M_2$ are strictly convex outward parallel hypersurfaces in the common collar. Here $\mathcal N$, $M_1$, $M_2$ and the cutoffs below depend only on $(M,g_0)$ and $\delta$, and not on $g$.

For $(x,\eta)\in S_g^*M_1$, let $\tau_g^+(x,\eta)$ be the first time at which the unit-speed geodesic starting at $x$ with initial velocity $\eta^\sharp$ reaches $\partial M_2$, where $\eta^\sharp$ is the metric dual of $\eta$. The bound $\tau_g\le T$ on $M$ and the common
collar give a uniform upper bound for $\tau_g^+$. Strict convexity of the collar and the fixed separation between $\partial M_1$ and $\partial M_2$ give a uniform positive lower bound for the outward normal velocity at this exit. Differentiating the transverse exit equation and using the geodesic-flow estimates therefore controls the required finite derivatives of $\tau_g^+$ uniformly, provided $N$ is sufficiently large.

Set $p_g(x,\xi)=|\xi|_g^2$, with Hamilton vector field $H_{p_g}$. The escape function
\[
a_g^0(x,\xi)
=
-\frac{|\xi|_g}{2}\,
\tau_g^+\!\left(x,\frac{\xi}{|\xi|_g}\right),
\qquad x\in M_1,\quad \xi\ne0,
\]
is homogeneous of degree one and satisfies $H_{p_g}a_g^0=p_g$, since $|\xi|_g$ is conserved and the Hamiltonian flow has speed $2|\xi|_g$. Multiplying $a_g^0$ by a fixed spatial cutoff supported in $M_1^\circ$ and equal to one on a fixed neighborhood of $M$, and by a fixed smooth scalar cutoff vanishing near zero and equal to one on $[\tfrac12,\infty)$, evaluated at $|\xi|_g$, gives a real order-one symbol
$a_g$ with uniformly bounded required finite symbol seminorms. The identity $H_{p_g}a_g=p_g$ holds over this neighborhood for $|\xi|_g$ near one.

With the subsequent spatial localizations supported in this neighborhood, the elliptic, positive-commutator, and Sobolev-order arguments of
\cite[Lemma~3.2]{Ma_Sahoo_Salo} give the estimate needed at $s=-2$ with uniform constants. Only finitely many symbol seminorms and sufficiently high fixed finite-order remainder estimates are required. The Hahn--Banach and duality argument in
\cite[Proposition~3.1]{Ma_Sahoo_Salo}, followed by the uniform equivalence of the spectral semiclassical $H^2$ norm and the corresponding weighted derivative norm, proves Proposition~\ref{prop_resolvent_estimate} with $C$ and $\lambda_0$ depending only on $(M,g_0)$, $B$, $\delta$, and $T$.

    Next, we observe that the conclusions of Proposition \ref{prop:boundary-estimate+geometry} still hold for some $\delta > 0$ and $\varepsilon > 0$ fixed for all $g \in \mathcal{M}_{N, B, \delta, T}(g_0)$. Indeed, this follows by direct inspection of its proof, since the metrics agree in $M \setminus M_{\delta_0}$ with $g_0$. Similarly, the uniform estimates and the splitting of solutions into $u = u^+ + u^-$, proved in Proposition \ref{prop:boundary-estimate+geometry} are still valid. Finally, the conclusions of Proposition \ref{prop:phase-proposition} remain valid uniformly for the metrics in $\mathcal{M}_{N, B, \delta, T}(g_0)$, since the metrics agree in $M \setminus M_{\delta_0}$ with $g_0$, and as discussed above there are uniform bounds for $\nabla^2 \Im \Phi$ along geodesics. In the proof of Theorem \ref{theorem:quantification} below we use these facts without further discussion.

    \subsection{Proof} We can now prove the main result of this section.
    \begin{proof}[Proof of Theorem \ref{theorem:quantification}]
        The proof is obtained by quantifying steps in the proof of Theorem \ref{thm_main_1}; we therefore omit some details to avoid repetition and closely follow the mentioned proof. Note that since by assumption all metrics agree with a fixed background metric $g_0$ near the boundary, Step 1 is not needed here. In particular, we do not consider any extensions and work directly on the original manifold $M$ throughout the proof.

        Recall that for $\varepsilon > 0$, we introduced
        \[
            \mc{G}_\varepsilon = \{(x, v) \in \partial_-SM \mid -g_x(v, \nu(x)) > \varepsilon\}.
        \]
    Let $\delta_0\in (0, \delta)$ be such that $(\partial M_{\eta})_{\eta \in [0, \delta_0]}$ are strictly convex (with respect to $g_0$). As explained in \S \ref{ssec:uniformity}, we may take a uniform $\varepsilon > 0$ that works for all metrics $g \in \mathcal{M}_{N, B, \delta, T}(g_0)$, so that the conclusions of Proposition \ref{prop:boundary-estimate+geometry} and \ref{prop:phase-proposition} are still valid.
        \medskip

        \emph{Step 1: the solutions and the integral identity.} We now let $u_1$, $u_2$, and $u$ be the Gaussian beams as in the proof of Theorem \ref{thm_main_1}; we similarly consider $r_1$, $r_2$, and $r$ to be the corrections to actual solutions. By Proposition \ref{prop:extension-2}, we see that
        \begin{align*}
            \langle{\partial_\nu(u_1 + r_1 - u_2 - r_2), u + r}\rangle_{L^2(\partial M)} + &\langle{(\Lambda_{g_2}^\lambda - \Lambda_{g_1}^\lambda)(u_1 + r_1), u + r}\rangle_{L^2(\partial M)}\\ 
            &= \langle{u_1 + r_1 - u_2 - r_2, \partial_\nu (u + r)}\rangle_{L^2(\partial M)}.
        \end{align*}
        We note that
        \begin{equation}\label{eq:difference-DN-maps}
            \langle{(\Lambda_{g_1}^\lambda - \Lambda_{g_2}^\lambda)(u_1 + r_1), u + r}\rangle_{L^2(\partial M)} = \mc{O}(\|\Lambda_{g_1}^\lambda - \Lambda_{g_2}^\lambda\|_{H^1_h \to L^2}).
        \end{equation}
        Here we used that $\Lambda_{g_i}^\lambda: H^1_h(\partial M) \to L^2(\partial M)$ is bounded as $\Lambda_{g_i}^\lambda$ is a first order pseudodifferential operator, see \cite{LeeUhlmann1989} (we do not make any claim about the uniformity in $h$ of the norm), as well as that
        \[
            \|r_1\|_{H^1_h(\partial M)}^2 = \|r_1\|_{L^2(\partial M)}^2 + h^2 \|dr_1\|_{L^2(\partial M)}^2 \leq C (\|r_1\|^2_{H^1(M)} + h^2\|r_1\|_{H^2(M)}^2) = \mc{O}(h^2), 
        \]
        as follows from the trace theorem and \eqref{eq:H^2-remainder} above (for $K = 1$). Moreover, we used the fact that $\|u_i\|_{H^1_h(\partial M)} = \mc{O}(1)$, $i = 1, 2$, and $\|u\|_{L^2(\partial M)} = \mc{O}(1)$, as follows from Proposition \ref{prop:boundary-estimate+geometry} and its proof; finally, we used \eqref{eq:remainder-estimate}. This shows that indeed
        \[
            \|u_1 + r_1\|_{H^1_h(\partial M)} = \mc{O}(1),\quad \|u + r\|_{L^2(\partial M)} = \mc{O}(1),
        \]
        establishing \eqref{eq:difference-DN-maps}. The analysis of the other terms remains the same and using \eqref{eq:difference-DN-maps}, as well as that $u_1^- = u_2^-$ (see also \eqref{eq:working-identity-1}) and \eqref{eq:backscattering-lower-order} (showing that the backscattering term is of lower order), we get that
        \begin{equation}\label{eq:working-identity-3}
            \langle{\partial_\nu(u_1^+ - u_2^+), u^-}\rangle_{L^2(\partial M)} - \langle{u_1^+ - u_2^+, \partial_\nu u^-}\rangle_{L^2(\partial M)} = \mc{O}(\lambda^{\frac{1}{2}}) + \mc{O}(\|\Lambda_{g_1}^\lambda - \Lambda_{g_2}^\lambda\|_{H^1_h \to L^2}).
        \end{equation}
        As a consequence, after expanding the integrals in powers of $\lambda$, we have in the same way as in \eqref{eq:working-identity-2} that
        \begin{align}\label{eq:working-identity-3'}
        \begin{split}
        \mc{O}(\lambda^{-\frac{1}{2}}) + &\mc{O}(h\|\Lambda_{g_1}^\lambda - \Lambda_{g_2}^\lambda\|_{H^1_h \to L^2}) = \underbrace{2\lambda^{\frac{n-1}{2}} \int_{\partial M} e^{-2\lambda \Im \Phi_2^+} a_0^{(2), +} \overline{a}_0^- \Re(\partial_{\nu}\Phi_2^+) (\chi_2^+)^2
\, d\vol}_{I_2:=}\\
        &- \underbrace{\lambda^{\frac{n-1}{2}} \int_{\partial M} e^{\I \lambda(\Phi_1^+ - \overline{\Phi}_2^+)} a_0^{(1), +} \overline{a}_0^- (\partial_\nu \overline{\Phi}_2^+ + \partial_\nu \Phi_1^+) \chi_1^+ \chi_2^+\, d\vol}_{I_1 :=},
        \end{split}
        \end{align}
        where we adopt the notation for the phases, amplitudes, and cut-offs from the proof of Theorem \ref{thm_main_1}.
        \medskip

        \emph{Step 2: the distance between exit points.} We will use the expansion \eqref{eq:I_1} for $I_2$, together with the corresponding remainder estimate \eqref{eq:remainder-stat-phase}. By the same analysis, we get that
        \[
            |I_2| \geq C_2 > 0,
        \]
        for some $C_2$ depending only on $\mathcal{M}_{N, B, \delta, T}(g_0)$. On the other hand, $I_1$ is bounded in the same way as in \eqref{eq:I_2-upper-bound}. Therefore, we get from \eqref{eq:working-identity-3'} that
        \begin{equation}\label{eq:lower-bound-working}
            0 < C_2 \leq \mc{O}(\lambda^{-\frac{1}{2}}) + \mc{O}(h\|\Lambda_{g_1}^\lambda - \Lambda_{g_2}^\lambda\|_{H^1_h \to L^2}) + \mc{O}(1) e^{-\frac{C\lambda}{4} d_{\partial M}(x_1, x_2)^2}.
        \end{equation}
        For some uniform $\lambda_0, \varepsilon_0 > 0$ depending only on $\mathcal{M}_{N, B, \delta, T}(g_0)$, we get that parts of the right hand side can be absorbed by the left hand side if $h\|\Lambda_{g_1}^\lambda - \Lambda_{g_2}^\lambda\|_{H^1_h \to L^2} < \varepsilon_0$, and if $\lambda > \lambda_0$, so that
        \[
            0 < \frac{C_2}{2} \leq \mc{O}(1) e^{-\frac{C\lambda}{4} d_{\partial M}(x_1, x_2)^2} \implies d_{\partial M}(x_1, x_2) \leq C \lambda^{-\frac{1}{2}}.
        \]
        \medskip

        \emph{Step 3: exit directions.} We will consider the quantity
        \[
            \Lambda := \sup_{(x, v) \in \mc{G}_{\varepsilon}} \|\mc{P}^{\partial M}_{x_1 \to x_2} \Pi S_{g_1} - \Pi S_{g_2}\|_{g_0},
        \]
        where $\mc{P}^{\partial M}_{x_1 \to x_2}$ denotes parallel transport with respect to $g_0$ along the unique shortest boundary geodesic from $x_1$ to $x_2$; this geodesic exists as soon as $\lambda_0$ is taken uniformly large enough by the previous step. Also, $\Pi$ denotes the orthogonal projection to $T \partial M$. We claim that $\Lambda = \mc{O}(\lambda^{-\frac{1}{2}})$ for some uniform implied constants. Assume for the sake of contradiction that $\Lambda > C_0\lambda^{-\frac{1}{2}}$ for some $C_0 \gg 1$ to be determined later, and consider $(x_0, v_0) \in \mc{G}_{\varepsilon}$ such that 
        \[
            \|\mc{P}^{\partial M}_{x_1 \to x_2} \Pi \dot{\gamma}_1 - \Pi \dot{\gamma}_2\|_{g_0} > C_0\lambda^{-\frac{1}{2}}.
        \]
        We treat the integral $I_1$ using the method of non-stationary phase. We write, for some $c > 0$ to be determined later (and related to $C_0$ introduced above) 
        \begin{align*}
            I_1 &= \underbrace{\lambda^{\frac{n - 1}{2}} \int_{d_{\partial M}(x, x_1) \geq c\lambda^{-\frac{1}{2}}} e^{\I \lambda \Phi} a_0^{(1), +} \overline{a}_0^- \chi_1^+ \chi_2^+ \partial_{\nu}(\Phi_1^+ + \overline{\Phi}_2^+)\, d\vol}_{I_{11}:=}\\ 
            &+ \underbrace{\lambda^{\frac{n - 1}{2}} \int_{d_{\partial M}(x, x_1) \leq c\lambda^{-\frac{1}{2}}} e^{\I \lambda \Phi} a_0^{(1), +} \overline{a}_0^- \chi_1^+ \chi_2^+ \partial_{\nu}(\Phi_1^+ + \overline{\Phi}_2^+)\, d\vol}_{I_{12}:=},
        \end{align*}
        where we recall that $\Phi = \Phi_1^+ - \overline{\Phi}_2^+$. For $I_{11}$, we note that
        \begin{align}\label{eq:I-11-bound}
        \begin{split}
            |I_{11}| &\leq \mc{O}(1) \lambda^{\frac{n-1}{2}} \int_{d_{\partial M}(x, x_1) \geq c\lambda^{-\frac{1}{2}}} e^{-C\lambda d_{\partial M}(x, x_1)^2} \chi_1^+\,d\vol\\
            &\leq e^{-\frac{c^2C}{2}} \mc{O}(1) \lambda^{\frac{n-1}{2}} \int_{d_{\partial M}(x, x_1) \geq c\lambda^{-\frac{1}{2}}} e^{-\frac{C}{2}\lambda d_{\partial M}(x, x_1)^2} \chi_1^+\,d\vol\\
            &\leq\mc{O}(1) e^{-\frac{c^2C}{2}} < \frac{C_2}{4},
        \end{split}
        \end{align}
        where in the first inequality we used Proposition \ref{prop:phase-proposition}, and in the second one we used the lower bound on $d_{\partial M}(x, x_1)$, while in the third one we used the bound similar to \eqref{eq:usual-bound}. In the last inequality we took $c$ uniformly large enough, where we recall $C_2$ was introduced in the preceding step.

        To bound $I_{12}$, we first compute
        \begin{align*}
            \nabla_{\partial M} \Re \Phi(x) &= \mc{P}^{\partial M}_{x_1 \to x}\nabla_{\partial M} \Re \Phi(x_1) + \mc{O}(d_{\partial M}(x, x_1))\\ 
            &= \mc{P}^{\partial M}_{x_1 \to x} (\Pi \dot{\gamma}_1 - \mc{P}^{\partial M}_{x_2 \to x_1} \Pi \dot{\gamma}_2) + \mc{O}(d_{\partial M}(x_1, x_2)) + \mc{O}(d_{\partial M}(x, x_1)),
        \end{align*}
        where in both equalities we used Taylor's expansion for $\nabla_{\partial M} \Re \Phi_1^+$ and $\nabla_{\partial M} \Re \Phi_2^+$. We conclude that for $d_{\partial M}(x, x_1) \leq c\lambda^{-\frac{1}{2}}$, we have
        \begin{equation}\label{eq:phase-lower-bound}
            \|\nabla_{\partial M} \Re \Phi(x)\| \geq C_0\lambda^{-\frac{1}{2}} - \mc{O}(\lambda^{-\frac{1}{2}}) - c \mc{O}(\lambda^{-\frac{1}{2}}) \geq \frac{C_0}{2} \lambda^{-\frac{1}{2}},
        \end{equation}
        where in the first estimate we also used Step 2, and the second one is valid for uniformly large enough $\lambda$, and also for $C_0$ uniformly large enough (depending on $c$). Thus for 
        \[
            L := \frac{\nabla_{\partial M} \overline{\Phi}}{|\nabla_{\partial M} \Phi|^2},
        \]
        as in Step 7 of the proof of Theorem \ref{thm_main_1}, we may compute
        \begin{align}\label{eq:I-12-bound}
        \begin{split}
            &|I_{12}| = \left|\frac{1}{\I \lambda} \lambda^{\frac{n - 1}{2}} \int_{d_{\partial M}(x, x_1) \leq c\lambda^{-\frac{1}{2}}} e^{\I \lambda \Phi} L^\top(a_0^{(1), +} \overline{a}_0^- \chi_1^+ \chi_2^+ \partial_{\nu}(\Phi_1^+ + \overline{\Phi}_2^+))\, d\vol\right.\\ 
            & \hspace{3cm}+\left.\frac{1}{\I \lambda} \lambda^{\frac{n-1}{2}} \int_{d_{\partial M}(x, x_1) = c\lambda^{-\frac{1}{2}}} e^{\I \lambda \Phi} a_0^{(1), +} \overline{a}_0^- \chi_1^+ \chi_2^+ \partial_{\nu}(\Phi_1^+ + \overline{\Phi}_2^+)\, \iota_L d\vol\right|\\
            &\leq C_0^{-2} \mc{O}(\lambda^{\frac{n - 1}{2}}) \int_{d_{\partial M}(x, x_1) \leq \mc{O}(1)} e^{-C\lambda d_{\partial M}(x, x_1)^2} \, d\vol\\ 
            &\hspace{4cm} + C_0^{-1} \mc{O}(\lambda^{\frac{1}{2}}) \mc{O}(\lambda^{\frac{n-3}{2}}) \int_{d_{\partial M}(x, x_1) = c\lambda^{-\frac{1}{2}}} e^{-C \lambda d_{\partial M}(x, x_1)^2} d\vol\\
            &\leq C_0^{-2}\mc{O}(1) + C_0^{-1} c^{n - 2} e^{-Cc^2} \mc{O}(1) < \frac{C_2}{8},
        \end{split}
        \end{align}
        where in the first equality we integrated by parts using Stokes' theorem, and $\iota_L$ denotes contraction with the vector field $L$. Note that we get a boundary term due to the fact that $\chi_1^+$ is not compactly supported in $\{x \mid d_{\partial M}(x, x_1) \leq c\lambda^{-\frac{1}{2}}\}$. We recall that $L^\top$ denotes the (real) transpose of $L$ so that by \eqref{eq:phase-lower-bound} and \eqref{eq:transpose}, we have
        \[
            |L^\top \phi(x)| \leq C_0^{-2} \mc{O}(\lambda) \|\phi\|_{C^1},\quad d_{\partial M}(x, x_1) \leq c \lambda^{-\frac{1}{2}},\quad \phi \in C^\infty_{\comp}(\supp \chi_1^+).
        \]
        In the first inequality we used this bound, together with Proposition \ref{prop:phase-proposition} and the uniform bound on the integrand; we also used that $\|L\|_{C^0} = C_0^{-1}\mc{O}(\lambda^{\frac{1}{2}})$ as a vector field. In the second inequality we used the usual integration argument (see \eqref{eq:usual-bound}), as well as that the $(n-2)$-dimensional volume of the geodesic sphere of radius $c\lambda^{-\frac{1}{2}}$ is proportional to $c^{n - 2}\lambda^{-\frac{n-2}{2}}$. The last inequality holds for $C_0$ and $\lambda$ uniformly large enough. Combining the estimates \eqref{eq:I-11-bound} and \eqref{eq:I-12-bound} for $I_{11}$ and $I_{12}$, and plugging inside \eqref{eq:working-identity-3'}, similarly to \eqref{eq:lower-bound-working}, we get
        \[
            0 < \frac{C_2}{2} < \frac{3C_2}{8},
        \]
        which is a contradiction. This completes the argument by contradiction, and shows that $\Lambda \leq C_0\lambda^{-\frac{1}{2}}$, so combined with Step 2, we have
        \[
            d(S_{g_1}, S_{g_2}) = \mc{O}(\lambda^{-\frac{1}{2}}).
        \]
        Before estimating the travel times, we record a consequence of the preceding estimates, namely
        \begin{equation}\label{eq:small-gradient}
            \nabla_{\partial M}\Phi(x_1) = \mc{O}(\lambda^{-\frac{1}{2}}) + \nabla_{\partial M}\Phi_1^+(x_1) - \mc{P}^{\partial M}_{x_2 \to x_1} \nabla_{\partial M}\overline{\Phi}_2^+(x_2) = \mc{O}(\lambda^{-\frac{1}{2}}),
        \end{equation}
        where in the first equality we used Taylor's expansion and Step 2, and in the second we used Step 3, as well as that the gradient of the imaginary part of the phases vanish at the exit points by Theorem \ref{thm:gbu-AC22}.

        \begin{remark}
            This step is slightly different from Step 7 of the uniqueness proof. The reason is that previously, we took the support of the Gaussian beam sufficiently small depending on $\Lambda$. This is not allowed here, because we want to make everything uniform with respect to $\mc{M}_{N, B, \delta, T}(g_0)$ and so we split $I_1$ into pieces corresponding to distance to $x_1$, and this is where the natural scale $\lambda^{-\frac{1}{2}}$ appears.
        \end{remark}
        \medskip

        \emph{Step 4: exit times.} Let us first motivate the need for the following version of the stationary phase lemma. According to \eqref{eq:small-gradient}, the phase $\Phi = \Phi_1^+ - \overline{\Phi}_2^+$ satisfies $\nabla_{\partial M} \Phi(x_1) = \mc{O}(\lambda^{-\frac{1}{2}})$, but it need not have critical points on $\partial M$. Also, in the local coordinate $y$ on $\partial M$, in the Gaussian scale $|y| = \mc{O}(\lambda^{-\frac{1}{2}})$, the linear term $\lambda \nabla_{\partial M} \Phi(x_1)\cdot y = \mc{O}(1)$ is not negligible and so a direct Taylor's expansion argument does not help. However, as we will demonstrate, there is a nearby unique critical point of an almost analytic extension of $\Phi$.

        We formulate an important remark about the uniformity of the estimates in Lemma \ref{lemma:stationary-phase} below over the set of metrics in $\mc{M}_{N, B, \delta, T}(g_0)$ and $(x, v) \in \mc{G}_\varepsilon$. Everywhere in the proof of the expansion of $I_1$, almost analytic estimates are used only to a fixed sufficiently high finite order. We may choose the (finite) almost analytic extensions so that the required derivative seminorms are controlled by finitely many seminorms of the phases and amplitudes. Taking $N$ sufficiently large therefore makes the remainder uniform. Similarly, the estimate on the uniqueness of the critical point up to negligible terms is only uniform up to $\mc{O}(\lambda^{-L})$ for $L$ fixed, after taking $N$ large enough. In summary, we take the convention that all estimates involving a superscript $\infty$, are uniform for any fixed superscript $L$, for $N$ sufficiently large.

        \begin{lemma}\label{lemma:stationary-phase}
            The following expansion holds:
            \begin{align*}
                I_1 = e^{\I \lambda \Phi(z_{\mathrm{crit}})} \frac{(2\pi)^{\frac{n-1}{2}}}{\sqrt{\det (-\I \nabla^2_{\partial M} \Phi(x_1))}} a_0^{(1), +}(x_1) \partial_\nu(\Phi_1^+ + \overline{\Phi}_2^+)(x_1) + \mc{O}(\lambda^{-\frac{1}{2}}), 
            \end{align*}
            where $z_{\mathrm{crit}} \in \mathbb{C}^{n - 1}$ is the unique solution of $\partial_z \Phi(z_{\mathrm{crit}}) = 0$ close to the origin, and where $\Phi$ denotes an almost analytic extension of the phase $\Phi|_{\mathbb{R}^{n-1}}$ in a normal coordinate chart based at $x_1$. Moreover, the critical point $z_{\mathrm{crit}}$ is unique up to a rapidly decaying term, i.e. if we choose another almost analytic extension $\Phi'$ of $\Phi$, then the corresponding critical point $z_{\mathrm{crit}}'$ satisfies
            \[
                z_{\mathrm{crit}} - z_{\mathrm{crit}}' = \mc{O}(\lambda^{-\infty}).
            \]
        \end{lemma}
        Here $\Phi$ is an almost analytic extension of $\Phi|_{\mathbb{R}^{n-1}}$ to $\mathbb{C}^{n - 1}$ in the sense that
        \[
            \partial_{\Bar{z}} \Phi = \mc{O}(|\Im z|^\infty),
        \]
        where $\partial_{\Bar{z}}$ and $\partial_z$ denote Wirtinger's derivatives.         
        \begin{proof}
            The statement and its proof are almost the same as the leading term expansion in \cite[Theorem 2.3]{Melin-Sjostrand-75}, but the assumptions are slightly different; therefore, we give a self-contained proof based on almost analytic functions, closely following \cite{Melin-Sjostrand-75}. We simplified the argument slightly as we only need the leading term expansion. Note first that the phase $\Phi$ in our case does not necessarily have any critical points. Let us write $b$ for the amplitude of $I_1$, and denote by the letter $\Phi$ and $b$ the almost analytic extensions of the same functions to $\mathbb{C}^{n-1}$, i.e. satisfying
            \[
                \partial_{\Bar{z}} \Phi = \mc{O}(|\Im z|^{\infty}),\quad \partial_{\Bar{z}} b = \mc{O}(|\Im z|^{\infty}).
            \]
            We first locate the critical point by using Taylor's expansion with Wirtinger's derivatives around zero
            \[
                \partial_z \Phi(z) = \partial_z \Phi(0) + \partial_z^2 \Phi(0) z + \mc{O}(|z|^2),
            \]
            where we also used that $\partial_{\Bar{z} z}^2 \Phi(0) = \partial_{\Bar{z}}^2\Phi(0) = 0$ following from almost analyticity. This implies, since $\partial_z^2 \Phi(0) = \partial_x^2\Phi(0) = \nabla^2 \Phi(0)$ is invertible by \eqref{eq:nabla-squared-invertible} and the lines around it, that the critical point needs to satisfy
            \begin{equation}\label{eq:implicit-equation}
                z = -[\partial_z^2 \Phi(0)]^{-1} \partial_z \Phi(0) + \mc{O}(|z|^2).
            \end{equation}
            Note that $\partial_z \Phi(0) = \partial_x \Phi(0) = \mc{O}(\lambda^{-\frac{1}{2}})$ by \eqref{eq:small-gradient}, so by the contraction mapping principle there is a unique solution $z_{\mathrm{crit}}$ of this equation close to the origin satisfying 
            \begin{equation}\label{eq:critical-point}
                |z_{\mathrm{crit}}| = \mc{O}(\lambda^{-\frac{1}{2}}).
            \end{equation} 
            (Equivalently, we may use the implicit function theorem directly to deduce the existence of $z_{\mathrm{crit}}$.)
            
            By using Taylor's expansion at $z = z_{\mathrm{crit}}$ of the phase, we get
            \begin{equation}\label{eq:taylor-expansion-crit}
                \Phi(z) = \Phi(z_{\mathrm{crit}}) + \frac{1}{2}\partial_z^2\Phi(z_{\mathrm{crit}})(z - z_{\mathrm{crit}}) \cdot (z - z_{\mathrm{crit}}) + \underbrace{R_3 + R_{\geq 4} + Q}_{R :=},
            \end{equation}
            where we write
            \[
                R_3 := \frac{1}{6} \partial_z^3 \Phi (z_{\mathrm{crit}}) (z - z_{\mathrm{crit}}, z - z_{\mathrm{crit}}, z - z_{\mathrm{crit}}),
            \]
            and $R_{\geq 4}$ for the remainder in the Taylor expansion of degree $\geq 4$; also, $Q$ by definition contains all terms in the Taylor expansion up to order $\leq 3$, which contain a $\partial_{\Bar{z}}$ derivatives. We also used that $z_{\mathrm{crit}}$ satisfies $\partial_z \Phi(z_{\mathrm{crit}}) = 0$. By almost analyticity of $\Phi$ and properties of Taylor's expansion, we have
            \begin{equation}\label{eq:taylor-easy}
                Q = \mc{O}(|\Im z_{\mathrm{crit}}|^{\infty}),\quad R_{\geq 4} = \mc{O}(|z - z_{\mathrm{crit}}|^4),\quad R_3 = \mc{O}(|z - z_{\mathrm{crit}}|^3).
            \end{equation}
             (Note that once an almost analytic extension is chosen, the critical point is fixed, but together with \eqref{eq:critical-point} we will be able to use the decay in $\lambda$ of the remainder to our favour.) Using this expansion and varying $z \in \mathbb{R}^{n - 1}$, it is possible to show (exactly the same as in \cite[Proof of Lemma 2.1]{Melin-Sjostrand-75})
            \begin{equation}\label{eq:crit-point-imaginary-positive}
                \Im \Phi(z_{\mathrm{crit}}) \geq C|\Im z_{\mathrm{crit}}|^2,
            \end{equation}
            for some uniform $C > 0$.

            Let us shift the contour of the integral using Stokes' theorem to a submanifold of $\mathbb{C}^{n - 1}$ containing $z_{\mathrm{crit}}$. Consider the affine homotopy 
            \[
                \Gamma: [0, 1] \times \mathbb{R}^{n - 1} \to \mathbb{C}^{n - 1},\quad \Gamma(s, x) := x + \Re z_{\mathrm{crit}} + \I s \Im z_{\mathrm{crit}},
            \]
            so that
            \begin{equation}\label{eq:stokes-1}
                \lambda^{-\frac{n-1}{2}} I_1 = \int_{\mathbb{R}^{n - 1}} e^{\I \lambda \Phi} b\, dx =  \int_{\mathbb{R}^{n - 1}} e^{\I \lambda \Phi} b\, dz = \int_{\mathbb{R}^{n - 1} + z_{\mathrm{crit}}}  e^{\I \lambda \Phi} b\, dz + \underbrace{\int_{\Gamma} \Bar{\partial}(e^{\I \lambda \Phi} b) \wedge dz}_{S :=},
            \end{equation}
            where in the last equality we used Stokes' theorem. We claim that $S = \mc{O}(\lambda^{-\infty})$ so that this term does not affect the stationary phase expansion. To do this, we estimate the imaginary part of $\Phi$ as follows. Taking imaginary parts of the Taylor expansion \eqref{eq:taylor-expansion-crit}, we get
            \begin{align*}
                &\Im \Phi(x + \Re z_{\mathrm{crit}} + \I s \Im z_{\mathrm{crit}}) = \Im \Phi(z_{\mathrm{crit}}) + \frac{1}{2} \Im [\partial_z^2 \Phi(z_{\mathrm{crit}})(x + \I (s - 1) \Im z_{\mathrm{crit}}, x + \I (s - 1) \Im z_{\mathrm{crit}})]\\ &+ \Im R(x + \Re z_{\mathrm{crit}} + \I s \Im z_{\mathrm{crit}})\\
                &= \Im \Phi(z_{\mathrm{crit}}) + \frac{1}{2} [\Im \partial_z^2 \Phi](z_{\mathrm{crit}})(x, x) - \frac{1}{2}(s - 1)^2[\Im \partial_z^2 \Phi](z_{\mathrm{crit}})(\Im z_{\mathrm{crit}}, \Im z_{\mathrm{crit}})\\ 
                &+ [\Re \partial^2_z\Phi](z_{\mathrm{crit}}) ((s-1)\Im z_{\mathrm{crit}}, x) + \Im R(x + \Re z_{\mathrm{crit}} + \I s \Im z_{\mathrm{crit}}).
            \end{align*}
            Inserting $s = 0$ in the previous equality we get
            \begin{align*}
                &\Im \Phi(x + \Re z_{\mathrm{crit}}) = \Im \Phi(z_{\mathrm{crit}}) + \frac{1}{2} [\Im \partial_z^2 \Phi](z_{\mathrm{crit}})(x, x) - \frac{1}{2}[\Im \partial_z^2 \Phi](z_{\mathrm{crit}})(\Im z_{\mathrm{crit}}, \Im z_{\mathrm{crit}})\\ 
                &- [\Re \partial^2_z\Phi](z_{\mathrm{crit}}) (\Im z_{\mathrm{crit}}, x) + \Im R(x + \Re z_{\mathrm{crit}}).
            \end{align*}
            Multiplying the second equality by $(s - 1)$ and adding to the first one, we obtain
            \begin{align*}
                &\Im \Phi(x + \Re z_{\mathrm{crit}} + \I s \Im z_{\mathrm{crit}}) + (s - 1)\Im \Phi(x + \Re z_{\mathrm{crit}}) =\\
                &= s \Im \Phi(z_{\mathrm{crit}}) + \frac{1}{2}s [\Im \partial_z^2 \Phi](z_{\mathrm{crit}})(x, x) - \frac{1}{2} s(s - 1) [\Im \partial_z^2 \Phi](z_{\mathrm{crit}})(\Im z_{\mathrm{crit}}, \Im z_{\mathrm{crit}})\\ 
                &+ (s - 1)\Im R(x + \Re z_{\mathrm{crit}}) + \Im R(x + \Re z_{\mathrm{crit}} + \I s \Im z_{\mathrm{crit}}).
            \end{align*}
            Since $\partial_z^2\Phi(0) = \partial_x^2 \Phi(0)$, by continuity we may assume that $[\Im \partial_z^2\Phi](z_{\mathrm{crit}}) > 0$ (and also using \eqref{eq:critical-point}). We therefore have
            \begin{align}\label{eq:job-almost-done}
            \begin{split}
                \Im \Phi(\Gamma(s, x)) &\geq Cs(|x|^2 + |\Im z_{\mathrm{crit}}|^2) + s \mc{O}(|x|^3 + |\Im z_{\mathrm{crit}}|^3) - |\Im R(\Gamma(s, x)) - \Im R(\Gamma(0, x))|\\
                &\geq \frac{C}{2}s(|x|^2 + |\Im z_{\mathrm{crit}}|^2) - |\Im R(\Gamma(s, x)) - \Im R(\Gamma(0, x))|,
            \end{split}
            \end{align}
            where we used that $\Im \Phi|_{\mathbb{R}^{n - 1}} \geq 0$, that $s \in [0, 1]$, as well as \eqref{eq:taylor-easy}, and \eqref{eq:crit-point-imaginary-positive}. In the second inequality we absorbed the higher order terms, by taking $\lambda$ large enough to treat the $|\Im z_{\mathrm{crit}}|^3$ based on \eqref{eq:critical-point}, and taking $x$ uniformly small (which corresponds to taking $\delta_{\mathrm{GB}}$ small enough, i.e. the support of $b$). We are left to estimate the last term; recall that $R = R_3 + R_{\geq 4} + Q$ which we estimate term by term. The term corresponding to $Q$ is estimated using the mean value theorem and \eqref{eq:taylor-easy}; note that $Q$ consists of polynomials in $z$ and $\Bar{z}$, so the $s$-derivatives of $Q$ are bounded as $|\Im z_{\mathrm{crit}}|^\infty$. The term corresponding to $R_{\geq 4}$ is again estimated using the mean value theorem and the estimate 
            \[
                |\partial_y R_{\geq 4}(\Gamma(s', x))| \leq C(|x|^3 + |\Im z_{\mathrm{crit}}|^3),\quad s' \in [0, 1],
            \]
            which follows directly from the integral formula for the remainder in Taylor's expansion. Finally, the term corresponding to $R_3$ is bounded directly using the formula
            \begin{align*}
                &\Im [\partial_z^3 \Phi(z_{\mathrm{crit}})(x -\I \Im z_{\mathrm{crit}} + \I s \Im z_{\mathrm{crit}}, x -\I \Im z_{\mathrm{crit}} + \I s \Im z_{\mathrm{crit}}, x -\I \Im z_{\mathrm{crit}} + \I s \Im z_{\mathrm{crit}})]\\ 
                & -\Im [\partial_z^3 \Phi(z_{\mathrm{crit}})(x - \I \Im z_{\mathrm{crit}}, x - \I \Im z_{\mathrm{crit}}, x - \I \Im z_{\mathrm{crit}})]\\
                &= s\mc{O}(|x|^3 + |\Im z_{\mathrm{crit}}|^3),
            \end{align*}
            where we simply observe that the constant terms cancel (as a polynomial in $s$). Therefore the last term in the expansion of $\Im \Phi$ in \eqref{eq:job-almost-done} can be absorbed into the leading term, and we get
            \begin{equation}\label{eq:job-done}
                \Im \Phi(\Gamma(s, x)) \geq Cs (|x|^2 + |\Im z_{\mathrm{crit}}|^2),\quad s \in [0, 1], x \in \supp(b|_{\mathbb{R}^{n-1}}).
            \end{equation}

            We may therefore using \eqref{eq:job-done} and almost analyticity bound 
            \begin{align*}
                |\partial_{\Bar{z}} (e^{\I \lambda \Phi} b)(\Gamma(s, x))| &= |e^{\I \lambda \Phi}(\I \lambda \partial_{\Bar{z}}\Phi b + \partial_{\Bar{z}} b)(\Gamma(s, x))|\\ 
                &\leq \mc{O}_N(1)e^{-C\lambda s |\Im z_{\mathrm{crit}}|^2} (s^2|\Im z_{\mathrm{crit}}|^2 \lambda)^N \lambda^{1 - N} = \mc{O}(\lambda^{1 - N}),
            \end{align*}
            for any $N$, where in the final estimate we used also that $s^2 \leq s$, thus proving the claim that $S = \mc{O}(\lambda^{-\infty})$. Thus \eqref{eq:stokes-1} gives
            \begin{align*}
                \lambda^{-\frac{n-1}{2}}I_1 + \mc{O}(\lambda^{-\infty}) &= \int_{z_{\mathrm{crit}} + \mathbb{R}^{n - 1}} e^{\I \lambda \Phi} b\, dz\\ 
                &= e^{\I \lambda \Phi(z_{\mathrm{crit}})} \int_{\mathbb{R}^{n - 1}} e^{\I \lambda (\frac{1}{2}\partial_z^2 \Phi(z_{\mathrm{crit}})(x, x) + R(x + z_{\mathrm{crit}}))} b(x + z_{\mathrm{crit}})\, dx.
            \end{align*}
            Let us now estimate the remainder term. Namely, we have
            \begin{align*}
                &|e^{\I \lambda R(x + z_{\mathrm{crit}})} - 1| = \left|\int_{0}^1 e^{\I \lambda s R(x + z_{\mathrm{crit}})} \I \lambda R(x + z_{\mathrm{crit}})\, ds\right|\\
                &\leq e^{\lambda \mc{O}(|x|^3 + |\Im z_{\mathrm{crit}}|^3)} \lambda \mc{O}(|x|^3 + |\Im z_{\mathrm{crit}}|^3),
            \end{align*}
            where we used \eqref{eq:taylor-easy} in the last bound. Therefore,
            \begin{align*}
                &\lambda^{\frac{n-1}{2}}\left|e^{\I \lambda \Phi(z_{\mathrm{crit}})} \int_{\mathbb{R}^{n - 1}} e^{\I \lambda (\frac{1}{2}\partial_z^2 \Phi(z_{\mathrm{crit}})(x, x))} (e^{\I \lambda R(x + z_{\mathrm{crit}})} -1) b(x + z_{\mathrm{crit}})\, dx\right|\\
                &\leq \lambda^{\frac{n-1}{2}}\int_{\mathbb{R}^{n-1}} e^{-C\lambda(|x|^2 + |\Im z_{\mathrm{crit}}|^2)} e^{\lambda \mc{O}(|x|^3 + |\Im z_{\mathrm{crit}}|^3)} \lambda \mc{O}(|x|^3 + |\Im z_{\mathrm{crit}}|^3)\, dx\\
                &\leq \lambda^{-\frac{1}{2}} \lambda^{\frac{n-1}{2}}\int_{\mathbb{R}^{n-1}} e^{-C\lambda(|x|^2 + |\Im z_{\mathrm{crit}}|^2)} \mc{O}([\lambda(|x|^2 + |\Im z_{\mathrm{crit}}|^2)]^{\frac{3}{2}})\, dx = \mc{O}(\lambda^{-\frac{1}{2}}),
            \end{align*}
            where in the first bound we applied \eqref{eq:crit-point-imaginary-positive} and used that $[\Im \partial_z^2 \Phi] (z_{\mathrm{crit}}) > 0$, and in the last line we absorbed the cubic terms, and estimated the integral in the usual way (see e.g. \eqref{eq:usual-bound}). Therefore,
            \begin{align*}
                I_1 &= \mc{O}(\lambda^{-\frac{1}{2}}) + e^{\I \lambda \Phi(z_{\mathrm{crit}})} \lambda^{\frac{n-1}{2}}\int_{\mathbb{R}^{n - 1}} e^{\frac{1}{2}\I \lambda \partial_z^2 \Phi(z_{\mathrm{crit}})(x, x)} b(x + z_{\mathrm{crit}})\, dx\\
                &= \mc{O}(\lambda^{-\frac{1}{2}}) + e^{\I \lambda \Phi(z_{\mathrm{crit}})} \frac{(2\pi)^{\frac{n-1}{2}}}{\sqrt{\det(-\I \partial_z^2 \Phi(z_{\mathrm{crit}})})} b(z_{\mathrm{crit}})\\
                &= \mc{O}(\lambda^{-\frac{1}{2}}) + e^{\I \lambda \Phi(z_{\mathrm{crit}})} \frac{(2\pi)^{\frac{n-1}{2}}}{\sqrt{\det(-\I \nabla^2_{\partial M} \Phi(x_1)})} a_0^{(1), +}(x_1) \partial_\nu(\Phi_1^+ + \overline{\Phi}_2^+)(x_1),
            \end{align*}
            where in the second line we used the stationary phase expansion for a quadratic phase, while in the third line we used Taylor's expansion and \eqref{eq:critical-point}. Note that $\partial_x^2 \Phi(0) = \partial_z^2 \Phi(0)$ is identified with $\nabla^2_{\partial M} \Phi(x_1)$ since we are in normal coordinates at $x_1$. This completes the proof of the expansion.

            To show the uniqueness part up to negligible error of the critical point, observe first that
            \[
                \Phi(z) - \Phi'(z) = \mc{O}(|\Im z|^\infty),
            \]
            since both $\Phi$ and $\Phi'$ are almost analytic and they agree on $\mathbb{R}^{n-1}$. Let $F$ and $F'$ denote the right hand sides in the implicit equation \eqref{eq:implicit-equation}. Then 
            \[
                |F(z) - F(z')| \leq \frac{1}{2}|z - z'|,\quad |F'(z) - F'(z')| \leq \frac{1}{2}|z - z'|,
            \]
            for small $z, z'$, since the gradients of $F$ and $F'$ are zero at the origin. We also have
            \[
                F(z) - F'(z) = \mc{O}(|\Im z|^\infty), 
            \]
            by the corresponding property of $\Phi$ and $\Phi'$. Notice that by definition $F(z_{\mathrm{crit}}) = z_{\mathrm{crit}}$ and $F'(z'_{\mathrm{crit}}) = z'_{\mathrm{crit}}$, so
            \begin{align}\label{eq:uniqueness-crit-pt}
            \begin{split}
                |z_{\mathrm{crit}} - z'_{\mathrm{crit}}| &\leq |F(z_{\mathrm{crit}}) - F(z_{\mathrm{crit}}')| + |F(z_{\mathrm{crit}}') - F'(z_{\mathrm{crit}}')|\\
                &\leq \frac{1}{2}|z_{\mathrm{crit}} - z_{\mathrm{crit}}'| + \mc{O}(\lambda^{-\infty}),
            \end{split}
            \end{align}
            where in the last inequality we used \eqref{eq:critical-point}. This implies the claim, and completes the proof of the lemma.            
        \end{proof}
        \begin{remark}\rm
            Lemma \ref{lemma:stationary-phase} admits a generalisation to an infinite expansion into negative powers of $\lambda$; the proof is by iteration of the given argument, see also \cite[Proof of Theorem 2.3]{Melin-Sjostrand-75}.
        \end{remark}

        We need a version of the preceding lemma where the exponent in the leading term asymptotics depends nicely on the parameter $(x, v) \in \mc{G}_{\varepsilon}$. 

        \begin{lemma}\label{lemma:smoothness-parameter}
            The function $\Phi(z_{\mathrm{crit}})$ appearing as the phase of the asymptotic expansion in Lemma \ref{lemma:stationary-phase} can be chosen to depend smoothly on $(x, v) \in \mc{G}_{\varepsilon}$. Moreover, we can choose this function $\Psi = \Psi(x, v)$ so that $\Psi \equiv 0$ near the boundary of $\mc{G}_\varepsilon$, so it extends smoothly by zero to $\partial_-SM$. Also, it holds that
            \[
                \Psi(x, v) = \ell_{g_1}(x, v) - \ell_{g_2}(x, v) + \mc{O}(\lambda^{-\frac{1}{2}}),\quad (x, v) \in \partial_-SM,
            \]
            as well as
            \[
                |\Im \Psi| = \mc{O}(\lambda^{-1}).
            \]
        \end{lemma}
        \begin{proof} We use the notation from the proof of Lemma \ref{lemma:stationary-phase}. Note that Lemma \ref{lemma:stationary-phase} extends to the case where the phase and the amplitude depend on an additional parameter (in a perturbative, i.e. local sense); in this case the parameter is $(x, v) \in \mc{G}_{\varepsilon}$. Indeed, this holds directly by inspection of the proof above, similarly to \cite[Proof of Theorem 2.3]{Melin-Sjostrand-75}. We may thus cover $\mc{G}_{\varepsilon}$ by a sufficiently small finite open cover $(U_{i})_{i \in I}$, so that the exit point $x_+(x, v) \in \partial M$ for $(x, v) \in U_i$ belongs to a fixed boundary chart. On each $U_i$ we know that that the critical point $z_{\mathrm{crit}} = z_{\mathrm{crit}}(x, v) \in \mathbb{C}^{n - 1}$ of the almost analytic extension of $\Phi = \Phi_{(x, v)}$,  depends smoothly on $(x, v)$. Thus for $(x, v) \in U_i$ we may write $\Psi_i(x, v) := \Phi_{(x, v)}(z_{\mathrm{crit}}(x, v))$ again depending smoothly on $(x, v)$. Let $(\chi_i)_{i \in I}$ be a partition of unity subordinate to this cover, and set
            \[
                \Psi(x, v) := \sum_{i \in I} \chi_i(x, v) \Psi_i(x, v).
            \]
        We note here that on overlaps of the charts corresponding to $x_+(U_i)$, there is a transition function $\psi: U \to U'$ so that without loss of generality $\psi(0) = 0$. Write $\Phi_U$ and $\Phi_{U'}$ for the almost analytic extensions of $\Phi$ in these charts, and also $z_{\mathrm{crit}, U}$ and $z_{\mathrm{crit}, U'}$ for the corresponding critical points. Extend $\psi$ to an almost analytic map denoted by the same letter; by the inverse function theorem, $\psi$ is locally a diffeomorphism near $0 \in \mathbb{C}^{n - 1}$. By the chain rule for Wirtinger's derivatives, $\psi_* \Phi_U$ is an almost analytic function, and $\partial_z \psi_* \Phi_U(\psi(z_{\mathrm{crit}, U})) = \mc{O}(\lambda^{-\infty})$. It follows similarly to \eqref{eq:uniqueness-crit-pt} that $\psi(z_{\mathrm{crit}, U})$ is $\mc{O}(\lambda^{-\infty})$ close to the actual critical point of $\psi_*\Phi_U$, and so by uniqueness of critical points proved in Lemma \ref{lemma:stationary-phase}, we have that
        \[
            \psi(z_{\mathrm{crit}, U}) - z_{\mathrm{crit}, U'} = \mc{O}(\lambda^{-\infty}).
        \]
        Therefore, we get that
        \begin{align*}
            &|\Phi_U(z_{\mathrm{crit}, U}) - \Phi_{U'}(z_{\mathrm{crit}, U'})| = |\psi_*\Phi_{U}( \psi(z_{\mathrm{crit}, U})) - \Phi_{U'}(z_{\mathrm{crit}, U'})|\\
            &\leq |\psi_*\Phi_{U}( \psi(z_{\mathrm{crit}, U})) - \psi_*\Phi_{U}(z_{\mathrm{crit}, U'})| + |\psi_*\Phi_{U}( z_{\mathrm{crit}, U'}) - \Phi_{U'}(z_{\mathrm{crit}, U'})|\\ 
            &\leq  \mc{O}(1) |\psi(z_{\mathrm{crit}, U}) - z_{\mathrm{crit}, U'}| + \mc{O}( |\Im z_{\mathrm{crit}, U'}|^\infty)\\ 
            &= \mc{O}(\lambda^{-\infty}),
        \end{align*}
        where in the third line we used that $\psi_* \Phi_U$ and $\Phi_{U'}$ agree on $\mathbb{R}^{n - 1} \cap U'$ and that they are almost analytic extensions, and in the last line we used \eqref{eq:critical-point}. Note that there are only finitely many fixed charts involved; this shows that $\Phi(z_{\mathrm{crit}})$ is in fact invariantly defined modulo a uniform error $\mc{O}(\lambda^{-\infty})$. In other words on overlaps of $U_i$ and $U_j$, we have shown that $\Psi_i(x, v) = \Psi_j(x, v) + \mc{O}(\lambda^{-\infty})$, which shows that in fact $\Psi(x, v)$ is the correct phase in the leading term expansion.

            We now check the vanishing near the boundary of $\mc{G}_\varepsilon$. Indeed, by construction, for $(x, v)$ with $-g_{x}(v, \nu(x))$ near $\varepsilon$, we know that the corresponding geodesics generated by $(x, v)$ agree, and so $\Phi_1^+ = \Phi_2^+$ near the exit point, and also $z_{\mathrm{crit}} = 0$ is the unique critical point of $\Phi$ (there is also no need for an almost analytic extensions). Thus we may choose that $\Psi(x, v) = 0$ near such directions. 
            
            Let us prove the expansion as the difference of the travel times. Indeed, this holds by Taylor's expansion of $\Phi(z_{\mathrm{crit}})$, because in a coordinate chart
            \[
                \Phi(z_{\mathrm{crit}}) = \Phi(0) + \mc{O}(\lambda^{-\frac{1}{2}}) = \ell_{g_1}(x, v) - \ell_{g_2}(x, v) + \mc{O}(\lambda^{-\frac{1}{2}}), 
            \]
            where we used the fact that the first derivative of $\Phi$ is uniformly bounded and \eqref{eq:critical-point}, and that $\Phi_i^+(x_i) = \ell_{g_i}(x_0, v_0)$ for $i = 1, 2$, as well as Step 2.
            
            Finally, we prove the bound on the imaginary part. We estimate, again by using Taylor's expansion that
            \[
                \Im\Phi(z_{\mathrm{crit}}) = \Im \Phi_2^+(0) + \Im(\partial_x \Phi(0)z_{\mathrm{crit}}) + \mc{O}(\lambda^{-1}) = \mc{O}(\lambda^{-1}),
            \]
            where we used that $\Im \Phi_1^+(0) = 0$, while in the estimate we applied Taylor's expansion together with Proposition \ref{prop:phase-proposition} and Step 2 (i.e. $d_{\partial M}(x_1, x_2) = \mc{O}(\lambda^{-\frac{1}{2}})$), and we used \eqref{eq:critical-point} and \eqref{eq:small-gradient} to treat the other term. This completes the proof.
        \end{proof}
        
        We now go back to showing that the distance functions agree to some negative power in $\lambda$. We argue as in the proof of Theorem \ref{thm_main_1}, Step 8, and for brevity we omit some steps. We apply Lemma \ref{lemma:stationary-phase}, and we write $\Psi(x, v)$ for the phase constructed in Lemma \ref{lemma:smoothness-parameter} depending smoothly on $(x, v)$. Therefore,
        \begin{align*}
            &\mc{O}(\lambda^{-\frac{1}{2}}) + \mc{O}(h\|\Lambda_{g_1}^\lambda - \Lambda_{g_2}^\lambda\|_{H^1_h \to L^2}) = \frac{(2\pi)^{\frac{n-1}{2}}}{\sqrt{\det(\nabla^2_{\partial M} 2\Im \Phi_2^+(x_2))}} 2 \Re(\partial_\nu \Phi_2^+(x_2)) a_0^{(2), +}(x_2)\\
            &- \frac{(2\pi)^{\frac{n-1}{2}}}{\sqrt{\det(-\I\nabla^2_{\partial M}(\Phi_1^+ - \overline{\Phi}_2^+)(x_1))}}e^{\I \lambda \Psi(x_0, v_0)} \partial_\nu(\Phi_1^+(x_1) + \overline{\Phi}_2^+(x_2)) a_0^{(1), +}(x_1).
        \end{align*}
        Note that
        \begin{align*}
            \partial_{\nu}\Phi_1^+(x_1) &= \langle{\partial_\nu, \dot{\gamma}_1}\rangle = \sqrt{1 - |\Pi \dot{\gamma}_1|^2} = \sqrt{1 - |\Pi \dot{\gamma}_2|^2 + \mc{O}(\lambda^{-\frac{1}{2}})} = \sqrt{1 - |\Pi \dot{\gamma}_2|^2} + \mc{O}(\lambda^{-\frac{1}{2}})\\ 
            &= \langle{\partial_\nu, \dot{\gamma}_2}\rangle + \mc{O}(\lambda^{-\frac{1}{2}}) = \partial_\nu \Phi_2^+(x_2) +  \mc{O}(\lambda^{-\frac{1}{2}}),
        \end{align*}
        where we used the preceding step in the third equality. Therefore,
        \[
            0 < 2\varepsilon_+ < 2\Re(\partial_\nu \Phi_2^+(x_2)) = \partial_\nu(\Phi_1^+(x_1) + \overline{\Phi}_2^+(x_2)) + \mc{O}(\lambda^{-\frac{1}{2}}),
        \]
        and we thus get 
        \begin{align*}
            &\mc{O}(\lambda^{-\frac{1}{2}}) + \mc{O}(h\|\Lambda_{g_1}^\lambda - \Lambda_{g_2}^\lambda\|_{H^1_h \to L^2}) = \frac{1}{\sqrt{\det(\nabla^2_{\partial M} 2\Im \Phi_2^+(x_2))}} a_0^{(2), +}(x_2)\\
            &- \frac{1}{\sqrt{\det(-\I \nabla^2_{\partial M}(\Phi_1^+ - \overline{\Phi}_2^+)(x_1))}}e^{\I \lambda \Psi(x_0, v_0)} a_0^{(1), +}(x_1).
        \end{align*}
        The determinants and the amplitudes are bounded from below and above uniformly by positive constants and we may thus write
        \[
            F(x, v) e^{\I \lambda \Psi(x, v)} - 1 = \mc{O}(\lambda^{-\frac{1}{2}}) + \mc{O}(h\|\Lambda_{g_1}^\lambda - \Lambda_{g_2}^\lambda\|_{H^1_h \to L^2}),\quad (x, v) \in \mc{G}_\varepsilon,
        \]
        where $F$ is implicitly defined. Note that $F$ depends smoothly on $(x, v)$, that $C^{-1} < |F| < C$ for some uniform $C > 1$, and that $F \equiv 1$ near the boundary of $\mc{G}_\varepsilon$. Indeed, for $(x, v)$ near the boundary of $\mc{G}_\varepsilon$, the geodesics for the two metrics agree, and so do the amplitudes and phases obtained by solving ODEs along the geodesics; thus $F$ extends smoothly by $1$ to $\partial_-SM$. Note also that $\Psi$ extends smoothly by zero from $\mc{G}_\varepsilon$ to $\partial_-SM$ by Lemma \ref{lemma:smoothness-parameter}, so
        \begin{equation}
            F(x, v) e^{-\lambda \Im \Psi(x, v)} e^{\I \lambda \Re \Psi(x, v)} - 1 = \mc{O}(\lambda^{-\frac{1}{2}}) + \mc{O}(h\|\Lambda_{g_1}^\lambda - \Lambda_{g_2}^\lambda\|_{H^1_h \to L^2}),\quad (x, v) \in \partial_-SM.
        \end{equation}
        It follows that 
        \[
            |F| e^{-\lambda \Im \Psi} = 1 + \mc{O}(\lambda^{-\frac{1}{2}}) + \mc{O}(h\|\Lambda_{g_1}^\lambda - \Lambda_{g_2}^\lambda\|_{H^1_h \to L^2}),
        \]
        and so by the triangle inequality that
        \[
            e^{\I \left(\lambda \Re \Psi + \arg\left(\frac{F}{|F|}\right)\right)} - 1 = \mc{O}(\lambda^{-\frac{1}{2}}) + \mc{O}(h\|\Lambda_{g_1}^\lambda - \Lambda_{g_2}^\lambda\|_{H^1_h \to L^2}),
        \]
        where $\arg(\bullet)$ denotes the argument of $\bullet$. 
        
        As before, let $z \in \partial_-SM$ and connect $z$ by a path $\Gamma$ in the same fibre of $\partial_-SM \to \partial M$ to a point $z_0 \not \in \mc{G}_\varepsilon$. Then, as in the proof of Theorem \ref{thm_main_1}, Step 8, we get
        \[
            -\I \int_\Gamma \frac{d\left(\frac{F}{|F|}\right)}{\frac{F}{|F|}} + \lambda \Re \Psi(z) =  \mc{O}(\lambda^{-\frac{1}{2}}) + \mc{O}(h\|\Lambda_{g_1}^\lambda - \Lambda_{g_2}^\lambda\|_{H^1_h \to L^2}).
        \]
        Using that $dF$ is uniformly bounded, and expanding $\Re \Psi$ using Lemma \ref{lemma:smoothness-parameter}, as well as that $\mc{O}(h\|\Lambda_{g_1}^\lambda - \Lambda_{g_2}^\lambda\|_{H^1_h \to L^2}) < \varepsilon_0$, we conclude that
        \[
            \ell_{g_1}(z) - \ell_{g_2}(z) = \mc{O}(\lambda^{-\frac{1}{2}}).
        \]
        Since $z$ was arbitrary, this completes the proof of the theorem.
    \end{proof}

    \section{The two-solution approach}\label{section:two-solutions}

    We now give a second proof of the theorems above based on using only two solutions. The crucial difference with the first proof is that we do not use the third solution $u$ which ``goes backwards", and there is no phenomenon of backscattering.

\subsection{Uniqueness}
    \begin{proof}[Proof of Theorem \ref{thm_main_1} (simplified)]
    We divide the proof into steps. We assume the boundary determination part and the geometric setup from Step 1 and the construction of $u_1$ and $u_2$, and $r_1$ and $r_2$ from Step 2 of the three-solution proof of Theorem \ref{thm_main_1} above. We emphasise that as a consequence, $(M, g_1)$ and $(M, g_2)$ are actually identified with suitable extensions where $g_1 = g_2$ in the exterior collar (see Step 1 for details).
    \medskip
    
        \emph{Step 1: integral identity and principal terms.} Using the second integral identity of Proposition \ref{prop:extension-2}, we get similarly to Step 3 of the above proof that (by using the estimates on remainder terms)
        \[
            \Im \langle{\partial_\nu(u_1 - u_2), u_1 - u_2}\rangle_{L^2(\partial M)} = \mc{O}(\lambda^{1 - K}).
        \]
        Using that $u_1^-$ and $u_2^-$ agree, we deduce that
        \begin{equation}\label{eq:new-working-identity}
            \Im \langle{\partial_\nu(u_1^+ - u_2^+), u_1^+ - u_2^+}\rangle_{L^2(\partial M)} = \mc{O}(\lambda^{1 - K}).
        \end{equation}
        Let us now determine the principal term in \eqref{eq:new-working-identity}, similarly to Step 4 of the previous proof, by replacing $u_i^+ = \lambda^{\frac{n-1}{4}} e^{\I \lambda \Phi_i^+} a^{(i), +} \chi_i^+$, for $i = 1, 2$. Using \eqref{eq:normal-derivative-expansion}, and Proposition \ref{prop:boundary-estimate+geometry} to treat the remainder terms, we get (for $K \geq 1$)
        \begin{align*}
            \lambda^{\frac{n-1}{2}} \Im \int_{\partial M} \I (e^{\I \lambda \Phi_1^+} a_0^{(1), +} \partial_\nu \Phi_1^+ \chi_1^+& - e^{\I\lambda \Phi_2^+} a_0^{(2), +} \partial_\nu \Phi_2^+ \chi_2^+)\\
            &\cdot(\overline{a_0^{(1), +}} e^{-\I\lambda \overline{\Phi_1^+}} \chi_1^+ - \overline{a_0^{(2), +}} e^{-\I\lambda \overline{\Phi_2^+}} \chi_2^+)\, d\vol = \mc{O}(\lambda^{-1}).
        \end{align*}
        A straightforward computation reveals that
        \begin{align*}
            &\Re \int_{\partial M} \left[e^{ \I \lambda (\Phi_1^+ - \overline{\Phi_2^+})} a_0^{(1), +} \overline{a_0^{(2), +}} \partial_\nu \Phi_1^+ + e^{ \I \lambda (\Phi_2^+ - \overline{\Phi_1^+})} a_0^{(2), +} \overline{a_0^{(1), +}} \partial_\nu \Phi_2^+\right]\chi_1^+ \chi_2^+\,d\vol\\
            &= \Re \int_{\partial M} e^{\I \lambda (\Phi_1^+ - \overline{\Phi_2^+})} a_0^{(1), +} \overline{a_0^{(2), +}} \partial_\nu(\Phi_1^+ + \overline{\Phi_2^+})\chi_1^+ \chi_2^+\,d\vol,
        \end{align*}
        and so expanding the preceding equality we get
        \begin{equation}\label{eq:new-main-working-identity}
        \begin{split}
            &J := \lambda^{\frac{n-1}{2}}\int_{\partial M} \left[e^{-2\lambda \Im \Phi_1^+} |a_0^{(1), +}|^2 \Re(\partial_\nu \Phi_1^+) (\chi_1^+)^2 + e^{-2\lambda \Im \Phi_2^+} |a_0^{(2), +}|^2 \Re(\partial_\nu \Phi_2^+) (\chi_2^+)^2\right]\,d\vol\\
            &= \underbrace{\lambda^{\frac{n-1}{2}} \Re \int_{\partial M} e^{\I \lambda (\Phi_1^+ - \overline{\Phi_2^+})} a_0^{(1), +} \overline{a_0^{(2), +}} \partial_\nu(\Phi_1^+ + \overline{\Phi_2^+}) \chi_1^+ \chi_2^+ \,d\vol}_{J_{\mathrm{mix}}:=} + \mc{O}(\lambda^{-1}).
        \end{split}
        \end{equation}
    \medskip

    \emph{Step 2: equality of exit points.} Recall that $\Lambda$ was defined in \eqref{eq:Lambda1} as the supremum of distances of exit points; we assume that $\Lambda > 0$. By the arguments in Step 6 of the previous proof, i.e. the stationary phase lemma with the uniform remainder estimate as in \eqref{eq:remainder-stat-phase}, using uniformity of amplitudes, and the positive lower bound on $\Re(\partial_\nu \Phi_i^+)$ on $\supp \chi_i^+$ for $i = 1, 2$, we get that the left hand side of \eqref{eq:new-main-working-identity} is bounded from below by a uniform constant $C > 0$
    \begin{equation}\label{eq:lower-bound-XXX}
        J \geq C > 0,
    \end{equation}
    for $\lambda$ uniformly large enough. Alternatively, one could avoid using the stationary phase lemma in the proof of \eqref{eq:lower-bound-XXX}. Indeed, recall that $(x_0, v_0) \in \mc{G}_\varepsilon$ and that for $i = 1, 2$, we write $S_{g_i}(x_0, v_0) = (x_i, v_i)$. In the same way as in \eqref{eq:uniform-lower-bound-exit} (using Theorem \ref{thm:gbu-AC22} and Proposition \ref{prop:boundary-estimate+geometry}), we may assume that $\Re(\partial_\nu \Phi_i^+) > \varepsilon_+/2 > 0$ on $\partial M \cap \supp \chi_i^+$. Similarly, we get that $|a_0^{(i), +}| \geq c > 0$ for some uniform $c > 0$, over $\partial M \cap \supp \chi_i^+$. By Proposition \ref{prop:phase-proposition}, $\Im \Phi_i^+(x) \leq C_0 d_{\partial M}(x, x_i)^2$ on $\supp \chi_i^+$ for some uniform $C_0 > 0$, and so
    \[
        e^{-2\lambda \Im \Phi_i^+} \geq e^{-2C_0},\quad d_{\partial M}(x, x_i) \leq \lambda^{-\frac{1}{2}}.
    \]
    We may take $\lambda$ uniformly large so that $\chi_i^+(x) = 1$ when $d_{\partial M}(x, x_i) \leq \lambda^{-\frac{1}{2}}$. Therefore,
    \[
        J \geq \frac{c^2 \varepsilon_+}{2} e^{-2C_0} \lambda^{\frac{n-1}{2}} \sum_{i = 1}^2 \vol_{\partial M}(B_{\partial M}(x_i, \lambda^{-\frac{1}{2}})) \geq C > 0,
    \]
    where we used that the volume of the boundary ball $B_{\partial M}(x_i, \lambda^{-\frac{1}{2}})$ of radius $\lambda^{-\frac{1}{2}}$ is uniformly proportional to $\lambda^{-\frac{n-1}{2}}$. This completes the alternative proof of \eqref{eq:lower-bound-XXX}.
    
    Arguing exactly the same as in \eqref{eq:I_2-upper-bound}, we get that 
    \[
        |J_{\mathrm{mix}}| \leq \mc{O}(1) e^{-C_{\mathrm{mix}} \lambda d_{\partial M}(x_1, x_2)^2},
    \]
    for some uniform $C_{\mathrm{mix}} > 0$. Combining the previous bounds with \eqref{eq:new-main-working-identity} gives
    \[
        0 < C < \mc{O}(1) e^{-C_{\mathrm{mix}} \lambda d_{\partial M}(x_1, x_2)^2} + \mc{O}(\lambda^{-1}).
    \]
    Taking $\lambda$ large enough depending on $\Lambda$ gives a contradiction to the assumption that $\Lambda > 0$. We conclude that the exit points agree.
    \medskip

    \emph{Step 3: equality of exit directions.} Recall that the quantity $\Lambda$ defined in \eqref{eq:exit-directions} is the supremum of differences between exit vectors projected to $T\partial M$. Assuming $\Lambda > 0$ and arguing exactly in the same way as in Step 7 of the former proof, integration by parts shows that
    \[
        J_{\mathrm{mix}} = \mc{O}(\lambda^{-1}).
    \]
    This contradicts \eqref{eq:new-main-working-identity} in the same way as in the previous step for $\lambda$ uniformly large enough. This shows that the scattering data agree, i.e. $S_{g_1} = S_{g_2}$.
    \medskip

    \emph{Step 4: equality of travel times.} This follows directly from Proposition \ref{prop:known?}. Alternatively, by the stationary phase lemma, we may expand the integrals in \eqref{eq:new-main-working-identity}, as in Step 8 of the previous proof. Arguing exactly in the same way, we get
    \[
        F(x, v) - \Re(e^{\I\lambda(\ell_{g_1} - \ell_{g_2})(x, v)} F_{\mathrm{mix}}(x, v)) = \mc{O}(\lambda^{-1}),\quad (x, v) \in \mc{G}_\varepsilon,
    \]
    where, writing $(x_+, v_+) = S_{g_1}(x, v) = S_{g_2}(x, v)$ to denote the exit point and direction of $\gamma_{x, v}$, we have
    \begin{align*}
        \frac{F(x, v)}{\pi^{\frac{n-1}{2}}} &= \frac{|a_0^{(2), +}|^2(x_+)\partial_\nu \Phi_2^+(x_+)}{\sqrt{\det \nabla^2_{\partial M} \Im \Phi_2^+(x_+)}} + \frac{|a_0^{(1), +}|^2(x_+)\partial_\nu \Phi_1^+(x_+)}{\sqrt{\det \nabla^2_{\partial M} \Im \Phi_1^+(x_+)}},\\
        F_{\mathrm{mix}}(x, v) &= \left[\det\left(\frac{\nabla^2_{\partial M} \Phi(x_{+})}{2\pi \I}\right)\right]^{-\frac{1}{2}} a_0^{(1), +}(x_+) \overline{a_0^{(2), +}}(x_+) \partial_{\nu}(\Phi_1^+ + \overline{\Phi}_2^+)(x_{+}).
    \end{align*}
    Note that $F(x, v)$ and $F_{\mathrm{mix}}(x, v)$ depend smoothly on $(x,v) \in \mc{G}_{\varepsilon}$, $F$ is real-valued and uniformly bounded from below by a positive constant. Therefore, writing $H := \frac{F_{\mathrm{mix}}}{F}$, we get
    \[
        1 - \Re(e^{\I\lambda(\ell_{g_1} - \ell_{g_2})(x, v)} H(x, v)) = \mc{O}(\lambda^{-1}),\quad (x, v) \in \mc{G}_\varepsilon.
    \]
    Note that $H$ does not depend on $\lambda$, and that the previous equality is actually valid on the whole $\partial_-SM$. Indeed, for $(x, v) \not\in \mc{G}_\varepsilon$, the metrics agree with $g_0$, and so the corresponding geodesics, amplitudes and phases are the same (as they are obtained by solving an ODE along the geodesic, with the same initial data); thus $\ell_{g_1} \equiv \ell_{g_2}$ and $H \equiv 1$ outside $\mc{G}_\varepsilon$. Therefore
    \begin{equation}\label{eq:travel-time}
                1 - \Re(e^{\I\lambda(\ell_{g_1} - \ell_{g_2})(x, v)} H(x, v)) = \mc{O}(\lambda^{-1}),\quad (x, v) \in \partial_-SM.
    \end{equation}
    Also, similarly to before, since the integrands are bounded uniformly from above and below in $\mc{G}_\varepsilon$, $H$ is bounded uniformly from above and below, i.e. there exists a uniform $C > 1$ such that $C^{-1} < |H| < C$ (on $\partial_-SM$). Moreover, we must have $|H| \geq 1$, since otherwise, taking $\lambda$ large enough would contradict \eqref{eq:travel-time}.

    Let $z \in \partial_-SM$, and choose a path $\Gamma: [0, 1] \to \partial_-SM$ in the fibre of $z$, of uniformly bounded length, such that $\Gamma(0) = z_0 \not \in \mc{G}_\varepsilon$ (so that $H(z_0) = 1$) and $\Gamma(1) = z$. We may write $H \circ \Gamma = |H\circ \Gamma| e^{\I f_\Gamma}$, similarly to the discussion in Step 8 of the three-solution proof of Theorem \ref{thm_main_1}, where $f_\Gamma(0) = 0$ and $f_\Gamma$ is real-valued. By \eqref{eq:travel-time}, we get along $\Gamma$ that
    \[
        1 - |H| \cos(\lambda(\ell_{g_1} - \ell_{g_2}) + f_\Gamma) = \mc{O}(\lambda^{-1}),
    \]
    Using the elementary bound that for some universal $C > 0$ we have, for $a, b \in \mathbb{R}, \theta > 0$ that
    \[
        |\cos a - \cos b| \leq \theta \implies a + b \in [-C \sqrt{\theta}, C \sqrt{\theta}] + 2\pi \mathbb{Z}\,\,\, \mathrm{or}\,\,\, a - b \in [-C \sqrt{\theta}, C \sqrt{\theta}] + 2\pi \mathbb{Z},
    \]
    we conclude that if we set (note that $\arccos$ takes values in $[0, \pi]$)
    \[
        G_\Gamma^\pm(t) := \lambda(\ell_{g_1} - \ell_{g_2}) (\Gamma(t)) + f_\Gamma(t) \pm \arccos(|H|^{-1}(\Gamma(t))),
    \]
    then for all $t \in [0, 1]$
    \begin{equation}\label{eq:or}
        G_\Gamma^+(t) \in [-C\lambda^{-\frac{1}{2}}, C\lambda^{-\frac{1}{2}}] + 2\pi \mathbb{Z}\quad \mathrm{or}\quad G_\Gamma^-(t) \in [-C\lambda^{-\frac{1}{2}}, C\lambda^{-\frac{1}{2}}] + 2\pi \mathbb{Z}.
    \end{equation}
    Because of this branching we need an additional elementary argument compare to the three-solution proof to finish the argument. We simply observe that
    \[
        2\arccos(|H|^{-1}) \in [0, \pi],
    \]
    because $|H| \geq 1$. (In fact since $|H|^{-1}$ is uniformly bounded, $2\arccos(|H|^{-1}) \in [0, \pi - \delta_H]$ for some uniform $\delta_H > 0$; we will however not need this.) Take $\lambda$ uniformly large enough so that the intervals in the right hand sides of \eqref{eq:or} are disjoint, and the gaps have size $> \pi$. By continuity as in the proof above, we conclude that if $G_\Gamma^+$ or $G_\Gamma^-$ take value in more than one connected component of $[-C\lambda^{-\frac{1}{2}}, C\lambda^{-\frac{1}{2}}] + 2\pi \mathbb{Z}$, then there exists $t\in [0, 1]$ such that both $G_\Gamma^+(t)$ and $G_\Gamma^-(t)$ do not belong to $[-C\lambda^{-\frac{1}{2}}, C\lambda^{-\frac{1}{2}}] + 2\pi \mathbb{Z}$ (they take value in one of the gaps), contradicting \eqref{eq:or}. Since $G_\Gamma^\pm(0) = 0$ by construction, we conclude that
    \[
        \lambda(\ell_{g_1} - \ell_{g_2}) + f_\Gamma \pm \arccos(|H|^{-1}) \in [-C\lambda^{-\frac{1}{2}}, C\lambda^{-\frac{1}{2}}],
    \]
    where for simplicity, by convention here we mean that this holds for at least one choice of the $\pm$ sign. More invariantly, we get that 
    \[
        -\I \int_\Gamma \frac{d\left(\frac{H}{|H|}\right)}{\frac{H}{|H|}} + \lambda(\ell_{g_1} - \ell_{g_2})(z) \pm \arccos(|H|^{-1}(z)) =\mc{O}(\lambda^{-\frac{1}{2}}).
    \]
    Since $H$, $dH$, and the length of $\Gamma$ are uniformly bounded by construction, and $|H|$ is uniformly bounded from below by a positive constant, and $z \in \partial_-SM$ is arbitrary, we conclude that $\ell_{g_1} - \ell_{g_2} = \mc{O}(\lambda^{-1})$. Taking $\lambda > \lambda_0$ uniformly large then gives that in fact $\ell_{g_1} \equiv \ell_{g_2}$. This completes the proof.
    \end{proof}

    \subsection{Quantification of uniqueness via the two-solution approach}

    \begin{proof}[Proof of Theorem \ref{theorem:quantification} (simplified)] 
        We divide the proof into several steps. We will use the strategy from the two-solutions proof of Theorem \ref{thm_main_1}. As in the first proof of Theorem \ref{theorem:quantification}, we note that all the Gaussian beams come with estimates uniform in $\mc{M}_{N, B, \delta, T}(g_0)$; we will use the notation from this proof.
        \medskip

        \emph{Step 1: the principal terms.} Using the second integral identity of Proposition \ref{prop:extension-2}, we get
        \[
            \Im \langle{\partial_\nu(u_1 + r_1 - u_2 - r_2), u_1 + r_1 - u_2 - r_2}\rangle_{L^2(\partial M)} = \Im \langle{(\Lambda_{g_1}^\lambda - \Lambda_{g_2}^\lambda)(u_2 + r_2), u_1 + r_1}\rangle_{L^2(\partial M)}.
        \]
        Using the estimates on remainder terms, see \eqref{eq:remainder-estimate} for $K = 1$, and estimating the term involving the DN map by Cauchy-Schwarz, and that $u_2$ is bounded in $H^1_h(\partial M)$, as well as that $u_1^- = u_2^-$ (as the metrics agree near the boundary), we get
        \[
            \Im \langle{\partial_\nu(u_1^+ - u_2^+), u_1^+ - u_2^+}\rangle_{L^2(\partial M)} = \mc{O}(1) + \mc{O}(\|\Lambda_{g_1}^\lambda - \Lambda_{g_2}^\lambda\|_{H^1_h \to L^2}).
        \]
        Expanding the Gaussian beams into powers of $\lambda$, we get in the same way as in \eqref{eq:new-main-working-identity},
        \begin{equation}\label{eq:quantitative-main-working-identity}
        \begin{split}
            &\underbrace{\lambda^{\frac{n-1}{2}}\int_{\partial M} \left[e^{-2\lambda \Im \Phi_1^+} |a_0^{(1), +}|^2 \Re(\partial_\nu \Phi_1^+) (\chi_1^+)^2 + e^{-2\lambda \Im \Phi_2^+} |a_0^{(2), +}|^2 \Re(\partial_\nu \Phi_2^+) (\chi_2^+)^2\right]\,d\vol}_{J := }\\
            &- \underbrace{\lambda^{\frac{n-1}{2}}\Re \int_{\partial M} e^{\I \lambda (\Phi_1^+ - \overline{\Phi_2^+})} a_0^{(1), +} \overline{a_0^{(2), +}} \partial_\nu(\Phi_1^+ + \overline{\Phi_2^+}) \chi_1^+ \chi_2^+ \,d\vol}_{J_{\mathrm{mix}}:=}\\
            &\hspace{5cm}= \mc{O}(\lambda^{-1}) + \mc{O}(h\|\Lambda_{g_1}^\lambda - \Lambda_{g_2}^\lambda\|_{H^1_h \to L^2}).
        \end{split}
        \end{equation}
    \medskip

    \emph{Step 2: exit points.} Similarly to the two-solution uniqueness proof of Theorem \ref{thm_main_1}, and using the same uniformity arguments as in the three-solution proof of Theorem \ref{theorem:quantification}, by the stationary phase lemma there is a uniform $C > 0$ such that
    \[
        J \geq C > 0.
    \]
    Again, alternatively, one can also argue without using the stationary phase lemma here, see the discussion after \eqref{eq:lower-bound-XXX}. Similarly, using direct estimation there is $C_{\mathrm{mix}} > 0$ such that
    \[
        |J_{\mathrm{mix}}| \leq \mc{O}(1) e^{-C_{\mathrm{mix}} \lambda d_{\partial M}(x_1, x_2)^2}.
    \]
    Combining this with \eqref{eq:quantitative-main-working-identity}, we get
    \[
        0 < C < \mc{O}(1) e^{-C_{\mathrm{mix}} \lambda d_{\partial M}(x_1, x_2)^2} + \mc{O}(\lambda^{-1}) + \mc{O}(h\|\Lambda_{g_1}^\lambda - \Lambda_{g_2}^\lambda\|_{H^1_h \to L^2}).
    \]
    In the same way as in \eqref{eq:lower-bound-working}, taking $\lambda > \lambda_0$ uniformly large and $\varepsilon_0 > 0$ uniformly small, we may absorb the last two terms of the right hand side to the left, and consequently we get that
    \[
        d_{\partial M}(x_1, x_2) = \mc{O}(\lambda^{-\frac{1}{2}}).
    \]
    \medskip

    \emph{Step 3: exit directions.} This step is entirely analogous to Step 3 of the first proof of Theorem \ref{theorem:quantification}. Namely, we assume that
    \[
            \Lambda := \sup_{(x, v) \in \mc{G}_{\varepsilon}} \|\mc{P}^{\partial M}_{x_1 \to x_2} \Pi S_{g_1}(x, v) - \Pi S_{g_2}(x, v)\|_{g_0} > C_0 \lambda^{-\frac{1}{2}},
    \]
    for some $C_0 > 0$, and we split $J_{\mathrm{mix}} = J_{\mathrm{mix}, 2} + J_{\mathrm{mix}, 1}$ into pieces at distance $\lesssim \lambda^{-\frac{1}{2}}$ and its complement, respectively. Then $J_{\mathrm{mix}, 2}$ is treated by integration by parts, and $J_{\mathrm{mix}, 1}$ is directly estimated. The argument is the same as in the three-solutions proof of Theorem \ref{theorem:quantification}, because the phase is the same. A contradiction is then reached to \eqref{eq:quantitative-main-working-identity}. This shows that $\Lambda = \mc{O}(\lambda^{-\frac{1}{2}})$, and so that $d(S_{g_1}, S_{g_2}) = \mc{O}(\lambda^{-\frac{1}{2}})$.
    \medskip

    \emph{Step 4: travel times.} We argue similarly to Steps 4 of the two-solution proof of Theorem \ref{thm_main_1} and the three-solution proof of Theorem \ref{theorem:quantification}. As mentioned before, the integrals in $J$ can be expanded using the usual stationary phase lemma, but $J_{\mathrm{mix}}$ requires an application of the stationary phase in the spirit of Melin-Sj\"ostrand, Lemma \ref{lemma:stationary-phase}. Note that this lemma applies verbatim to treat (the slightly different) integral $J_{\mathrm{mix}}$ because the phases are the same. Note that the subsequent Lemma \ref{lemma:smoothness-parameter} applies verbatim and shows that we may pick the value of the phase at the critical point to smoothly depend on the parameter $(x, v) \in \mc{G}_\varepsilon$, denoted by $\Psi(x, v)$. In summary, we obtain from \eqref{eq:quantitative-main-working-identity} that
    \begin{equation}\label{eq:final-bound}
        \begin{split}
            &\mc{O}(\lambda^{-1}) + \mc{O}(h\|\Lambda_{g_1}^\lambda - \Lambda_{g_2}^\lambda\|_{H^1_h \to L^2})\\ 
            &= \frac{(2\pi)^{\frac{n-1}{2}}}{\sqrt{\det(\nabla^2_{\partial M} 2\Im \Phi_1^+(x_1))}} \partial_\nu \Phi_1^+(x_1) |a_0^{(1), +}|^2 (x_1)\\ 
            &+ \frac{(2\pi)^{\frac{n-1}{2}}}{\sqrt{\det(\nabla^2_{\partial M} 2\Im \Phi_2^+(x_2))}} \partial_\nu \Phi_2^+(x_2) |a_0^{(2), +}|^2 (x_2)\\ 
            &- \Re \frac{(2\pi)^{\frac{n-1}{2}}}{\sqrt{\det(-\I\nabla^2_{\partial M}(\Phi_1^+ - \overline{\Phi}_2^+)(x_1))}}e^{\I \lambda \Psi(x_0, v_0)} (\partial_\nu\Phi_1^+(x_1) + \partial_\nu \Phi_2^+(x_2)) a_0^{(1), +}(x_1) \overline{a_0^{(2), +}}(x_2)\\
            &+ \mc{O}(\lambda^{-\frac{1}{2}}).
        \end{split}
    \end{equation}
    Note that the last $\mc{O}(\lambda^{-\frac{1}{2}})$ gathers errors from the stationary phase estimate in Lemma \ref{lemma:stationary-phase}, and also errors obtained by Taylor's expansion of $\partial_\nu \overline{\Phi}_2^+(x_2) = \partial_\nu \overline{\Phi}_2^+(x_1) + \mc{O}(\lambda^{-\frac{1}{2}})$, where uniform estimates on the amplitude are used, as well as Step 2; similar holds for $a_0^{(2), +}(x_2)$. The same analysis as in Step 4 of the three-solution proof of Theorem \ref{theorem:quantification} shows that there is a function $H$ depending smoothly on $(x, v) \in \mc{G}_\varepsilon$, bounded uniformly from both above and below by positive constants, such that
    \begin{equation}
        1 - \Re(e^{\I \lambda \Psi(x, v)}H(x, v)) = \mc{O}(\lambda^{-\frac{1}{2}}) + \mc{O}(h\|\Lambda_{g_1}^\lambda - \Lambda_{g_2}^\lambda\|_{H^1_h \to L^2}),\quad (x, v) \in \mc{G}_\varepsilon.
    \end{equation}
    Note that thanks to Lemma \ref{lemma:smoothness-parameter} once again, $\Psi \equiv 0$ and $H \equiv 1$ near the boundary of $\mc{G}_\varepsilon$, this estimate extends to
    \begin{equation}
        1 - \Re(e^{\I \lambda \Psi(x, v)}H(x, v)) = \mc{O}(\lambda^{-\frac{1}{2}}) + \mc{O}(h\|\Lambda_{g_1}^\lambda - \Lambda_{g_2}^\lambda\|_{H^1_h \to L^2}),\quad (x, v) \in \partial_-SM.
    \end{equation}
    The argument then proceeds in the same way as before. Indeed, if $z \in \partial_-SM$ is connected by a path $\Gamma$ in the same spherical fibre as $z$ to a point $z_0 \not \in \mc{G}_\varepsilon$ (so $H(z_0) = 1$), and writing $H \circ \Gamma = |H \circ \Gamma| e^{\I f_\Gamma}$ (with $f_\Gamma(0) = 0$) we get along $\Gamma$ that
    \begin{equation}\label{eq:XXX}
        1 - |H| e^{-\lambda \Im \Psi}\cos(\lambda \Re \Psi + f_\Gamma) = \mc{O}(\lambda^{-\frac{1}{2}}) + \mc{O}(h\|\Lambda_{g_1}^\lambda - \Lambda_{g_2}^\lambda\|_{H^1_h \to L^2}).
    \end{equation}
    Note that this implies 
    \[
        |H| e^{-\lambda \Im \Psi} \geq 1 - \mc{O}(\lambda^{-\frac{1}{2}}) - \mc{O}(h\|\Lambda_{g_1}^\lambda - \Lambda_{g_2}^\lambda\|_{H^1_h \to L^2})
    \]
    which in turn shows that (for $\lambda$ uniformly large enough and for $\varepsilon_0$ uniformly small enough)
    \[
        0 < |H|^{-1} e^{\lambda \Im \Psi} \leq 1 + \mc{O}(\lambda^{-\frac{1}{2}}) + \mc{O}(h\|\Lambda_{g_1}^\lambda - \Lambda_{g_2}^\lambda\|_{H^1_h \to L^2}).
    \]
    (In fact $|H|^{-1} e^{\lambda \Im \Psi}$ is also bounded uniformly from below by a positive constant by Lemma \ref{lemma:smoothness-parameter}, but we will not need this.) After multiplying \eqref{eq:XXX} with $|H|^{-1} e^{\lambda \Im \Psi}$ (note that it is uniformly bounded), it follows that
    \[
        \min(1, |H|^{-1}e^{\lambda \Im \Psi}) - \cos(\lambda \Re \Psi + f_\Gamma) = \mc{O}(\lambda^{-\frac{1}{2}}) + \mc{O}(h\|\Lambda_{g_1}^\lambda - \Lambda_{g_2}^\lambda\|_{H^1_h \to L^2}),
    \]
    We then obtain in the verbatim same way as in the two-solution proof of Theorem \ref{thm_main_1} that
    \[
        -\I \int_\Gamma \frac{d\left(\frac{H}{|H|}\right)}{\frac{H}{|H|}} + \lambda \Re \Psi (z) \pm \arccos(\min(1, |H|^{-1}e^{\lambda \Im \Psi})(z)) =\mc{O}(\lambda^{-\frac{1}{4}}) + \mc{O}(h^{\frac{1}{2}}\|\Lambda_{g_1}^\lambda - \Lambda_{g_2}^\lambda\|_{H^1_h \to L^2}^{\frac{1}{2}}).
    \]
    Again, for simplicity, by convention here we mean that this holds for at least one choice of the $\pm$ sign. This together with the expansion of $\Re \Psi$ from Lemma \ref{lemma:smoothness-parameter} gives that $\ell_{g_1}(z) - \ell_{g_2}(z) = \mc{O}(\lambda^{-\frac{1}{2}})$ with uniform implied estimates. As $z \in \partial_-SM$ was arbitrary, this completes the proof of the theorem.
    \medskip
    
    \emph{Alternative proof of Step 4.} Let us argue based on Proposition \ref{prop:identity}. Let $z = (x, v) \in \partial_-SM$, and let $\gamma$ be a smooth path of uniformly bounded length such that $z_0 = \gamma(0) \not \in \mc{G}_\varepsilon$ (so that $\ell_{g_1}(z_0) = \ell_{g_2}(z_0)$) and $\gamma(1) = z$ (as before, we may take such a path in the same fibre as $z$). Integrating the identity of Proposition \ref{prop:identity} along $\gamma$, for $i = 1, 2$, and subtracting the obtained expressions gives
    \[
        \ell_{g_1}(z) - \ell_{g_2}(z) = \int_{S_{g_1}(\gamma)} \alpha - \int_{S_{g_2}(\gamma)} \alpha.
    \]
    For each $t \in [0, 1]$, by Steps 2 and 3, we have that $d(S_{g_1}(\gamma(t)), S_{g_2}(\gamma(t))) = \mc{O}(\lambda^{-\frac{1}{2}})$. Here we use that by compactness, this distance induced by the Riemannian metric is uniformly equivalent near the diagonal to the distance introduced in \S \ref{ssec:distance-lens-data}. The manifold $SM$ and so $SM|_{\partial M}$ as a submanifold is equipped with the natural Sasaki metric induced by $g_0$ (see \cite{Paternain-99}); alternatively, any fixed Riemannian metric would work for the following. We denote by $\exp$ the associated exponential map. By taking $\lambda$ uniformly large enough, for every $t \in [0, 1]$ there is a unique constant speed shortest geodesic $[0, 1] \ni s \mapsto \Gamma(t, s)$ connecting $S_{g_1}(\gamma(t))$ and $S_{g_2}(\gamma(t))$. We may write 
    \[
        \Gamma(t, s) = \exp_{S_{g_1}(\gamma(t))}\left(s\exp_{S_{g_1}(\gamma(t))}^{-1}(S_{g_2}(\gamma(t)))\right).
    \]
    This shows that $\Gamma$ is smooth, and using that $dS_{g_1}$ and $dS_{g_2}$ are uniformly bounded for the metrics inside $\mc{M}_{N, B, \delta, T}(g_0)$ (by uniform bounds on the geodesic flow and transversality of geodesics generated by $\mc{G}_\varepsilon$), that $|\partial_t \Gamma| = \mc{O}(1)$. By definition, we also have
    \begin{equation}\label{eq:transverse-derivative}
        |\partial_s \Gamma| = |\exp_{S_{g_1}(\gamma(t))}^{-1}(S_{g_2}(\gamma(t))| = d(S_{g_1}(\gamma(t)), S_{g_2}(\gamma(t))) = \mc{O}(\lambda^{-\frac{1}{2}}).
    \end{equation}
    By Stokes' theorem, we conclude that
    \begin{equation}\label{eq:stokes-lens}
        \ell_{g_1}(z) - \ell_{g_2}(z) = \int_{\Gamma([0, 1]^2)} d\alpha - \int_{\Gamma(\{1\} \times [0, 1])} \alpha + \int_{\Gamma(\{0\} \times [0, 1])} \alpha.
    \end{equation}
    Note that, by the preceding disucssion,
    \[
        \left|\int_{\Gamma([0, 1]^2)} d\alpha \right| = \left|\int_{[0, 1]^2} d\alpha(\Gamma(s, t))(\partial_t\Gamma, \partial_s \Gamma)\,ds dt\right| = \mc{O}(\lambda^{-\frac{1}{2}}).
    \]
    The last two terms in \eqref{eq:stokes-lens} are equal to $\mc{O}(\lambda^{-\frac{1}{2}})$, using \eqref{eq:transverse-derivative}. This shows that $\ell_{g_1}(z) - \ell_{g_2}(z) = \mc{O}(\lambda^{-\frac{1}{2}})$, and since $z$ was arbitrary this concludes the proof.    
    \end{proof}
    \begin{remark}
        It is natural to try to use interpolation in the alternative proof of Step 4 above. Indeed, using that $\nabla^2 S_{g_i}$ are uniformly bounded for $i = 1, 2$, interpolating we get $d(dS_{g_1}, dS_{g_2}) = \mc{O}(\lambda^{-\frac{1}{4}})$, yielding $|\ell_{g_1} - \ell_{g_2}| = \mc{O}(\lambda^{-\frac{1}{4}})$ by direct estimation from Proposition \ref{prop:identity}. Note that this is weaker than what is obtained from an application of Stokes' theorem as demonstrated above.
    \end{remark}
\bibliography{ref}
\bibliographystyle{siam}
    \end{document}

%% file: drawing.pdf_tex
\begingroup%
  \makeatletter%
  \providecommand\color[2][]{%
    \errmessage{(Inkscape) Color is used for the text in Inkscape, but the package 'color.sty' is not loaded}%
    \renewcommand\color[2][]{}%
  }%
  \providecommand\transparent[1]{%
    \errmessage{(Inkscape) Transparency is used (non-zero) for the text in Inkscape, but the package 'transparent.sty' is not loaded}%
    \renewcommand\transparent[1]{}%
  }%
  \providecommand\rotatebox[2]{#2}%
  \newcommand*\fsize{\dimexpr\f@size pt\relax}%
  \newcommand*\lineheight[1]{\fontsize{\fsize}{#1\fsize}\selectfont}%
  \ifx\svgwidth\undefined%
    \setlength{\unitlength}{568.5170908bp}%
    \ifx\svgscale\undefined%
      \relax%
    \else%
      \setlength{\unitlength}{\unitlength * \real{\svgscale}}%
    \fi%
  \else%
    \setlength{\unitlength}{\svgwidth}%
  \fi%
  \global\let\svgwidth\undefined%
  \global\let\svgscale\undefined%
  \makeatother%
  \begin{picture}(1,0.76107747)%
    \lineheight{1}%
    \setlength\tabcolsep{0pt}%
    \put(0,0){\includegraphics[width=\unitlength,page=1]{drawing.pdf}}%
  \end{picture}%
\endgroup%